\documentclass[10pt,a4paper,reqno]{amsart}

\usepackage[utf8]{inputenc}
\usepackage[T1]{fontenc}
\usepackage{amsthm,amsmath,amsfonts,amssymb,mathtools}
\usepackage[numbers]{natbib}
\usepackage{hyperref}
\hypersetup{hidelinks,
 pdftitle={Propagation of Chaos for Belavkin Equations Beyond Pure States},
 pdfauthor={Gaoyue Guo}}
\usepackage{enumitem}
\usepackage{microtype}
\usepackage{array,longtable,needspace}
\usepackage{mathrsfs}
\usepackage[a4paper,margin=1in]{geometry}
\usepackage{xcolor}

\allowdisplaybreaks
\numberwithin{equation}{section}

\theoremstyle{plain}
\newtheorem{theorem}{Theorem}[section]
\newtheorem{proposition}[theorem]{Proposition}
\newtheorem{lemma}[theorem]{Lemma}

\newtheorem{hypothesis}[theorem]{Hypothesis}

\theoremstyle{definition}
\newtheorem{definition}[theorem]{Definition}

\theoremstyle{remark}
\newtheorem{remark}[theorem]{Remark}

\renewcommand{\d}{\mathrm d}
\renewcommand{\i}{\mathrm i}
\newcommand{\Tr}{\mathrm{Tr}}
\newcommand{\Law}{\mathrm{Law}}
\newcommand{\Ent}{\mathrm{Ent}}

\newcommand{\E}{\mathbb E}

\newcommand{\Q}{\mathbb Q}
\newcommand{\cS}{\mathcal S}
\newcommand{\cL}{\mathcal L}
\newcommand{\cF}{\mathcal F}
\newcommand{\cD}{\mathcal D}
\newcommand{\cC}{\mathcal C}
\newcommand{\cE}{\mathcal E}
\newcommand{\cY}{\mathcal Y}
\newcommand{\HH}{\mathcal H}
\newcommand{\KK}{\mathcal K}
\newcommand{\1}{\mathbf 1}
\newcommand{\grel}{\mathfrak d}
\newcommand{\recordentropy}{\mathscr H}
\date{}

\title[Belavkin equations beyond pure states]
{Propagation of Chaos for Belavkin Equations Beyond Pure States}

\author[G. Guo]{Gaoyue Guo}
\address{Universit\'e Paris-Saclay, CentraleSup\'elec, MICS and CNRS FR-3487}
\email{gaoyue.guo@centralesupelec.fr}


\begin{document}

\begin{abstract}
We prove quantitative propagation of chaos for finite-dimensional quantum
mean-field systems undergoing continuous diffusive observation. The object
approximated is the random conditional state given the records of all the
particles, rather than its average over the observations. Partial tracing does
not remove the exterior observation noises, so the deterministic hierarchy
argument for closed systems does not directly control this conditional state.
We treat arbitrary one-particle density matrices, general measurement
operators, and every efficiency in $(0,1]$. Two initial-data regimes are
considered: approximation of the full state by a tensor product, without
permutation invariance, and marginal chaoticity of permutation-invariant
states, without full-state tensorization. The first approach combines
purification, fully observed dilation, conditional factorization, and a
particle-number-uniform stability estimate for the linear Zakai equation.
The second propagates the relative entropy of the joint observation record
and conditional quantum state against a product filtering reference. A
centered Hamiltonian fluctuation estimate controls the interacting part,
while the common observation instruments contract entropy. Regularizing only
the reference removes its faithfulness requirement without perturbing the
interacting initial state. For faithful one-particle initial states, the
entropy estimate gives an $N^{-1/2}$ trace-norm bound for exact products and
its stable extension to strongly approximately tensorized data. These results
justify a one-particle description of fixed conditional marginals, including
for initially correlated symmetric systems.
\end{abstract}

\maketitle

\setcounter{tocdepth}{2}\tableofcontents

\section{Introduction}
\label{sec:introduction}

Propagation of chaos ensures that a fixed subsystem of a large interacting 
population can be described by independent copies of an effective nonlinear
dynamics.  For a continuously observed quantum system,  its relevant state is
conditional on an indirect measurement record,  and the question is therefore stronger than
the usual convergence of reduced density matrices averaged over observations:

\medskip

\noindent \emph{Does a one-particle description remain valid for the conditional state
of a fixed subsystem when the observer has access to the records of the
entire population?} 

\medskip

In this paper,  we address this question for finite-dimensional quantum
systems with mean-field Hamiltonian interactions under diffusive
measurements.  The $N$-particle conditional state is a density operator $\rho_t^N$ on
$\HH^{\otimes N}$,  and its first
$n$-particle marginal,  $\rho_t^{N:n}$,  is identified by the partial trace of $\rho_t^N$  which keeps the first $n$ tensor factors.  We prove estimates implying
\begin{equation}
 \sup_{0\le t\le T}\E\left[
   \left\|\rho_t^{N:n}-\gamma_t^1\otimes\cdots\otimes\gamma_t^n\right\|_1
 \right]\longrightarrow0, 
 \qquad\text{for every fixed }n\ge1\text{ and }T>0.
 \label{eq:intro-target}
\end{equation}
Here the $\gamma^j$ are independent copies of a nonlinear one-particle
Belavkin equation, coupled through their deterministic mean state, and driven
by the same innovation Brownian motions as the corresponding coordinates of
the interacting system. 

\medskip

\noindent \textbf{Observation changes the mean-field problem.}
For closed systems the many-body density matrix follows the linear,
deterministic von Neumann equation. Taking partial traces gives a BBGKY
hierarchy whose interaction term couples an $n$-particle marginal to an
$(n+1)$-particle marginal. Local trace-preserving dissipators retain this
structure for an unconditioned Lindblad evolution. By contrast, the Belavkin equation is nonlinear and stochastic,   and a measurement of a particle outside the subsystem can change the subsystem's conditional state
as the particles are correlated.   Consequently, partial tracing leaves
stochastic integrals against exterior observation noises.  Their coefficients
are conditional covariances; there are $N-n$ exterior particle labels, and
these terms do not carry the $1/N$ prefactor of the pair interaction.
Section~\ref{subsec:difficulty-and-open-problem} displays this obstruction
explicitly.

In particular, an unconditioned mean-field theorem cannot simply be followed
by conditioning. Small error after averaging does not control error before
averaging, and convergence of joint laws alone does not generally control
conditional distributions. Passing to the linear Zakai equation resolves the
normalization nonlinearity, but does not by itself remove the information
carried by the exterior records. The central issue is to control
\emph{marginalization and conditioning simultaneously}, rather than to treat
the observation noise as a bounded perturbation of a closed-system hierarchy.

\medskip

\noindent \textbf{Relation to previous work.}
Mean-field limits for closed quantum systems are established by several
complementary approaches. These include BBGKY hierarchies
\citep{spohn1980,bardos2000,bardos2002}, semiclassical and phase-space methods
\citep{frohlich2007,Ammari2009}, counting estimates
\citep{knowlespickl2010,pickl2011}, and quantum empirical-measure and transport
methods \citep{golsepaul2017,golsepaul2019}. Relative-entropy estimates for
mixed-state mean-field limits, including deterministic Lindblad dynamics,
are developed in  \citep{Guo2026}. Mixed states, quantum entropy, and
mean-field interactions are therefore not new objects in themselves. The
additional conclusion sought here concerns the conditional many-body state
and a random product of local filters, in the presence of all observation
records. The deterministic results just cited do not supply that conclusion.

The quantum-filtering framework and the stochastic master equations used
below belong to the established theory; see
Bouten--van Handel--James \citep{bouten2007}, Pellegrini \citep{pellegrini2008}, and Mirrahimi--van Handel \citep{mirrahimivanhandel2007}.
Kolokoltsov's propagation-of-chaos results
\citep{kolokoltsov2021,kolokoltsov2022} provide the pure-state,
perfect-observation estimate on which our first reduction is based. In that
regime, a wave-function representation and a counting functional control the
trace distance from a prescribed rank-one state. Infinite-dimensional
wave-function limits are also studied by de Bouard--Guo--H\'erouard \citep{bouard2026}.
Our purpose is different from extending the class of unbounded spatial
operators: we retain a fixed finite-dimensional one-particle space, but allow
mixed conditional states, inefficient detection, and correlated initial
many-body states under the conditions specified below. 

\medskip

\noindent  \textbf{Two initial-data regimes.}
Theorem~\ref{thm:main-approx} assumes that the  initial state $\rho^N_0$ approaches
$\gamma_0^{\otimes N}$ in trace norm. It requires no permutation invariance of
$\rho_0^N$ and no rank restriction on $\gamma_0$.  The error is bounded by the initial trace distance plus $N^{-1/4}$
for every fixed $\eta\in(0,1]$. Theorem~\ref{thm:main-chaotic} weakens full-state
approximation to marginal chaoticity,  i.e.,  $\rho^{N:n}_0\to \gamma_0^{\otimes n}$,  assuming permutation invariance of
$\rho_0^N$.  This permits global correlations that need not be
small in full trace norm. Both theorems allow arbitrary measurement
operators and every efficiency.  Faithfulness enters only in stronger quantitative conclusions and in an intermediate proof step.

\medskip

\noindent \textbf{The first mechanism: reduction and uniform stability.}
For exact products $\rho^N_0=\gamma_0^{\otimes N}$ at perfect efficiency, we purify $\gamma_0$, lift the
coefficients, and prove that partial tracing the enlarged dynamics recovers
both the interacting and the limiting equations. Purification itself is a
standard quantum-information construction
\citep{nielsenchuang2010,watrous2018},  and its required role here is compatibility
with the stochastic mean-field comparison. Inefficient observation is then
represented by a fully observed dilation. Although projecting the records
recovers the original equation, it does not preserve a tensor product at
finite $N$. A compensating Girsanov transform, classical relative entropy,
and its conditional chain rule establish approximate conditional
factorization. Finally, the  linear Zakai propagator preserves trace
in expectation. This yields a trace-norm contraction with no factor growing
in $N$, and hence a uniform transfer estimate for initial-state
perturbations. This route gives the general estimate in
Theorem~\ref{thm:main-approx} independently of the second mechanism.

\medskip

\noindent \textbf{The second mechanism: entropy before discarding the record.}
For permutation-invariant initial data, we compare the joint classical
observation record and quantum conditional state with the product of
one-particle filtering references. Its relative entropy decomposes into the
classical entropy of the record law and the expected quantum relative
entropy of the conditional states. Quantum instruments and such entropy
comparisons are established tools \citep{BarchielliLupieri2004,BarchielliLupieri2006}.
The new estimate is a particle-number-uniform stability bound for this joint
entropy under an interacting Hamiltonian, with a random tensor-product
reference. Common local observation instruments are contractive,  and the
Hamiltonian defect is controlled through centered fluctuations and a uniform
exponential estimate. Block superadditivity then bounds a fixed conditional
marginal, and a local stability argument recovers the synchronous innovation
coupling in \eqref{eq:intro-target}.

The random reference creates an additional boundary issue: its eigenvalues
can become small along an observation path. Starting from a faithful
one-particle state, logarithmic-eigenvalue estimates close the entropy bound
and give the $N^{-1/2}$ rate for exact products. For an arbitrary-rank state,
we regularize only the reference and update it at finitely many deterministic
times, keeping the interacting system unchanged. This avoids an invalid
uniform-in-$N$ approximation of tensor powers in trace norm. The resulting
time-dependent exponent is weaker than the faithful estimate. It also does
not replace the original arbitrary-rank exponents in
Theorem~\ref{thm:main-approx},  and the two proof mechanisms retain distinct
quantitative strengths.

\medskip

\noindent \textbf{Consequences and scope.}
These are the probabilistic conclusions that
motivate the model reduction: a $d^N\times d^N$ conditional density matrix is
replaced, for fixed marginals, by $d\times d$ nonlinear filters and the
deterministic mean-state equation. Quantum filtering and quantum games
provide a natural setting for subsequent applications
\citep{kolokoltsov2022,mirrahimivanhandel2007}. 

\medskip

\noindent \textbf{Organization.}
Section~\ref{sec:framework} gives the framework and the two main theorems,
each with a general statement (i) and a faithful refinement (ii).
Section~\ref{sec:exact} combines the original exact-product arguments:
purification at perfect efficiency, followed by fully observed dilation
and conditional factorization at arbitrary efficiency.
Section~\ref{sec:entropy} proves the faithful relative-entropy estimate
and its local version with explicit dependence on the reference eigenvalues.
Section~\ref{sec:extensions} contains two distinct extensions: uniform
stability under changes of the interacting initial state, and removal of
faithfulness by updates of the entropy reference alone.  It completes the remaining proofs.  The appendices collect standard entropy tools, technical
stochastic estimates, and the initial-state analysis under permutation
invariance.

\medskip

\noindent \textbf{Acknowledgments.}
This work was supported by ANR-25-CE40-0714 (MATH-SPA). The author thanks
Hao Liang and Zhenfu Wang for helpful discussions.

\section{Mathematical framework and main results}
\label{sec:framework}

Throughout the paper, \(\HH\) is reserved for a finite-dimensional complex Hilbert
space,  and we write
$d:=\dim(\HH)\ge 2$.  It represents the state space of a single quantum system,
for instance a spin, a qubit, a qudit, or a finite-level atom.  We keep this
notation abstract and do not choose coordinates.  A quantum state
may be represented by a \emph{wave function}  taking values in $\HH$,  or more generally by a  \emph{density operator} on \(\HH\).  This is essential for continuously monitored open quantum systems. Indeed, an indirect measurement is
usually modeled by coupling the system to an auxiliary probe or environment and then measuring the
outgoing probe. Conditioning on the random measurement record produces a stochastic evolution of the
system state, often called a \emph{quantum trajectory},  see e.g. Pellegrini \cite{pellegrini2008} and Bouten--van Handel--James
\cite{bouten2007}.

For \(N\) identical particles,  the one-particle space \(\HH\) is replaced by their
tensor product
\(
 \HH^{\otimes N}.
\)
One-particle observables are lifted to individual tensor factors, two-particle
interactions are lifted to pairs of tensor factors, and reduced states of
subsystems are obtained by partial traces.  These elementary operations provide
the notation used to formulate the \(N\)-particle Belavkin equation and its
mean-field limit introduced later.

\subsection{Quantum states and observables}
\label{subsec:states-and-composite-systems}

Denote by $\cL(\HH)$ the space of linear operators on $\HH$. For
$X\in\cL(\HH)$, let
\[
 \|X\|:=\sup_{\|\psi\|_{\HH}\le1}\|X\psi\|_{\HH},
 \qquad
 \|X\|_1:=\Tr\bigl(\sqrt{X^*X}\bigr),
 \qquad
 \|X\|_2:=\sqrt{\Tr(X^*X)}
\]
be the \emph{operator},  \emph{trace} and \emph{Hilbert--Schmidt norms}. Their elementary
comparison is
\[
 \|X\|\le\|X\|_2\le\|X\|_1\le d\|X\|.
\]
The  \emph{Hilbert--Schmidt inner product} is
$\langle X,Y\rangle_{\mathrm{HS}}:=\Tr(X^*Y)$, and
$[X,Y]:=XY-YX$, $\{X,Y\}:=XY+YX$ denote the  \emph{commutator} and \emph{anticommutator}.
We write
\[
  \cL^*(\HH):=\{X\in\cL(\HH):X^*=X\},\quad
  \cL^+(\HH):=\{X\in\cL^*(\HH):X\ge0\},\quad  \cS(\HH):=\{\rho\in\cL^+(\HH):\Tr(\rho)=1\}.
\]
Elements of $\cS(\HH)$ are called density operators,  or \emph{density matrices} in our  finite-dimensional setting,  and represent \emph{quantum states},  while elements of $\cL^*(\HH)$ stand for  \emph{observables}.   If
$A\in\cL^*(\HH)$ and $\rho\in\cS(\HH)$, then
\(
  \Tr(\rho A)
\)
stands for the expectation of the observable $A$ in the state $\rho$. 

\medskip

A state $\rho\in\cS(\HH)$ is called \emph{pure} if it has rank one,  i.e.,  $\operatorname{rank}(\rho)=1$.  
Equivalently, there exists some unit vector $\psi\in\HH$,  identified as a wave function,  such that
\(
  \rho=|\psi\rangle\langle\psi|,
\)
where $ |\psi\rangle\langle\psi| \in \cL(\HH)$ denotes the \emph{projection operator} given as
\[
  |\psi\rangle\langle\psi|\,x
  :=
  \langle\psi,x\rangle_{\HH}\,\psi,
  \qquad
  x\in\HH.
\]
Thus a wave function determines a pure density operator, and two wave functions
which differ only by a global phase determine the same physical state.  A state $\rho$ 
which is not pure is called \emph{mixed},  and it   is called \emph{faithful} if
$\rho>0$. 
Every state admits a  spectral representation
\[
 \rho=\sum_{r=1}^{\operatorname{rank}(\rho)}
       \lambda_r|\psi_r\rangle\langle\psi_r|,
 \qquad \lambda_r>0,\quad\sum_r\lambda_r=1,
\]
where $(\psi_r)_r$ is an orthonormal family in $\HH$ and the numbers
\(\lambda_r\) are scalar weights.  A mixed state may describe classical statistical uncertainty,
but it may also arise as the reduced state of a larger entangled quantum
system.  This distinction is one reason why density operators, rather than
only wave functions, are the natural state variables for open quantum
systems.

\medskip

For $N\ge1$, let
\(
  [N]:=\{1,\ldots,N\}
\) 
and
  \(
  \HH^N:=\HH^{\otimes N}
\)
be the $N$-tensor product of $\HH$.  The space \(\HH^N\) is the Hilbert space of \(N\) labeled copies of the
one-particle system.  Product density operators describe uncorrelated
configurations, while a general element of $\cS(\HH^N)$ may contain both
classical correlations and quantum entanglement.  

If \(J=\{j_1<\cdots<j_r\}\subset[N]\), we set
\[
  \HH^J
  :=
  \HH_{j_1}\otimes\cdots\otimes\HH_{j_r},
  \qquad
  \HH^{J^c}
  :=
  \bigotimes_{j\in J^c} \HH_j,
\]
where \(\HH_j\) denotes the \(j\)-th copy of \(\HH\). We use
the convention that the tensor product over the empty set is \(\mathbb C\).
Thus \(\HH^J\) is canonically identified with \(\HH^{|J|}\).  Let 
\[
  \Tr_{J^c}:
  \cL(\HH^N)
  \longrightarrow
  \cL(\HH^{J})
\]
denote the \emph{partial trace}  which keeps the tensor factors indexed by $J$ and
traces out the complementary factors.  It is characterized by the duality
relation
\[
  \Tr\left(
    A\,\Tr_{J^c}(X)
  \right)
  =
  \Tr\left(
    (A\otimes I_{\HH^{J^c}})X
  \right)
\]
for every local observable $A$ on the factors indexed by $J$, after the
canonical reshuffling of tensor factors.  When no confusion is possible, we identify \(\HH^J\) with
\(\HH^{|J|}\).  In particular, the partial trace
extracts the state seen by a subsystem.

For $A\in\cL(\HH)$ and $j\in[N]$, let $A_j\in\cL(\HH^N)$ denote the lift
of $A$ to the $j$-th tensor factor:
\[
  A_j
  :=
  I_{\HH}^{\otimes(j-1)}
  \otimes A\otimes
  I_{\HH}^{\otimes(N-j)}.
\]
Similarly, if $O \in\cL(\HH\otimes\HH)$ and $1\le j<k\le N$, then
$O_{jk}$ denotes the lift of $O$ to the $j$-th and $k$-th tensor factors.

\subsection{Quantum system under continuous indirect measurements}
\label{subsec:continuously-monitored-systems}

We fix a self-adjoint one-particle Hamiltonian
\[
  h=h^*\in\cL^*(\HH),
\]
$m\ge 1$ measurement,  or coupling,  operators
\[
  L^{(1)},\ldots,L^{(m)}\in\cL(\HH),
\]
and a self-adjoint and symmetric two-body  interaction
\[
  V=V^*\in\cL^*(\HH\otimes\HH).
\]
The operator
\(h\) generates the coherent evolution of each isolated particle.  The operator
\(V\) describes the interaction between two particles.  The operator
\(L^{(k)}\) describes the coupling of a particle to the \(k\)-th output
channel,  and the associated observed quadrature is
\(L^{(k)}+(L^{(k)})^*\).  Thus the operator \(L^{(k)}\) leads to  both
the dissipative back-action of the environment and the stochastic update
generated by the measurement record.

For $\sigma\in\cS(\HH)$, define the effective one-particle interaction $ V^\sigma\in  \cL^*(\HH)$ by
\begin{equation}
  V^\sigma
  :=
  \Tr_{\{2\}}
  \bigl(
    V(I_{\HH}\otimes\sigma)
  \bigr),
  \label{eq:effective-potential}
\end{equation}
where we recall that $\Tr_{\{2\}}$ denotes the partial trace over the second tensor factor.  \(V^\sigma\) stands for  the mean field created by infinitely many particles of the same  state 
\(\sigma\).  
The cyclicity of the partial trace on the traced factor gives
\begin{equation}
  \Tr_{\{2\}}
  \bigl(
    [V,\gamma\otimes\sigma]
  \bigr)
  =
  [V^\sigma,\gamma],
  \qquad \mbox{ for all }
  \gamma,\sigma\in\cS(\HH).
  \label{eq:effective-commutator}
\end{equation}
For every $L\in\cL(\HH)$, define the maps $\cD_L$,  $\cC_L$ and $\cE_L$ on $\cL(\HH)$ by
\begin{align}
  \cD_L(X)
  &:=
  LXL^*-\frac12\{L^*L,X\},
  \label{eq:D-def}
  \\
  \cC_L(X)
  &:=
  LX+XL^*,
  \label{eq:C-def}
  \\
  \cE_L(X)
  &:=
  \cC_L(X)
  -
  \Tr\bigl((L+L^*)X\bigr)X.
  \label{eq:E-def}
\end{align}
The linear map $\cD_L$ is the Lindblad dissipator generated by the coupling \(L\).
It describes the irreversible back-action of the environment on the state.  The
map $\cC_L$ is the linear coefficient appearing in the  Zakai 
equation, while $\cE_L$ is its normalized correction
which updates the conditional state after a measurement,  and
keeps the trace equal to one.  We use the same notation with a subscript \(j\)
when the corresponding operator acts on the \(j\)-th tensor factor.

The mean-field $N$-particle Hamiltonian is given by
\begin{equation}
  H^N
  :=
  \sum_{j=1}^N h_j
  +
  \frac1N
  \sum_{1\le j<\ell\le N}
  V_{j\ell}.
  \label{eq:HN}
\end{equation}
The factor $N^{-1}$ is the mean-field scaling: every particle interacts weakly
with every other particle, but the total interaction experienced by one
particle remains of order one as $N\longrightarrow\infty$.

\medskip

We now describe the observation model.  The system is coupled to an
auxiliary probe or output field, and the outgoing probe is measured.  The
conditional state, given the observation record, then follows a stochastic
trajectory.  In the diffusive case such trajectories are described by Belavkin
equations, or equivalently by stochastic master equations,  see, e.g.  Bouten--van Handel--James \cite{bouten2007} and Pellegrini
\cite{pellegrini2008}.  We assume that each particle is monitored
separately through the same \(m\) output channels.

Let \(\eta\in(0,1]\) be the common\footnote{For notational simplicity we take a common efficiency $\eta$.
The same arguments extend to channel-dependent efficiencies
$\eta_1,\ldots,\eta_m\in(0,1)$, with constants depending on these parameters.} measurement efficiency.   Perfect detection
corresponds to \(\eta=1\), whereas \(0<\eta<1\) means that only part of the
information carried by the output fields is observed. 

For an initial state $\rho_0^N\in\cS(\HH^N)$, the state of the \(N\)-particle system evolves according to 
the Belavkin equation
\begin{equation}
\begin{aligned}
  \d\rho_t^N
  &=
  -\i[H^N,\rho_t^N]\,\d t
  +
  \sum_{k=1}^m\sum_{j=1}^N
  \cD_{L_j^{(k)}}(\rho_t^N)\,\d t
 +
  \sqrt{\eta}
  \sum_{k=1}^m\sum_{j=1}^N
  \cE_{L_j^{(k)}}(\rho_t^N)\,\d W_t^{k,j},
\end{aligned}
  \label{eq:N-Belavkin}
\end{equation}
where $W^{k,j}$ are independent real-valued Brownian motions.  The Hamiltonian commutator gives the coherent many-body evolution.  The
Lindblad terms describe the full back-action of the couplings to the output
fields, whether or not those fields are detected.  The stochastic terms encode
the information gained from the observed part of the outputs,  corresponding to the fact that 
$\eta$ appears in the innovation coefficients.  

\medskip

There is an important pure-state reduction.  If \(\eta=1\) and the initial
state $\rho^N_0$ is pure, then the solution of \eqref{eq:N-Belavkin} remains pure,  saying  \(\rho_t^N=|\Psi_t^N\rangle\langle\Psi_t^N|\),  where the wave function
\(\Psi_t^N\) satisfies the \emph{nonlinear stochastic Schr\"odinger equation}.
This formulation with respect to the  \(N\)-particle wave function  has been extensively studied, both in finite
dimension and in infinite-dimensional Hilbert spaces,  see e.g. Mora--Rebolledo \cite{mora2008}.  Under the pure-state assumption, recent work of de Bouard--Guo--H\'erouard
\cite{bouard2026} also derives infinite-dimensional mean-field limits in this
wave-function setting.  For  the general Belavkin equation \eqref{eq:N-Belavkin},  finite-dimensional well-posedness with respect to the strong solution  was proved by Pellegrini \cite{pellegrini2008}  for
\(\dim(\HH)=2\) and by Mirrahimi--van Handel
\cite{mirrahimivanhandel2007}  for general
finite-dimensional \(\HH\).  More recently,  Kolokoltsov obtained
infinite-dimensional well-posedness results in
\citep{kolokoltsov2025a,kolokoltsov2025b}.

For general $N$ interacting particles, however, this conditional state belongs to
$\cS(\HH^N)$, whose dimension grows exponentially in $N$.  A
propagation-of-chaos theorem is therefore a model-reduction result which justifies
replacing the many-particle conditional dynamics, at the level of any fixed
number of particles, by independent copies of an effective one-particle filter.  As $N\longrightarrow \infty$, we recover formally the limiting dynamics,  i.e.,  matrix-valued McKean--Vlasov equation
\begin{equation}
\begin{aligned}
  \d\gamma_t^j
  =
  \left(
    -\i[h+V^{\xi_t},\gamma_t^j]
    +
    \sum_{k=1}^m
    \cD_{L^{(k)}}(\gamma_t^j)
  \right)\d t+
  \sqrt{\eta}
  \sum_{k=1}^m
  \cE_{L^{(k)}}(\gamma_t^j)\,\d W_t^{k,j},
  \qquad
  \xi_t:=\E[\gamma_t^j],
\end{aligned}
  \label{eq:MF-Belavkin}
\end{equation}
where $ \gamma_0^j=
  \gamma_0\in\cS(\HH)$.  
The processes $\gamma^1,\gamma^2,\ldots$ are independent and identically distributed.  Their interaction
occurs only through the deterministic mean state $\xi_t$.  Taking expectations
in \eqref{eq:MF-Belavkin} gives the Hartree equation 
\begin{equation}
  \frac{\d}{\d t}\xi_t
  =
  -\i[h+V^{\xi_t},\xi_t]
  +
  \sum_{k=1}^m
  \cD_{L^{(k)}}(\xi_t),
  \qquad
  \xi_0=\gamma_0.
  \label{eq:mean-state-equation}
\end{equation}
Thus the random limiting state is driven by the local measurement record,  whereas 
its mean-field Hamiltonian is determined by the average state of the ensemble.  Well-posedness of   \eqref{eq:MF-Belavkin} is derived by 
Amini--Amini--Chalal--Guo
 \cite{Amini2025} for finite-dimensional $\HH$,  while it remains  open  for general infinite-dimensional $\HH$.  In the present finite-dimensional setting,  we shall use this well-posedness without further comment.

\medskip 

For $J\subset [N]$,  define the $J$-marginal by
\[
  \rho_t^{N:J}
  :=
  \Tr_{J^c}(\rho_t^N)
\]
and write in particular 
\(
  \rho_t^{N:n}
  :=
  \Tr_{[N]\setminus[n]}(\rho_t^N)
\) for 
$1\le n\le N$.   The product of the limiting one-particle processes is given as 
\(
  \Gamma_t^n
  :=
  \gamma_t^1\otimes\cdots\otimes\gamma_t^n.
\)

\begin{definition}[Propagation of chaos]
\label{def:propagation-of-chaos}
Let $\rho_0^N\in\cS(\HH^N)$ and $\gamma_0\in\cS(\HH)$.  Propagation of chaos is said to hold if, for every
fixed $n\ge1$ and $T>0$,
\begin{equation}
  \sup_{0\le t\le T}
  \E\left[
    \left\|
      \rho_t^{N:n}-\Gamma_t^n
    \right\|_1
  \right]
  \longrightarrow0,
  \qquad
  \text{as }N\longrightarrow \infty.
  \label{eq:propagation-of-chaos}
\end{equation}
\end{definition}

The convergence in \eqref{eq:propagation-of-chaos} says that every fixed
collection of particles is asymptotically described by independent copies of the
nonlinear one-particle filter.  Under the pathwise well-posedness of \eqref{eq:N-Belavkin} and  \eqref{eq:MF-Belavkin},  the
joint law of the synchronously coupled processes is uniquely determined.  Hence
the expectation in \eqref{eq:propagation-of-chaos} is independent of the
 filtered probability space on which the coupled equations  \eqref{eq:N-Belavkin},  \eqref{eq:MF-Belavkin} are
defined.

\medskip

For later use, let $\mathfrak S_N$ denote the permutation group of $[N]$.
For $\pi\in\mathfrak S_N$, let $U_\pi\in\cL(\HH^N)$ be the unitary operator 
defined by
\[
  U_\pi(x_1\otimes\cdots\otimes x_N)
  :=
  x_{\pi^{-1}(1)}\otimes\cdots\otimes x_{\pi^{-1}(N)}.
\]
\begin{definition}[Permutation invariance]
\label{def:exchangeability}
A  state $\rho^N\in\cS(\HH^N)$ is called permutation invariant if 
\[
  U_\pi\rho^N U_\pi^*=\rho^N,
  \qquad
  \mbox{ for all } \pi\in\mathfrak S_N.
\]
This is the labelled-particle analogue of the symmetry appearing in bosonic
closed quantum systems, although we do not restrict the state space to the symmetric
tensor product.
\end{definition}

This symmetry is preserved by the Belavkin dynamics in the sense of distributions.  Indeed, if
$\rho_0^N$ is permutation invariant, then
$U_\pi\rho_t^N U_\pi^*$ solves the same equation as $\rho_t^N$, but driven by
the relabelled Brownian family
\[
  (W^{k,\pi^{-1}(j)})_{1\le k\le m,\ 1\le j\le N}.
\]
Since this relabelled family has the same law as the original Brownian family,
and since $H^N$ and the measurement coefficients are invariant under particle
relabeling, uniqueness in law implies
\[
  (U_\pi\rho_t^N U_\pi^*)_{t\ge0}
  \stackrel{\mathrm{law}}{=}
  (\rho_t^N)_{t\ge0}.
\]
We shall refer to this property as \emph{exchangeability}.  The same terminology will
be used for coupled families, such as
\[
  \bigl(R^N,Z^1,\ldots,Z^N\bigr),
\]
where exchangeability means invariance in law under the simultaneous relabeling
\[
  \bigl(R^N,Z^1,\ldots,Z^N\bigr)
  \longmapsto
  \bigl(U_\pi R^N U_\pi^*,Z^{\pi^{-1}(1)},\ldots,Z^{\pi^{-1}(N)}\bigr).
\]
In particular, exact tensor-product initial data
$\rho_0^N=\gamma_0^{\otimes N}$ yield the exchangeability properties used
below.

\subsection{Main results}
\label{subsec:main-results}

We now state the main results, where  both theorems
allow arbitrary $\gamma_0\in\cS(\HH)$, arbitrary $L^{(k)}$, and every
$\eta\in(0,1]$.

\begin{theorem}
\label{thm:main-approx}
Assume that
\begin{equation}
  \delta_N:=\|\rho_0^N-\gamma_0^{\otimes N}\|_1\longrightarrow0.
  \label{eq:strong-tensorization}
\end{equation}
For every fixed $n\ge1$ and $T>0$, the following estimates hold for $N\ge n$.
\begin{enumerate}[label=\textup{(\roman*)},leftmargin=*,itemsep=.4em]
\item\label{item:strong-general}
For arbitrary $\gamma_0\in\cS(\HH)$, there is $C>0$, independent of $N$ and the initial states, such that
\begin{equation}
  \sup_{0\le t\le T}\E\left[\|\rho_t^{N:n}-\Gamma_t^n\|_1\right]\le C\delta_N+CN^{-1/4}.
  \label{eq:main-approx-bound}
\end{equation}

\item\label{item:strong-faithful}
If, in addition, $\gamma_0$ is faithful, there is $C'>0$, independent of $N$ and $\rho_0^N$, such that
\begin{equation}
  \sup_{0\le t\le T}\E\left[\|\rho_t^{N:n}-\Gamma_t^n\|_1\right]\le C'\delta_N+C'N^{-1/2}.
  \label{eq:main-approx-faithful-bound}
\end{equation}
\end{enumerate}
The constant $C$ depends only on $d,m,\|h\|,\|V\|,\|L^{(1)}\|,\ldots,\|L^{(m)}\|,\eta,T,n$, and $C'$ may additionally depend on the inverse of the smallest eigenvalue of $\gamma_0$.
\end{theorem}

The proof of \eqref{eq:main-approx-bound} combines purification, observation
reduction, and initial-state stability. The sharper estimate
\eqref{eq:main-approx-faithful-bound} has a completely different proof: the faithful
entropy estimate is applied to the exact-product comparison state, followed
by the same stability argument. Thus the faithful refinement does not replace
the arbitrary-rank proof.

\medskip

For later use, define the marginal errors by
\begin{equation}
 \delta_N^n:=\|\rho_0^{N:n}-\gamma_0^{\otimes n}\|_1,
 \qquad 1\le n\le N.
 \label{eq:weak-tensorization}
\end{equation}
\begin{theorem}
\label{thm:main-chaotic}
Assume that $\rho_0^N$ is permutation invariant and $\delta_N^2\longrightarrow0$.
For every fixed $n\ge1$ and $T>0$, the following estimates hold for $N\ge\max\{2,n\}$.
\begin{enumerate}[label=\textup{(\roman*)},leftmargin=*,itemsep=.4em]
\item\label{item:chaotic-general}
For arbitrary $\gamma_0\in\cS(\HH)$, there exist $\tau_*>0$ and $C_{T,n}>0$, independent of $N$ and the initial states, such that
\begin{equation}
  \sup_{0\le t\le T}\E\left[\|\rho_t^{N:n}-\Gamma_t^n\|_1\right]\le C_{T,n}\left(\delta_N^2+\frac1N\right)^{\vartheta(T)},
  \label{eq:main-chaotic-bound}
\end{equation}
where
\begin{equation}
  \vartheta(T):=2^{-3-\lceil T/\tau_*\rceil}.
  \label{eq:chaotic-time-exponent}
\end{equation}

\item\label{item:chaotic-faithful}
If, in addition, $\gamma_0$ is faithful, there is $C_{T,n,\gamma_0}>0$, independent of $N$ and $\rho_0^N$, such that
\begin{equation}
  \sup_{0\le t\le T}\E\left[\|\rho_t^{N:n}-\Gamma_t^n\|_1\right]\le C_{T,n,\gamma_0}\left(\delta_N^2+\frac1N\right)^{1/4}.
  \label{eq:main-chaotic-faithful-bound}
\end{equation}
\end{enumerate}
The constant $\tau_*$ depends only on $d,m,\|h\|,\|V\|,\|L^{(1)}\|,\ldots,\|L^{(m)}\|,\eta$.
The constant $C_{T,n}$ may additionally depend on $T,n$, and $C_{T,n,\gamma_0}$ may also depend on the inverse of the smallest eigenvalue of $\gamma_0$.
\end{theorem}

Under permutation invariance in fixed finite dimension, $\delta_N^2\longrightarrow 0$
is equivalent to $\delta_N^n\longrightarrow 0$ for every fixed $n$.  Its proof,
together with the quantitative initial-entropy estimate, is given in Appendix~\ref{app:initial}. 

\medskip

No optimality is asserted for the exponents. In
\eqref{eq:main-chaotic-bound}, the exponent deteriorates because the
reference is regularized on successive short intervals. In the faithful
case no such updates are needed,  and the exponent $1/4$ in
\eqref{eq:main-chaotic-faithful-bound} is a coarse conversion to an initial
trace-norm error.  The sharper entropy estimate in
Proposition~\ref{prop:faithful-entropy} retains more information
and gives $N^{-1/2}$ for exact faithful products.

\begin{remark}
The following examples distinguish the two initial-data assumptions with 
$$\HH=\mathbb C^2,\qquad |0\rangle:=\begin{pmatrix}1\\0\end{pmatrix},
  \qquad
  |1\rangle:=\begin{pmatrix}0\\1\end{pmatrix}.
$$
\begin{enumerate}
\item For even $N$, take 
\[
\gamma_0=\frac{1}{2} I_2,\qquad 
  \rho_0^N
  :=
  \binom{N}{N/2}^{-1}
  \sum_{\substack{(x_1,\ldots,x_N)\in\{0,1\}^N\\
                  x_1+\cdots+x_N=N/2}}
  \bigl(|x_1\rangle\langle x_1|\bigr)
  \otimes\cdots\otimes
  \bigl(|x_N\rangle\langle x_N|\bigr).
\]
By definition,  $\rho_0^N$ is permutation invariant.  Its fixed marginals $\rho_0^{N:n}$ are sampling
without replacement and converge to $\gamma_0^{\otimes n}$, but
\[
 \|\rho_0^N-\gamma_0^{\otimes N}\|_1
 =2\left(1-2^{-N}\binom{N}{N/2}\right)\longrightarrow2.
\]
Thus Theorem~\ref{thm:main-chaotic} yields propagation of chaos for this
sequence, along even $N$, although the strong tensorization assumption
of Theorem~\ref{thm:main-approx} fails.

\item The symmetry assumption also controls the population seen by a fixed subsystem.  Take 
\[
h=0=L^{(k)},\qquad 
Z=\begin{pmatrix}
1 & 0 \\
0 & -1
\end{pmatrix},  \qquad V=Z\otimes Z .
\]
Put
$|+\rangle:=(|0\rangle+|1\rangle)/\sqrt2$,
$\gamma_0:=|+\rangle\langle+|$, and $\sigma:=|0\rangle\langle0|$.
With $k_N:=\lfloor\sqrt N\rfloor$, set
\[
 \rho_0^N=\gamma_0^{\otimes k_N}\otimes\sigma^{\otimes(N-k_N)}.
\]
Every fixed initial prefix is eventually exactly tensorized, but the
population is predominantly in state $\sigma$. The commuting Hamiltonian
\eqref{eq:HN} gives
\[
 \rho_t^{N:1}\longrightarrow e^{-\i tZ}\gamma_0 e^{\i tZ},
 \qquad \gamma^1_t=\gamma_0,
 \qquad
 \lim_{N\to\infty}\|\rho_t^{N:1}-\gamma^1_t\|_1=2|\sin(t)|.
\]
Although $\delta_N^n=0$ eventually for every fixed $n$, the limiting
error $2|\sin(t)|$ is strictly positive for $0<t<\pi$.
Thus propagation of chaos towards the solution of \eqref{eq:MF-Belavkin}
initialized at $\gamma_0$ fails on every interval $[0,T]$ with $T>0$.
This example shows that the permutation-invariance assumption in
Theorem~\ref{thm:main-chaotic} cannot simply be omitted while retaining
only convergence of the fixed initial marginals.
\end{enumerate}
\end{remark}

\subsection{Comparison with closed systems and the role of conditioning}
\label{subsec:difficulty-and-open-problem}

When all measurement operators vanish, equation \eqref{eq:N-Belavkin}
reduces to the deterministic von Neumann equation
\[
  \i\partial_t\rho_t^N=[H^N,\rho_t^N].
\]
For closed quantum systems, propagation of chaos and the derivation of
Hartree equations have been established by several complementary methods.
These include BBGKY hierarchies, as in Spohn \citep{spohn1980} and
Bardos--Golse--Mauser \citep{bardos2000}, deviation and counting estimates,
as in Knowles--Pickl \citep{knowlespickl2010} and Pickl \citep{pickl2011},
Egorov-type theorems \citep{frohlich2007}, expansions of the Heisenberg
dynamics and phase-space methods \citep{Ammari2009}, quantum
empirical-measure techniques \citep{golsepaul2017,golsepaul2019}, and
quantum relative entropy \citep{Guo2026}. The BBGKY approach exploits the
linearity of the many-body evolution and the triangular structure of the
marginal equations: the equation for an $n$-particle marginal involves the
$(n+1)$-particle marginal through the interaction term. Relative entropy
provides another means of controlling correlations. In particular,
Guo--Liang--Wang \citep{Guo2026} obtain quantitative estimates for mixed
states by controlling entropy production against a tensor-product
reference.

Continuous observation changes the structure of this comparison. The state
$\rho_t^N$ in \eqref{eq:N-Belavkin} is conditioned on the records of all
particles, and its normalized evolution is nonlinear and stochastic.
Taking a partial trace retains the effect of measurements on particles
outside the chosen subsystem. Indeed, for $J\subset[N]$, $j\in [N]\setminus J$ and
$1\le k\le m$,
\begin{equation}
  \Tr_{J^c}\bigl(\cE_{L_j^{(k)}}(\rho_t^N)\bigr)
  ={}\Tr_{J^c}\bigl(\big(L_j^{(k)}+(L_j^{(k)})^*\big)\rho_t^N\bigr)-\Tr\bigl(\big(L_j^{(k)}+(L_j^{(k)})^*\big)\rho_t^N\bigr)
       \rho_t^{N:J}.
\label{eq:exterior-covariance}
\end{equation}
Testing this identity against an observable on $\HH^J$ gives its covariance
with the measured quadrature of particle $j$, evaluated in the conditional
state. The first term depends on the marginal on $J\cup\{j\}$, whereas the
second is a product involving the marginals on $J$ and $\{j\}$. Thus the
marginal equations involve labeled higher-order states and nonlinear
products of lower-order states, together with the exterior innovations.
They no longer form the linear triangular hierarchy used for closed
systems. Moreover, there are $N-|J|$ exterior particle labels, and their
stochastic terms have no $1/N$ prefactor. Controlling each fixed marginal
therefore has to be combined with an estimate for the aggregate effect of
these observation channels.

The same issue appears when the equations are tested against local
observables: the resulting conditional covariances involve higher
correlations. Passing to the linear Zakai equation handles normalization,
but its solution still depends on the full observation record. Averaging
over the records gives an unconditioned Lindblad evolution, whereas
\eqref{eq:propagation-of-chaos} requires the trace norm to be controlled
before taking expectation. The methods for closed systems thus provide
useful comparison quantities, but an additional argument is needed to
control their behavior under stochastic conditioning.

\medskip

The proofs use two different starting points, provided respectively by pure
and faithful one-particle reference states. We introduce the following assumptions to organize the successive reductions.

\begin{hypothesis}
\label{hyp:perfect-efficiency-tensorization-purity}
The following four sets of assumptions will be used separately below.
\begin{enumerate}
\item Perfect tensorization $\rho_0^N=\gamma_0^{\otimes N}$, perfect
 efficiency $\eta=1$, and  purity condition 
 $\operatorname{rank}(\gamma_0)=1$.

\item Perfect tensorization $\rho_0^N=\gamma_0^{\otimes N}$ and perfect
 efficiency $\eta=1$.

\item Perfect tensorization $\rho_0^N=\gamma_0^{\otimes N}$.

\item Permutation invariance of $\rho_0^N$ and faithfulness
 $\gamma_0>0$.
\end{enumerate}
\end{hypothesis}

The first three cases successively remove purity condition and perfect efficiency
while retaining exact tensorization. The fourth case is the starting point
of the entropy approach, which allows correlated initial states. 

\medskip

Under Hypothesis~\ref{hyp:perfect-efficiency-tensorization-purity}(1),
Kolokoltsov  established propagation
of chaos for skew-adjoint operators $L^{(1)},\ldots,  L^{(m)}$ in \citep{kolokoltsov2022},  and then for general measurement operators in \citep{kolokoltsov2021},  by extending the counting method to continuously observed quantum
systems. The key comparison uses the rank-one structure: for $p\in\cS(\HH)$ with $\operatorname{rank}(p)=1$,
\begin{equation}
  \|\rho-p\|_1^2
  \le 4\bigl(1-\Tr(\rho p)\bigr),
  \qquad \mbox{for all } \rho\in\cS(\HH).
  \label{eq:pure-state-trace-bound}
\end{equation}
The quantity $1-\Tr(\rho p)$ measures the mass outside the prescribed
one-dimensional subspace. Applied to the one-particle marginals, it gives
a scalar counting functional controlling the trace distance from a product
of pure states. Perfect observation preserves purity, so the one-particle
reference remains a rank-one projection. The resulting estimate, recalled
in Subsection~\ref{subsec:purification}, yields the rate $N^{-1/4}$ for
fixed marginals on finite time intervals. For a mixed reference $\gamma$,
this comparison requires a different argument.  Inefficient observation also allows pure initial states
to become mixed conditional states.

Under Hypothesis~\ref{hyp:perfect-efficiency-tensorization-purity}(2), we
represent $\gamma_0$ as the partial trace of a pure state on an enlarged
Hilbert space.  Here we lift the Hamiltonian,
interaction, and measurement operators and verify that partial tracing
recovers both the interacting equation and its limiting equation. The
pure-state estimate on the enlarged space then gives the same trace-norm
rate for arbitrary $\gamma_0$. Under
Hypothesis~\ref{hyp:perfect-efficiency-tensorization-purity}(3), we add the
unobserved part of each measurement channel and apply the preceding result
to the fully observed system. Conditioning on the actual records recovers
the original dynamics. Since conditioning on all records can correlate the
reference particles, a classical relative-entropy comparison of their
path laws is used to control the conditional-factorization error. A local
comparison then recovers the synchronous innovation coupling in
Definition~\ref{def:propagation-of-chaos}. These arguments constitute
Section~\ref{sec:exact}. Finally, the positivity of the linear Zakai
propagator and preservation of trace in expectation give stability under
full-state perturbations, with constants uniform in $N$. This proves
Theorem~\ref{thm:main-approx}(i) in Section~\ref{sec:extensions}.

\medskip

Hypothesis~\ref{hyp:perfect-efficiency-tensorization-purity}(4) provides
a second starting point. A faithful reference state has a logarithm on the
whole one-particle space, which makes quantum relative entropy available
for comparison with arbitrary many-body states. Purity and faithfulness
therefore supply different comparison structures: the former gives the
rank-one projection in \eqref{eq:pure-state-trace-bound}, while the latter
gives the logarithm used in the entropy estimate. The deterministic
relative-entropy method of Guo--Liang--Wang \citep{Guo2026} suggests how to
control the mean-field Hamiltonian difference. In the observed system,
the reference is a random product of one-particle filters, and the
interacting and reference observation laws are different. These features
require the observation record and the conditional quantum state to be
compared together.

We use the relative entropy of this joint classical and quantum object.
It decomposes into the classical relative entropy between observation laws
and the expected quantum relative entropy between conditional states.
The quantum-instrument formulation provides this decomposition and
contraction under a common observation step
\citep{BarchielliLupieri2004,BarchielliLupieri2006}. The interaction leaves
a Hamiltonian defect in the entropy-production inequality. Centering its
one-particle factors under the product reference permits an exponential
estimate uniform in $N$. Permutation invariance then distributes the
entropy bound over disjoint blocks and controls any fixed conditional
marginal. As in the first approach, a local stability argument passes from
filters evaluated on common observations to filters driven by common
innovations.  This gives Proposition~\ref{prop:faithful-entropy} in Section~\ref{sec:entropy}.
Lemma~\ref{lem:faithful-initial-conversion} bounds the initial relative entropy per particle by $C_{d,\gamma_0}(\delta_N^2+N^{-1})^{1/2}$.
Combining these two estimates proves Theorem~\ref{thm:main-chaotic}(ii).

The dependence on the reference eigenvalues must also be controlled.
A faithful filter remains positive definite on finite time intervals, but
its smallest eigenvalue may become arbitrarily small along observation
paths. Logarithmic-eigenvalue estimates control this
randomness in the exponential estimate. For an arbitrary-rank
$\gamma_0$, Section~\ref{sec:extensions} regularizes and updates only the
reference at finitely many deterministic times. The interacting state and
its initial condition remain unchanged. The controlled entropy cost of
each update, together with the local entropy estimate, removes
faithfulness and gives Theorem~\ref{thm:main-chaotic}(i), retaining
permutation invariance. The successive updates account for the
time-dependent exponent $\vartheta(T)$.

Finally, applying the faithful entropy proposition to exact product data
and then using the same initial-state stability estimate gives
Theorem~\ref{thm:main-approx}(ii). Permutation invariance is used only for
the exact-product comparison state on this route. Thus the two parts of
Theorem~\ref{thm:main-approx} share their final stability step, while their
exact-product estimates are obtained respectively from the pure-state
reduction and the faithful entropy method. The common linear constructions
and the initial-state stability estimate are formulated for arbitrary
initial states and can be used in each of the cases above.

\section{Exact tensorization: purification and observation reduction}
\label{sec:exact}

This section proves the exact tensor-product estimates used in
Theorem~\ref{thm:main-approx}(i). Under
Hypothesis~\ref{hyp:perfect-efficiency-tensorization-purity}(3), we show
that, for every fixed $n\ge1$, $T>0$, and $0<\eta\le1$,
\begin{equation}
  \sup_{0\le t\le T}
  \E\left[\|\rho_t^{N:n}-\Gamma_t^n\|_1\right]
  \le C_{\eta,T,n}N^{-1/4},
  \qquad N\ge n.
  \label{eq:exact-product-goal}
\end{equation}
By Cauchy--Schwarz,   \eqref{eq:exact-product-goal} indeed follows from the stronger mean-square estimate
\begin{equation}
  \sup_{0\le t\le T}
  \E\left[\|\rho_t^{N:n}-\Gamma_t^n\|_1^2\right]
  \le C_{\eta,T,n}N^{-1/2},
  \qquad N\ge n.
  \label{eq:exact-product-square}
\end{equation}
We first work under
Hypothesis~\ref{hyp:perfect-efficiency-tensorization-purity}(2):
the pure-state estimate under
Hypothesis~\ref{hyp:perfect-efficiency-tensorization-purity}(1)
is transferred to mixed states by purification. We then remove perfect
efficiency by adjoining the unobserved channels and conditioning on the
actual observations. Throughout these reductions, we retain the square
of the trace distance. Conditional Jensen's inequality, the path-space
entropy estimate, and the local comparison of filters then preserve the
$N^{-1/2}$ mean-square bound.

\subsection{Canonical observations and the common linear reference}
\label{subsec:common-canonical}

Subsection~\ref{subsec:common-canonical} sets up a common observation-space
representation used throughout the paper. In Section~\ref{sec:exact},
this framework allows us to extend the analysis from full observation
($\eta=1$) to partial observation ($0<\eta<1$). We begin by constructing
the Belavkin dynamics from a linear Zakai equation on the observation
space.

\medskip

We use $X$ to denote a generic element of either $\cL(\HH^N)$ or
$\cL(\HH)$, with the underlying space clear from the context.

Set
\begin{equation}
  \mathfrak L_N(X)
  :=
  -\i[H^N,X]
  +
  \sum_{k=1}^m\sum_{j=1}^N\cD_{L_j^{(k)}}(X),\qquad \forall  X\in\cL(\HH^N).
  \label{eq:eta-LN-def}
\end{equation}
Recall that \(\xi\) is the unique deterministic solution to 
\eqref{eq:mean-state-equation}.  Define the time-dependent one-particle
linear drift
\begin{equation}
  \mathfrak l_t(X)
  :=
  -\i[h+V^{\xi_t},X]
  +
  \sum_{k=1}^m\cD_{L^{(k)}}(X),\qquad \forall  X\in\cL(\HH).
  \label{eq:eta-little-l-def}
\end{equation}
For later use, write
\begin{equation}
  A^{(k)}
  :=
  L^{(k)}+(L^{(k)})^*,
  \qquad
  b^{(k)}(X)
  :=
  \Tr\bigl(A^{(k)}X\bigr),\qquad \forall  X\in\cL(\HH),
  \label{eq:eta-A-b-def}
\end{equation}
and, 
\begin{equation}
  b_j^{N,k}(X)
  :=
  \Tr\bigl(A_j^{(k)}X\bigr)
  =
  \Tr\bigl(A^{(k)}X^{N:(j)}\bigr),\qquad \forall  X\in\cL(\HH^N),
  \label{eq:eta-bNj-def}
\end{equation}
where
\(
  X^{N:(j)}
  :=
  \Tr_{[N]\setminus\{j\}}(X).
\)

Let $N\ge1$ and $T>0$. For $0\le t\le T$, let
\[
  \mathsf Y_t:=C_0([0,t];\mathbb R^m)
\]
be the space of continuous paths on 
$[0,t]$ taking values in $\mathbb R^m$ vanishing at zero,  and let $\mathbb W_T$ denote Wiener measure on $\mathsf Y_T$.
The common observation space and its Wiener law are
\begin{equation}
  \Omega_N:=\mathsf Y_T^N,
  \qquad
  \Q^N:=\mathbb W_T^{\otimes N}.
  \label{eq:canonical-observation-space}
\end{equation}
For $(y^1,\ldots,y^N)\in\Omega_N$, define the coordinate processes
on $\Omega_N$
\[
  Y_t^j(y^1,\ldots,y^N):=y_t^j,
  \qquad
  Y_t^{k,j}(y^1,\ldots,y^N):=(y_t^j)^k,
  \qquad
  1\le j\le N,\quad 1\le k\le m.
\]
Let $\cY_t^N$ be the usual augmentation under $\Q^N$ of
\begin{equation}
  \cY_t^{N,0}
  :=
  \sigma(Y_s^{k,j}:0\le s\le t,\ 1\le k\le m,\ 1\le j\le N).
  \label{eq:eta-YN-def}
\end{equation}
On $\mathsf Y_T$, we use the coordinate process $Y_t(y):=y_t$ and the
analogous augmented filtration $\cY_t$ under $\mathbb W_T$.  

\medskip

For an initial condition $\theta_0^N\in\cS(\HH^N)$, let
$\widehat\rho^{N,\theta}$ be the unique strong solution under $\Q^N$ to the
linear Zakai equation
\begin{equation}
\begin{aligned}
  \d\widehat\rho_t^{N,\theta}
  &=
  \mathfrak L_N\bigl(\widehat\rho_t^{N,\theta}\bigr)\,\d t
  +
  \sqrt{\eta}\sum_{k=1}^m\sum_{j=1}^N
  \cC_{L_j^{(k)}}\bigl(\widehat\rho_t^{N,\theta}\bigr)\,\d Y_t^{k,j},
  \qquad
  \widehat\rho_0^{N,\theta}=\theta_0^N.
\end{aligned}
  \label{eq:reference-linear}
\end{equation}
The superscript \(\theta\) is only a label for the initial condition.
The coefficients and the reference Brownian motions are the same for all
choices of \(\theta_0^N\).

The next lemma shows that the linear evolution preserves positivity
and preserves the trace in expectation.

\begin{lemma}
\label{lem:CP-reference}
For every \(t\in[0,T]\), equation \eqref{eq:reference-linear} defines a
random completely positive linear map
\[
  \Phi_t^{N,\eta}:\cL(\HH^N)\longrightarrow\cL(\HH^N)
\]
by
\(
  \Phi_t^{N,\eta}(\theta_0^N)
  :=
  \widehat\rho_t^{N,\theta}
 \).  Moreover, for every \(B\in\cL(\HH^N)\),
\begin{equation}
  \E_{\Q^N}\left[
    \Tr\bigl(\Phi_t^{N,\eta}(B)\bigr)
  \right]
  =
  \Tr(B).
  \label{eq:reference-mean-trace}
\end{equation}
For $B\ge0$ and $B\ne0$,
\(
  \Tr\bigl(\Phi_t^{N,\eta}(B)\bigr)>0
\) for all $t\in[0,T]$.  
If $B>0$, then
\(
  \Phi_t^{N,\eta}(B)>0
\) for all $t\in[0,T]$,  and,  for every $n\ge 1$,
\[
  \E_{\Q^N}\left[
    \sup_{0\le t\le T}
    \left(
      \|\Phi_t^{N,\eta}(B)\|^n
      +
      \|(\Phi_t^{N,\eta}(B))^{-1}\|^n
    \right)
  \right]<\infty.
\]
The same conclusions hold if $H^N$ is replaced by a bounded,
deterministic, continuous self-adjoint Hamiltonian $t\longmapsto H_t^N$.
\end{lemma}
This lemma records the faithfulness and moment properties needed
later in the entropy calculation, and its proof is given in
Appendix~\ref{app:zakai}.

\medskip

Set
\begin{equation}
  Z_t^{N,\theta}
  :=
  \Tr\bigl(\widehat\rho_t^{N,\theta}\bigr),
  \qquad
  \rho_t^{N,\theta}
  :=
  \frac{\widehat\rho_t^{N,\theta}}{Z_t^{N,\theta}}.
  \label{eq:perturbation-normalization}
\end{equation}
By Lemma~\ref{lem:CP-reference}, 
$\widehat\rho_t^{N,\theta}\ge 0$ and $Z_t^{N,\theta}>0$ for all 
$t\in [0,T]$.  Taking the trace in \eqref{eq:reference-linear} yields
\begin{equation}
\begin{aligned}
  \d Z_t^{N,\theta}
  &=
  \sqrt{\eta}\,Z_t^{N,\theta}
  \sum_{k=1}^m\sum_{j=1}^N
  b_j^{N,k}(\rho_t^{N,\theta})\,\d Y_t^{k,j}.
\end{aligned}
  \label{eq:perturbation-Z-SDE}
\end{equation}
Since
\(
 \sup_{R\in \cS(\HH^N)} \left|b_j^{N,k}(R)\right|
  \le
  \|A^{(k)}\|
\), 
\(Z^{N,\theta}\) is a strictly positive martingale with mean one by Novikov's condition.
Define thus the probability measure
\begin{equation}
  \frac{\d\Q^{N,\theta}}{\d\Q^N}
  \bigg|_{\cY_T^N}
  :=
  Z_T^{N,\theta}.
  \label{eq:physical-likelihood-family}
\end{equation}
By Girsanov's theorem, the processes
\begin{equation}
  W_t^{N,\theta,k,j}
  :=
  Y_t^{k,j}
  -
  \sqrt{\eta}\int_0^t
  b_j^{N,k}(\rho_s^{N,\theta})\,\d s
  \label{eq:physical-innovation-family}
\end{equation}
form a family of independent Brownian motions under $\Q^{N,\theta}$.  Applying It\^o's formula to
$\rho_t^{N,\theta}=\widehat\rho_t^{N,\theta}/Z_t^{N,\theta}$ yields 
\begin{equation}
\begin{aligned}
  \d\rho_t^{N,\theta}
  &=
  \mathfrak L_N(\rho_t^{N,\theta})\,\d t
  +
  \sqrt{\eta}\sum_{k=1}^m\sum_{j=1}^N
  \cE_{L_j^{(k)}}(\rho_t^{N,\theta})\,\d W_t^{N,\theta,k,j}.
\end{aligned}
  \label{eq:perturbation-nonlinear-N}
\end{equation}
Hence, under $\Q^{N,\theta}$, the process $\rho^{N,\theta}$ is a weak
solution to the \(N\)-particle Belavkin
equation \eqref{eq:N-Belavkin} with initial condition
$\theta_0^N$.

We now specialize the preceding construction to the system introduced
in Section~\ref{sec:framework}, i.e.,  $\theta_0^N=\rho^N_0$.  We note that the general  condition $\theta_0^N$ was introduced to allow comparisons between different
initial states on the same observation space in the stability analysis
of Section~\ref{sec:extensions}.  For $\theta_0^N=\rho^N_0$,  set
\[
  \widehat\rho_t^N:=\Phi_t^{N,\eta}(\rho_0^N),
  \qquad
  Z_t^N:=\Tr(\widehat\rho_t^N),
  \qquad
  \rho_t^N:=\frac{\widehat\rho_t^N}{Z_t^N}.
\]
We denote by $\Q^{N,\eta}$ the probability measure on
$(\Omega_N,\cY_T^N)$ defined in
\eqref{eq:physical-likelihood-family} with $\theta_0^N=\rho_0^N$.
Its dependence on $\rho_0^N$ and the fixed horizon $T$ is suppressed
in the notation.  
For the same choice of initial state, we write $W^{k,j}$ for the
innovation processes $W^{N,\theta,k,j}$ defined in
\eqref{eq:physical-innovation-family}.

Under $\Q^{N,\eta}$, the processes $W^{k,j}$ are independent Brownian
motions, and $\rho^N$ is a weak solution to
\eqref{eq:N-Belavkin} with initial condition $\rho_0^N$.
On the same probability space, let $\gamma^j$ be the solution to the mean-field equation~\eqref{eq:MF-Belavkin} driven by $(W^{k,j})_{k=1}^m$
with initial condition $\gamma_0$.

\medskip

We next construct the one-particle filter as a function of its own
observation path. This separates it from the target process $\gamma^j$,
which has just been coupled through the interacting innovations.

On the canonical space $(\mathsf Y_T,\cY_T,\mathbb W_T)$,  consider the linear equation
\begin{equation}
\begin{aligned}
  \d\widehat{\overline\gamma}_t
  &=
  \mathfrak l_t(\widehat{\overline\gamma}_t)\,\d t
  +
  \sqrt{\eta}\sum_{k=1}^m
  \cC_{L^{(k)}}(\widehat{\overline\gamma}_t)\,\d Y_t^k,
  \qquad
  \widehat{\overline\gamma}_0=\gamma_0.
\end{aligned}
  \label{eq:eta-local-linear-SDE}
\end{equation}
The solution of \eqref{eq:eta-local-linear-SDE} defines a
progressively measurable map
$Y\longmapsto\widehat{\overline\gamma}(Y)$.  Set
\begin{equation}
  \overline Z_t:=\Tr(\widehat{\overline\gamma}_t),
  \qquad
  \Pi_t(Y):=\frac{\widehat{\overline\gamma}_t}{\overline Z_t},
  \qquad
  \left.\frac{\d\mathbf q_T^Y}{\d\mathbb W_T}\right|_{\cY_T}
  :=\overline Z_T.
  \label{eq:eta-Pi-def}
\end{equation}
 Applying It\^o's formula to
the normalization in \eqref{eq:eta-Pi-def} gives
\begin{equation}
\begin{aligned}
  \d\Pi_t
  &=
  \mathfrak l_t(\Pi_t)\,\d t
  +
  \sqrt{\eta}\sum_{k=1}^m
  \cE_{L^{(k)}}(\Pi_t)
  \left(
    \d Y_t^k
    -
    \sqrt{\eta}\,b^{(k)}(\Pi_t)\,\d t
  \right).
\end{aligned}
  \label{eq:eta-local-observation-SDE}
\end{equation}
Under the measure $\mathbf q_T^Y$, the processes
\[
  \overline W_t^k
  :=
  Y_t^k-\sqrt\eta\int_0^t b^{(k)}(\Pi_s(Y))\,\d s
\]
are independent Brownian motions. Moreover,
\begin{equation}
  \E_{\mathbf q_T^Y}[\Pi_t(Y)]
  = 
  \E_{\mathbb W_T}[\widehat{\overline\gamma}_t]
  =
  \xi_t,
  \label{eq:reference-Pi-mean}
\end{equation}
where the second equality uses the definition of 
$\xi$ in  \eqref{eq:mean-state-equation}.

\medskip

We now evaluate the one-particle reference filter along the observation paths of the interacting system.  On \((\Omega_N,\cY_T^N,\Q^{N,\eta})\), define
\begin{equation}
  \overline\gamma_t^{N,j}
  :=
  \Pi_t(Y^j),
  \qquad
  Y^j=(Y^{1,j},\ldots,Y^{m,j}).
  \label{eq:eta-bar-gamma-N-def}
\end{equation}
Set further
\begin{equation}
  \overline\Gamma_t^{N,n}
  :=
  \overline\gamma_t^{N,1}
  \otimes\cdots\otimes
  \overline\gamma_t^{N,n}.
  \label{eq:eta-bar-Gamma-product}
\end{equation}
Define the unnormalized product reference and its likelihood by
\begin{equation}
  \widehat{\overline\Gamma}_t^N
  :=
 \widehat{\overline\gamma}_t(Y^1)
\otimes\cdots\otimes
\widehat{\overline\gamma}_t(Y^N), 
  \qquad
  \overline Z_t^N
  :=
  \prod_{j=1}^N\overline Z_t(Y^j),
  \qquad
  \overline\Q^N
  :=
  (\mathbf q_T^Y)^{\otimes N}.
  \label{eq:product-reference-definitions}
\end{equation}
Then
\[
  \widehat{\overline\Gamma}_t^N
  =
  \overline Z_t^N\overline\Gamma_t^{N,N},
  \qquad
  \left.\frac{\d\overline\Q^N}{\d\Q^N}\right|_{\cY_t^N}
  =
  \overline Z_t^N.
\]
The processes
\[
  \overline W_t^{k,j}
  :=
  Y_t^{k,j}
  -
  \sqrt\eta\int_0^t
  b^{(k)}(\overline\gamma_s^{N,j})\,\d s
\]
are independent Brownian motions under $\overline\Q^N$.
Under $\Q^N$ and $\overline\Q^N$,  the processes
$(\overline\gamma^{N,j})_{j=1}^N$ are independent, since
$(Y^j)_{j=1}^N$ are independent. Under $\Q^{N,\eta}$,
however, $(\overline\gamma^{N,j})_{j=1}^N$ need not be
independent, while each $\gamma^j$ is driven by
$(W^{k,j})_{k=1}^m$.

The notation introduced in this subsection is summarized below.  Hats denote
unnormalized positive matrices. Bars distinguish the filtering reference
from the innovation-driven target.

\begingroup
\small
\renewcommand{\arraystretch}{1.1}
\begin{longtable}{@{}p{.33\textwidth}p{.62\textwidth}@{}}
$\mathsf Y_T,\ \mathbb W_T$
& One-particle observation-path space and Wiener law.\\

$\Omega_N,\ \Q^N$
& $N$-particle observation-path space and product Wiener law.\\

$Y,\ Y^j,\ Y^{k,j}$
& Canonical observation coordinates, with augmented filtrations
$\cY_t$ and $\cY_t^N$.\\

$\Phi_t^{N,\eta}$
& Completely positive linear Zakai propagator.\\

$\widehat\rho^{N,\theta},\ Z^{N,\theta},\ \rho^{N,\theta}$
& Unnormalized state, trace likelihood, and normalized state for
initial condition $\theta_0^N$.\\

$\Q^{N,\theta},\ W^{N,\theta,k,j}$
& Corresponding physical law and innovation Brownian motions.\\

$\widehat\rho^N,\ Z^N,\ \rho^N$
& The preceding three processes with $\theta_0^N=\rho_0^N$.\\

$\Q^{N,\eta},\ W^{k,j}$
& Physical law and innovation Brownian motions for $\rho_0^N$.\\

$\gamma^j,\ \Gamma^n$
& Innovation-driven limiting filters and their tensor product.\\

$\widehat{\overline\gamma},\ \overline Z,\ \Pi$
& One-particle unnormalized Zakai state, trace likelihood,
and normalized filter.\\

$\mathbf q_T^Y,\ \overline\Q^N$
& One-particle reference observation law and its $N$-fold product.\\

$\overline\gamma^{N,j},\ \overline\Gamma^{N,n}$
& Reference filter $\Pi(Y^j)$ and the product of the first
$n$ filters.\\

$\widehat{\overline\Gamma}^N,\ \overline Z^N$
& Unnormalized product reference and its trace likelihood
relative to $\Q^N$.\\

$\overline W^k,\ \overline W^{k,j}$
& Innovations under $\mathbf q_T^Y$ and $\overline\Q^N$,
respectively.\\
\end{longtable}
\endgroup

We now turn to the exact tensor-product estimates. We first treat
perfect efficiency by purification, and then handle inefficient
observation through a fully observed dilation and conditional
factorization.

\subsection{Perfect efficiency: purification}
\label{subsec:purification}
\label{sec:purification}

Let Hypothesis~\ref{hyp:perfect-efficiency-tensorization-purity}(2) hold
throughout Subsection~\ref{subsec:purification}, namely,
\[
  \rho_0^N=\gamma_0^{\otimes N}
  \qquad \mbox{and}\qquad
  \eta=1.
\]
The goal of this subsection is to show the propagation
of chaos using a purification reduction, as summarized in Proposition
\ref{prop:main-exact-eta-one}\footnote{The purification
argument also applies in infinite-dimensional settings whenever the
lifted dynamics is well posed and satisfies the pure-state estimate.}
below.

\begin{proposition}
\label{prop:main-exact-eta-one}
Let Hypothesis~\ref{hyp:perfect-efficiency-tensorization-purity}(2) hold
with $\gamma_0\in\cS(\HH)$.
Then, for every fixed $n\ge1$ and $T>0$, there exists
$C_{T,n}>0$, independent of $N$ and $\gamma_0$,  such that
\begin{equation}
  \sup_{0\le t\le T}
  \E\left[\|\rho_t^{N:n}-\Gamma_t^n\|_1^2\right]
  \le C_{T,n}N^{-1/2},
  \qquad N\ge n.
  \label{eq:eta-one-mixed-square}
\end{equation}
In particular, after increasing the constant if necessary,
\begin{equation}
  \sup_{0\le t\le T}
  \E\left[
    \left\|\rho_t^{N:n}-\Gamma_t^n\right\|_1
  \right]
  \le
  C_{T,n}N^{-1/4}.
  \label{eq:eta-one-mixed-rate}
\end{equation}
\end{proposition}

The proof of Proposition \ref{prop:main-exact-eta-one} replies on a transfer principle. Estimate inequality   \eqref{eq:eta-one-mixed-rate}
 is already known in 
\citep{kolokoltsov2021}
when $\gamma_0$ is a pure state.  For mixed $\gamma_0$, the idea is that it can be written as the
partial trace of a pure state on an enlarged one-particle Hilbert space.
The Hamiltonian, interaction and measurement operators are lifted by letting
them act trivially on the auxiliary factor. The lifted dynamics is then a
perfectly observed pure-state dynamics, so the result of Kolokoltsov
\citep{kolokoltsov2021} applies. The physical mixed dynamics is recovered
exactly by partial trace, and the desired estimate follows from
trace-norm contractivity.

We first record the convergence result of Kolokoltsov's pure-state
propagation theorem in the notation of the present paper, see
\citep{kolokoltsov2021} for related details.

\begin{theorem}[Kolokoltsov, 2021]
\label{thm:pure-input}
Let Hypothesis~\ref{hyp:perfect-efficiency-tensorization-purity}(1) hold.  
For $j\in[N]$,  recall $\rho_t^{N:(j)}
  :=
  \Tr_{[N]\setminus\{j\}}(\rho_t^N)$ and set
\begin{equation}
  \alpha_{N,j}(t)
  :=
  1-\Tr\bigl(\rho_t^{N:(j)}\gamma_t^j\bigr).
  \label{eq:pure-counting-functional}
\end{equation}
Then, for every $T>0$, there exists a constant $C_T<\infty$, independent
of $N$,  $j$ and $\gamma_0$,  such that
\begin{equation}
  \sup_{0\le t\le T}
  \E\bigl[\alpha_{N,j}(t)\bigr]
  \le C_TN^{-1/2}.
  \label{eq:pure-counting-rate}
\end{equation}
Consequently,  for every fixed $n\ge1$,
\begin{equation}
  \sup_{0\le t\le T}
  \E\left[\|\rho_t^{N:n}-\Gamma_t^n\|_1^2\right]
  \le 4nC_TN^{-1/2},
  \qquad N\ge n.
  \label{eq:pure-trace-norm-square}
\end{equation}
\end{theorem}

\begin{remark}
The proof makes use of the following inequalities,   as the product $\Gamma_t^n$ is also pure, 
\begin{align}
 \left\|\rho_t^{N:n}-\Gamma_t^n\right\|_1^2
  \le
  4 \Big(1-\Tr\bigl(\rho_t^{N:n}\Gamma_t^n\bigr) \Big)  
    \le 4\sum_{j=1}^n\alpha_{N,j}(t).
  \label{eq:product-counting-bound}
\end{align}
No optimality is asserted for the rate.
\end{remark}

Next, we adopt the purification approach and verify the lifted dynamics
to remove the quantum purity assumption on $\gamma_0$.

\begin{lemma}
\label{lem:purification}
For every $\gamma_0\in\cS(\HH)$, there exist a finite-dimensional
auxiliary Hilbert space $\KK$, with
\(
  \dim(\KK)\le\dim(\HH),
\)
and a unit vector
\(
  \Psi_0\in\widetilde\HH:=\HH\otimes\KK
\)
such that, with
\(
  \widetilde\gamma_0:=|\Psi_0\rangle\langle\Psi_0|,
\)
one has
\begin{equation}
  \Tr_{\KK}(\widetilde\gamma_0)=\gamma_0.
  \label{eq:initial-purification}
\end{equation}
\end{lemma}

\begin{proof}
Denote $ r_0:=\operatorname{rank}(\gamma_0)$.  Let
\[
  \gamma_0
  =
  \sum_{r=1}^{r_0}\lambda_r|u_r\rangle\langle u_r|,\qquad \mbox{with}\qquad   \lambda_r>0
  \quad \mbox{and}\quad 
  \sum_{r=1}^{r_0}\lambda_r=1
\]
be the spectral decomposition of $\gamma_0$, where
$(u_r)_{1\le r\le r_0}\subset\HH$ are orthogonal unit vectors.  Let $\KK$ be an $r_0$-dimensional Hilbert space with orthonormal basis
$(e_r)_{r=1}^{r_0}$. Define
\[
  \Psi_0
  :=
  \sum_{r=1}^{r_0}\sqrt{\lambda_r}\,u_r\otimes e_r.
\]
Then $\|\Psi_0\|_{\widetilde\HH}=1$, and a straightforward computation yields
\[
  \widetilde\gamma_0
  =
  \sum_{r,s=1}^{r_0}
  \sqrt{\lambda_r\lambda_s}\,
  |u_r\rangle\langle u_s|
  \otimes
  |e_r\rangle\langle e_s|.
\]
Taking the partial trace over $\KK$ gives
\[
\begin{aligned}
  \Tr_{\KK}(\widetilde\gamma_0)
  =
  \sum_{r,s=1}^{r_0}
  \sqrt{\lambda_r\lambda_s}\,
  |u_r\rangle\langle u_s|
  \langle e_s,e_r\rangle_{\KK}=
  \sum_{r=1}^{r_0}\lambda_r|u_r\rangle\langle u_r|
  =
  \gamma_0.
\end{aligned}
\]
This proves the lemma.
\end{proof}

Set
\(
  \widetilde\HH:=\HH\otimes\KK
\)
and define the lifted one-particle operators by
\begin{equation}
  \widetilde h:=h\otimes I_{\KK},
  \qquad
  \widetilde L^{(k)}:=L^{(k)}\otimes I_{\KK},
  \qquad
  k=1,\ldots,m.
  \label{eq:lifted-one-particle-operators}
\end{equation}
Under the canonical identification
\(
  (\HH\otimes\KK)\otimes(\HH\otimes\KK)
  \simeq
  (\HH\otimes\HH)\otimes(\KK\otimes\KK),
\)
define
\begin{equation}
  \widetilde V:=V\otimes I_{\KK\otimes\KK}.
  \label{eq:lifted-interaction}
\end{equation}
Thus $\widetilde V$ acts as $V$ on the two physical factors and trivially
on the two auxiliary factors. Define further
\(
  \widetilde\HH^N:=\widetilde\HH^{\otimes N}
\)
and
\begin{equation}
  \widetilde H^N
  :=
  \sum_{j=1}^N\widetilde h_j
  +
  \frac1N\sum_{1\le j<\ell\le N}\widetilde V_{j\ell}.
  \label{eq:lifted-HN}
\end{equation}
Let
\[
  \mathbf T_N
  :=
  \Tr_{\KK^{\otimes N}}
  :
  \cL(\widetilde\HH^N)\longrightarrow\cL(\HH^N)
\]
denote the partial trace over all auxiliary factors.
The next lemma expresses the compatibility of the lifted coefficients with
the physical dynamics.

\begin{lemma}
\label{lem:trace-down}
For every $X\in\cL(\widetilde\HH^N)$, $j\in[N]$ and $1\le k\le m$,
\begin{align}
  \mathbf T_N\left(-\i[\widetilde H^N,X]\right)
  &=
  -\i[H^N,\mathbf T_N(X)],
  \label{eq:trace-down-H}\\
  \mathbf T_N\left(\cD_{\widetilde L_j^{(k)}}(X)\right)
  &=
  \cD_{L_j^{(k)}}\bigl(\mathbf T_N(X)\bigr),
  \label{eq:trace-down-D}\\
  \mathbf T_N\left(\cC_{\widetilde L_j^{(k)}}(X)\right)
  &=
  \cC_{L_j^{(k)}}\bigl(\mathbf T_N(X)\bigr),
  \label{eq:trace-down-C}\\
  \mathbf T_N\left(\cE_{\widetilde L_j^{(k)}}(X)\right)
  &=
  \cE_{L_j^{(k)}}\bigl(\mathbf T_N(X)\bigr).
  \label{eq:trace-down-E}
\end{align}
Moreover, if
\(
  \widetilde\sigma\in\cS(\widetilde\HH)
\)
and
\(
  \sigma:=\Tr_{\KK}(\widetilde\sigma),
\)
then
\begin{equation}
  \widetilde V^{\widetilde\sigma}
  =
  V^\sigma\otimes I_{\KK}.
  \label{eq:trace-down-effective-potential}
\end{equation}
\end{lemma}

\begin{proof}
Under the canonical identification
\(
  \widetilde\HH^N\simeq\HH^N\otimes\KK^{\otimes N},
\)
the lifted Hamiltonian satisfies
\(
  \widetilde H^N=H^N\otimes I_{\KK^{\otimes N}}.
\)
For every $A\in\cL(\HH^N)$,
\begin{equation}
\begin{aligned}
  \mathbf T_N\left((A\otimes I_{\KK^{\otimes N}})X\right)
  &=
  A\mathbf T_N(X),
  \qquad
  \mathbf T_N\left(X(A\otimes I_{\KK^{\otimes N}})\right)
  =
  \mathbf T_N(X)A.
\end{aligned}
  \label{eq:partial-trace-pull-through}
\end{equation}
Identity \eqref{eq:trace-down-H} follows immediately.

Since
\(
  \widetilde L_j^{(k)}
  =
  L_j^{(k)}\otimes I_{\KK^{\otimes N}},
\)
the same pull-through identities give
\[
\begin{aligned}
  \mathbf T_N\left(
    \widetilde L_j^{(k)}X(\widetilde L_j^{(k)})^*
  \right)
  &=
  L_j^{(k)}\mathbf T_N(X)(L_j^{(k)})^*,\qquad \mathbf T_N\left(
    (\widetilde L_j^{(k)})^*\widetilde L_j^{(k)}X
  \right)=
  (L_j^{(k)})^*L_j^{(k)}\mathbf T_N(X),\\
  \mathbf T_N\left(
    X(\widetilde L_j^{(k)})^*\widetilde L_j^{(k)}
  \right)
  &=
  \mathbf T_N(X)(L_j^{(k)})^*L_j^{(k)}.
\end{aligned}
\]
This proves \eqref{eq:trace-down-D}, and
\eqref{eq:trace-down-C} follows in the same way. Furthermore,
\[
\begin{aligned}
  \Tr\left(
    \bigl(\widetilde L_j^{(k)}+(\widetilde L_j^{(k)})^*\bigr)X
  \right)
  &=
  \Tr\left(
    \bigl(L_j^{(k)}+(L_j^{(k)})^*\bigr)\mathbf T_N(X)
  \right).
\end{aligned}
\]
Combining this identity with \eqref{eq:trace-down-C} proves
\eqref{eq:trace-down-E}.

Finally, \eqref{eq:trace-down-effective-potential} follows from the
defining duality of the partial trace. Indeed, for
$A\in\cL(\HH)$ and $B\in\cL(\KK)$,
\[
\begin{aligned}
  \Tr\left((A\otimes B)\widetilde V^{\widetilde\sigma}\right)
  =
  \Tr(B)\,\Tr\left((A\otimes\sigma)V\right)=
  \Tr(B)\,\Tr\left(AV^\sigma\right)
  =
  \Tr\left((A\otimes B)(V^\sigma\otimes I_{\KK})\right).
\end{aligned}
\]
Operators of the form $A\otimes B$ span $\cL(\widetilde\HH)$, which proves
the identity.
\end{proof}

Let $\widetilde\rho^N$ solve the lifted $N$-particle Belavkin equation
\begin{equation}
\begin{aligned}
  \d\widetilde\rho_t^N
  =
  -\i[\widetilde H^N,\widetilde\rho_t^N]\,\d t
  +
  \sum_{k=1}^m\sum_{j=1}^N
  \cD_{\widetilde L_j^{(k)}}(\widetilde\rho_t^N)\,\d t +
  \sum_{k=1}^m\sum_{j=1}^N
  \cE_{\widetilde L_j^{(k)}}(\widetilde\rho_t^N)\,\d W_t^{k,j},
  \qquad
  \widetilde\rho_0^N=\widetilde\gamma_0^{\otimes N}.
\end{aligned}
  \label{eq:lifted-N-eta-one}
\end{equation}
Let $\widetilde\gamma^j$ solve the lifted limiting equation
\begin{equation}
\begin{aligned}
  \d\widetilde\gamma_t^j
  &=    -\i[\widetilde h+\widetilde V^{\widetilde\xi_t},
        \widetilde\gamma_t^j] \d t
    +
    \sum_{k=1}^m
    \cD_{\widetilde L^{(k)}}(\widetilde\gamma_t^j)
  \d t +
  \sum_{k=1}^m
  \cE_{\widetilde L^{(k)}}(\widetilde\gamma_t^j)\,\d W_t^{k,j},
  \\
  \widetilde\gamma_0^j
  &=
  \widetilde\gamma_0,
  \qquad
  \widetilde\xi_t:=\E[\widetilde\gamma_t^j].
\end{aligned}
  \label{eq:lifted-MF-eta-one}
\end{equation}
The initial states in both lifted equations
\eqref{eq:lifted-N-eta-one} and \eqref{eq:lifted-MF-eta-one} are pure.
Since every output channel is perfectly observed, the lifted equations
preserve purity. The next lemma shows that, after taking the partial trace,
the lifted equations reproduce the solutions to \eqref{eq:N-Belavkin}
and \eqref{eq:MF-Belavkin}.

\begin{lemma}
\label{lem:trace-down-dynamics}
The processes
\begin{equation}
  \rho_t^N:=\mathbf T_N(\widetilde\rho_t^N),
  \qquad
  \gamma_t^j:=\Tr_{\KK}(\widetilde\gamma_t^j)
  \label{eq:trace-down-processes}
\end{equation}
solve respectively \eqref{eq:N-Belavkin} and \eqref{eq:MF-Belavkin}
with $\eta=1$ and initial conditions
$\gamma_0^{\otimes N}$ and $\gamma_0$.
\end{lemma}

\begin{proof}
Note first $ \mathbf T_N\left(\widetilde\gamma_0^{\otimes N}\right)
 =\left(\Tr_{\KK}(\widetilde\gamma_0)\right)^{\otimes N}
 =\gamma_0^{\otimes N}$
at time zero.  Applying the bounded linear map $\mathbf T_N$ to the integral form of
\eqref{eq:lifted-N-eta-one} and using Lemma~\ref{lem:trace-down}
yields \eqref{eq:N-Belavkin} with $\eta=1$.

For the limiting equation, set
\(
  \xi_t:=\E[\gamma_t^j].
\)
Linearity of expectation and partial trace gives
\[
  \Tr_{\KK}(\widetilde\xi_t)
  =
  \Tr_{\KK}\bigl(\E[\widetilde\gamma_t^j]\bigr)
  =
  \E\left[\Tr_{\KK}(\widetilde\gamma_t^j)\right]
  =
  \E[\gamma_t^j]
  =
  \xi_t.
\]
Consequently, Lemma~\ref{lem:trace-down} gives
\[
  \widetilde V^{\widetilde\xi_t}=V^{\xi_t}\otimes I_{\KK}.
\]
Applying $\Tr_{\KK}$ to \eqref{eq:lifted-MF-eta-one} therefore yields
\eqref{eq:MF-Belavkin} with $\eta=1$. Pathwise uniqueness identifies
the trace-down processes with the solutions introduced in
Section~\ref{sec:framework}.
\end{proof}

We shall also use the following standard contractivity property.

\begin{lemma}
\label{lem:partial-trace-contraction}
For every $n\ge1$ and $X\in\cL(\widetilde\HH^{\otimes n})$,
\begin{equation}
  \left\|\Tr_{\KK^{\otimes n}}(X)\right\|_1
  \le
  \|X\|_1.
  \label{eq:partial-trace-contraction}
\end{equation}
\end{lemma}

\begin{proof}
By trace duality,
\[
\begin{aligned}
  \left\|\Tr_{\KK^{\otimes n}}(X)\right\|_1=
  \sup_{A\in\cL(\HH^{\otimes n}):\,\|A\|\le1}
  \left|
    \Tr\left(A^*\Tr_{\KK^{\otimes n}}(X)\right)
  \right|=
  \sup_{A\in\cL(\HH^{\otimes n}):\,\|A\|\le1}
  \left|
    \Tr\left((A^*\otimes I_{\KK^{\otimes n}})X\right)
  \right|
  \le
  \|X\|_1,
\end{aligned}
\]
because
\(
  \|A\otimes I_{\KK^{\otimes n}}\|=\|A\|.
\)
\end{proof}

Now we collect all ingredients to prove
Proposition~\ref{prop:main-exact-eta-one}.

\begin{proof}[Proof of Proposition~\ref{prop:main-exact-eta-one}]
For $1\le n\le N$, set
\[
  \widetilde\rho_t^{N:n}
  :=
  \Tr_{[N]\setminus[n]}(\widetilde\rho_t^N)
  \qquad \mbox{ and }\qquad
  \widetilde\Gamma_t^n
  :=
  \widetilde\gamma_t^1\otimes\cdots\otimes\widetilde\gamma_t^n.
\]
Since $\widetilde\gamma_0$ is pure, apply
\eqref{eq:pure-trace-norm-square} on the enlarged one-particle space
$\widetilde\HH$. There exists $\widetilde C_{T,n}<\infty$, independent
of $N$, such that
\begin{equation}
  \sup_{0\le t\le T}
  \E\left[
    \left\|\widetilde\rho_t^{N:n}-\widetilde\Gamma_t^n\right\|_1^2
  \right]
  \le \widetilde C_{T,n}N^{-1/2}.
  \label{eq:lifted-pure-chaos-rate}
\end{equation}
In view of Lemma~\ref{lem:trace-down-dynamics} and the associativity
of partial traces, it holds that
\(
  \rho_t^{N:n}
  =
  \Tr_{\KK^{\otimes n}}(\widetilde\rho_t^{N:n}).
\)
Moreover,
\[
\begin{aligned}
  \Gamma_t^n
  &=
  \Tr_{\KK}(\widetilde\gamma_t^1)
  \otimes\cdots\otimes
  \Tr_{\KK}(\widetilde\gamma_t^n)
  =
  \Tr_{\KK^{\otimes n}}(\widetilde\Gamma_t^n),
\end{aligned}
\]
which gives
\[
  \rho_t^{N:n}-\Gamma_t^n
  =
  \Tr_{\KK^{\otimes n}}
  \left(\widetilde\rho_t^{N:n}-\widetilde\Gamma_t^n\right).
\]
Lemma~\ref{lem:partial-trace-contraction} therefore yields
\[
  \left\|\rho_t^{N:n}-\Gamma_t^n\right\|_1
  \le
  \left\|\widetilde\rho_t^{N:n}-\widetilde\Gamma_t^n\right\|_1,
\]
Squaring this pathwise contraction, taking expectations, and using
\eqref{eq:lifted-pure-chaos-rate} proves
\eqref{eq:eta-one-mixed-square}. Cauchy--Schwarz gives
\eqref{eq:eta-one-mixed-rate}.

The constant can be chosen independently of $\gamma_0$. Indeed,
$\dim(\KK)\le d$, the lifted one-particle operator norms are unchanged,
and $\|\widetilde V\|_2=\dim(\KK)\|V\|_2$. The constants in the
pure-state estimate can therefore be bounded in terms of the original
fixed coefficients and $d$.
\end{proof}

\begin{remark}
The exact tensor-product assumption is used because
\[
  \gamma_0^{\otimes N}
  =
  \mathbf T_N\left(\widetilde\gamma_0^{\otimes N}\right),
\]
where $\widetilde\gamma_0^{\otimes N}$ is itself a pure tensor-product
state. A general perturbation of $\gamma_0^{\otimes N}$ need not admit
a compatible pure tensor-product lifting. Strong perturbations of the
initial condition are therefore treated separately.
\end{remark}

\subsection{Fully observed dilation and conditional projection}
\label{subsec:full-dilation}
\label{sec:inefficient}
\label{subsec:eta-dilation-projection}

Throughout this subsection and the next two subsections, let $0<\eta<1$
and Hypothesis~\ref{hyp:perfect-efficiency-tensorization-purity}(3) hold,
namely,
\[
  \rho_0^N=\gamma_0^{\otimes N}.
\]
We aim to show the propagation of chaos using a conditional expectation
argument, as summarized in Proposition~\ref{prop:main-exact-eta-less-one}.

\begin{proposition}
\label{prop:main-exact-eta-less-one}
Let Hypothesis~\ref{hyp:perfect-efficiency-tensorization-purity}(3) hold
with $\gamma_0\in\cS(\HH)$ and $\eta\in(0,1)$.
Then, for every fixed $n\ge1$ and $T>0$, there exists
$C_{\eta,T,n}>0$, independent of $N$ and $\gamma_0$,
such that
\begin{equation}
  \sup_{0\le t\le T}
  \E\left[\|\rho_t^{N:n}-\Gamma_t^n\|_1^2\right]
  \le C_{\eta,T,n}N^{-1/2},
  \qquad N\ge n.
  \label{eq:eta-main-square}
\end{equation}
Consequently, after increasing the constant if necessary,
\begin{equation}
  \sup_{0\le t\le T}
  \E\left[\|\rho_t^{N:n}-\Gamma_t^n\|_1\right]
  \le C_{\eta,T,n}N^{-1/4}.
  \label{eq:eta-main-rate}
\end{equation}
\end{proposition}

The proof uses a fully observed dilation. For every measurement
operator \(L^{(k)}\), introduce the two operators
\begin{equation}
  G^{(k,1)}:=\sqrt{\eta}\,L^{(k)},
  \qquad
  G^{(k,2)}:=\sqrt{1-\eta}\,L^{(k)},
  \qquad
  k=1,\ldots,m.
  \label{eq:eta-dilated-operators}
\end{equation}
Since the coefficients are real and nonnegative, it follows that
\begin{align}
  &\cD_{G^{(k,1)}}+\cD_{G^{(k,2)}}=\cD_{L^{(k)}},
  \label{eq:eta-dilated-D}\\
  &\cE_{G^{(k,1)}}=\sqrt{\eta}\,\cE_{L^{(k)}},
  \qquad
  \cE_{G^{(k,2)}}=\sqrt{1-\eta}\,\cE_{L^{(k)}}.
  \label{eq:eta-dilated-E}
\end{align}
Thus the enlarged equation becomes a perfectly observed equation with
\(2m\) measurement operators.
Proposition~\ref{prop:main-exact-eta-one} applies to that equation.
The inefficient equation is then recovered by conditioning on the first
channel in each pair. The nontrivial point is that conditioning does not
factorize at finite \(N\), and this will be handled by a compensating
Girsanov transformation and a relative entropy estimate.  More precisely, we first apply conditional Jensen's inequality to the
propagation-of-chaos estimate for the fully observed system. We then
use conditional relative entropy to compare the conditional expectation
of its limiting tensor product with $\overline\Gamma_t^{N,n}$.
Finally, a one-particle stability estimate compares
$\overline\Gamma_t^{N,n}$ with $\Gamma_t^n$.
Each step is estimated in squared trace norm.

We use $\mathfrak L_N$, $\mathfrak l_t$, $A^{(k)}$, $b^{(k)}$, and
$b_j^{N,k}$ from Subsection~\ref{subsec:common-canonical}. To retain the
observed and unobserved outputs simultaneously, extend the canonical space to
\begin{equation}
  \Omega_N^{\mathrm{ext}}
  :=
  (\mathsf Y_T\times\mathsf Y_T)^N,
  \qquad
  \Q^{\mathrm{ext}}
  :=
  (\mathbb W_T\otimes\mathbb W_T)^{\otimes N}.
  \label{eq:extended-space}
\end{equation}
The coordinate processes are thus denoted by $(Y^{k,j},U^{k,j})$. Their usual augmented
joint filtration is denoted by $\cF_t^{N,\mathrm{ext}}$.
Let 
$$\operatorname{pr}_Y:\Omega_N^{\mathrm{ext}}\longrightarrow \Omega_N$$
be the
projection onto the observed coordinates. Observation-measurable processes
and events are pulled back by this projection when used on the extended
space, and $\cY_t^N$ there denotes the resulting observed subfiltration.
The Wiener law on this extension is $\Q^{\mathrm{ext}}$, whereas
$\Q^N$ continues to denote the Wiener law on $\Omega_N$.

For $\vartheta\in\cL(\HH^N)$, let
$\widehat R^{N,\vartheta}$ be the adapted process satisfying the Zakai
equation below:
\begin{equation}
\begin{aligned}
  \d\widehat R_t^{N,\vartheta}
  &=
  \mathfrak L_N(\widehat R_t^{N,\vartheta})\,\d t
  +
  \sqrt{\eta}\sum_{k=1}^m\sum_{j=1}^N
  \cC_{L_j^{(k)}}(\widehat R_t^{N,\vartheta})\,\d Y_t^{k,j}+
  \sqrt{1-\eta}\sum_{k=1}^m\sum_{j=1}^N
  \cC_{L_j^{(k)}}(\widehat R_t^{N,\vartheta})\,\d U_t^{k,j}, \\
  \widehat R_0^{N,\vartheta}& =\vartheta.
\end{aligned}
  \label{eq:eta-full-linear-N}
\end{equation}

\begin{lemma}
\label{lem:full-linear-reference}
For every $t\in[0,T]$, equation \eqref{eq:eta-full-linear-N} defines a
random completely positive linear map
\[
  \Psi_t^{N,\eta}:\cL(\HH^N)\longrightarrow\cL(\HH^N)
\]
by $\Psi_t^{N,\eta}(\vartheta):=\widehat R_t^{N,\vartheta}$.
Moreover, for every $B\in\cL(\HH^N)$,
\begin{equation}
  \E_{\Q^{\mathrm{ext}}}\left[\Tr\bigl(\Psi_t^{N,\eta}(B)\bigr)\right]=\Tr(B).
  \label{eq:eta-full-mean-trace}
\end{equation}
For $B\ge0$ and $B\ne0$, $\Tr\bigl(\Psi_t^{N,\eta}(B)\bigr)>0$
for all $t\in[0,T]$.
If $B>0$, then $\Psi_t^{N,\eta}(B)>0$ for all $t\in[0,T]$,
and, for every $n\ge1$,
\[
  \E_{\Q^{\mathrm{ext}}}\left[\sup_{0\le t\le T}
  \left(\|\Psi_t^{N,\eta}(B)\|^n+\|(\Psi_t^{N,\eta}(B))^{-1}\|^n\right)\right]<\infty.
\]
The same conclusions hold if $H^N$ is replaced by a bounded,
deterministic, continuous self-adjoint Hamiltonian $t\longmapsto H_t^N$.
\end{lemma}

\begin{proof}
Apply Lemma~\ref{lem:CP-reference} at full efficiency, with the
measurement operators $G^{(k,1)}$ and $G^{(k,2)}$ defined in
\eqref{eq:eta-dilated-operators} and the independent Wiener coordinates
$(Y^{k,j},U^{k,j})_{k,j}$ under $\Q^{\mathrm{ext}}$.
\end{proof}

We now specialize to
\(
  \vartheta=\gamma_0^{\otimes N}
\)
and omit the superscript $\vartheta$. Set
\begin{equation}
  M_t^N:=\Tr(\widehat R_t^N)
  \qquad\mbox{ and }\qquad
  R_t^N:=\frac{\widehat R_t^N}{M_t^N}.
  \label{eq:eta-MR-def}
\end{equation}
By Lemma~\ref{lem:full-linear-reference}, $\widehat R_t^N\ge0$ and
$M_t^N>0$ for all $t\in[0,T]$, so $R_t^N\in\cS(\HH^N)$. Taking the trace in \eqref{eq:eta-full-linear-N} yields
\[
\begin{aligned}
  \d M_t^N
  &=
  M_t^N\sum_{k=1}^m\sum_{j=1}^N
  \bigg(
    \sqrt{\eta}\,b_j^{N,k}(R_t^N)\,\d Y_t^{k,j}
    +
    \sqrt{1-\eta}\,b_j^{N,k}(R_t^N)\,\d U_t^{k,j}
  \bigg).
\end{aligned}
\]
The functions $b_j^{N,k}$ are uniformly bounded on the state space, so
Novikov's condition holds. Thus $M^N$ is a strictly positive martingale with mean one.  In addition,
\begin{equation}
  \frac{\d\Q^{N,\mathrm{full}}}{\d\Q^{\mathrm{ext}}}
  \bigg|_{\cF_T^{N,\mathrm{ext}}}
  :=
  M_T^N
  \label{eq:eta-full-measure}
\end{equation}
defines a probability measure equivalent to $\Q^{\mathrm{ext}}$.
Under $\Q^{N,\mathrm{full}}$, the processes
\begin{align}
  B_t^{k,j,1}
  &:=
  Y_t^{k,j}
  -
  \sqrt{\eta}\int_0^t b_j^{N,k}(R_s^N)\,\d s,
  \label{eq:eta-B1-full}\\
  B_t^{k,j,2}
  &:=
  U_t^{k,j}
  -
  \sqrt{1-\eta}\int_0^t b_j^{N,k}(R_s^N)\,\d s
  \label{eq:eta-B2-full}
\end{align}
form a family of independent Brownian motions.
Normalizing \eqref{eq:eta-full-linear-N} together with It\^o's formula yields
\begin{equation}
\begin{aligned}
  \d R_t^N
  &=
  \mathfrak L_N(R_t^N)\,\d t
  +
  \sqrt{\eta}\sum_{k=1}^m\sum_{j=1}^N
  \cE_{L_j^{(k)}}(R_t^N)\,\d B_t^{k,j,1}+
  \sqrt{1-\eta}\sum_{k=1}^m\sum_{j=1}^N
  \cE_{L_j^{(k)}}(R_t^N)\,\d B_t^{k,j,2}.
\end{aligned}
  \label{eq:eta-full-nonlinear-N}
\end{equation}
On the same probability space, let \(X^1,\ldots,X^N\) solve
\begin{equation}
\begin{aligned}
  \d X_t^j
  &=
  \mathfrak l_t(X_t^j)\,\d t
  +
  \sqrt{\eta}\sum_{k=1}^m
  \cE_{L^{(k)}}(X_t^j)\,\d B_t^{k,j,1}+
  \sqrt{1-\eta}\sum_{k=1}^m
  \cE_{L^{(k)}}(X_t^j)\,\d B_t^{k,j,2},
  \qquad
  X_0^j=\gamma_0,
\end{aligned}
  \label{eq:eta-full-MF}
\end{equation}
where we note that \(X^1,\ldots,X^N\) are well defined under both
probability measures $\Q^{\mathrm{ext}}$,
$\Q^{N,\mathrm{full}}$.
The processes \(X^1,\ldots,X^N\) are independent and identically
distributed under \(\Q^{N,\mathrm{full}}\).  Taking expectations in \eqref{eq:eta-full-MF}, we find that
$t\longmapsto\E_{\Q^{N,\mathrm{full}}}[X_t^j]$ solves the linear
equation $\d x_t=\mathfrak l_t(x_t) \d t$ with initial state $\gamma_0$.
Since $\xi$ solves the same equation, uniqueness gives
\begin{equation}
  \E_{\Q^{N,\mathrm{full}}}[X_t^j]=\xi_t.
  \label{eq:eta-full-MF-mean}
\end{equation}

\begin{lemma}
\label{lem:eta-full-exchangeability}
Under $\Q^{N,\mathrm{full}}$, the family
\(
  \bigl(R^N,(X^j,Y^j)_{j=1}^N\bigr)
\)
is exchangeable. In particular, the path-space law of the pairs
$(X^j,Y^j)$ is invariant under every permutation of the particle labels.
\end{lemma}

\begin{proof}
Fix $\pi\in\mathfrak S_N$. Conjugating the state by $U_\pi$ and replacing
$(B^{k,j,a})_{k,j,a}$ with $(B^{k,\pi^{-1}(j),a})_{k,j,a}$ leaves the
coupled equations \eqref{eq:eta-full-nonlinear-N} and
\eqref{eq:eta-full-MF} unchanged, because $H^N$ is permutation invariant,
all particles have the same measurement operators, and the initial state
is $\gamma_0^{\otimes N}$. The permuted Brownian family has the same law
as the original one. Uniqueness in law for the coupled system therefore
gives the asserted equality in law. The observation coordinates are
recovered from
\[
  \d Y_t^{k,j}
  =
  \sqrt{\eta}\,b_j^{N,k}(R_t^N)\,\d t
  +
  \d B_t^{k,j,1},
\]
so they are permuted simultaneously with the corresponding particles.
\end{proof}

For $1\le n\le N$,  recall our convention 
$R_t^{N:n}:=\Tr_{[N]\setminus[n]}(R_t^N)$.

\begin{proposition}
\label{prop:eta-full-chaos}
For every fixed \(n\ge1\) and \(T>0\), there exists
\(C_{T,n}^{\mathrm{full}}<\infty\), independent of
\(N\) and $\gamma_0$,  such that
\begin{equation}
  \sup_{0\le t\le T}
  \E_{\Q^{N,\mathrm{full}}}\left[
    \left\|R_t^{N:n}-X_t^1\otimes\cdots\otimes X_t^n\right\|_1^2
  \right]
  \le C_{T,n}^{\mathrm{full}}N^{-1/2},
  \qquad N\ge n.
  \label{eq:eta-full-chaos-square}
\end{equation}
In particular, after increasing $C_{T,n}^{\mathrm{full}}$ if necessary,
\begin{equation}
\begin{aligned}
  \varepsilon_{N,n}^{\mathrm{full}}(T)
  &:=
  \sup_{0\le t\le T}
  \E_{\Q^{N,\mathrm{full}}}
  \left[
    \left\|
      R_t^{N:n}-X_t^1\otimes\cdots\otimes X_t^n
    \right\|_1
  \right]
  \le
  C_{T,n}^{\mathrm{full}}N^{-1/4}.
\end{aligned}
  \label{eq:eta-full-chaos}
\end{equation}
\end{proposition}

\begin{proof}
In view of \eqref{eq:eta-dilated-D}--\eqref{eq:eta-dilated-E},
equation \eqref{eq:eta-full-nonlinear-N} is exactly the perfect-efficiency
Belavkin equation with the \(2m\) measurement operators
\[
  G^{(k,1)},\ G^{(k,2)},
  \qquad k=1,\ldots,m,
\]
and equation \eqref{eq:eta-full-MF} is the corresponding mean-field equation.  The squared estimate \eqref{eq:eta-one-mixed-square} from
Proposition~\ref{prop:main-exact-eta-one}, applied to this finite family
of operators, gives \eqref{eq:eta-full-chaos-square}. Cauchy--Schwarz
then gives \eqref{eq:eta-full-chaos}.
\end{proof}

We now retain only the observed outputs $Y^{k,j}$, using the observed
filtration $\cY_t^N$ already defined in \eqref{eq:eta-YN-def}.
The next proposition identifies the canonical Zakai solution of
Subsection~\ref{subsec:common-canonical} with the conditional projection
of the full linear evolution:
\begin{equation}
  \widehat\rho_t^N
  =
  \E_{\Q^{\mathrm{ext}}}[\widehat R_t^N\mid\cY_t^N].
  \label{eq:eta-hat-rho-conditional}
\end{equation}

\begin{proposition}
\label{prop:eta-conditional-reduction}
First,  the identity~\eqref{eq:eta-hat-rho-conditional} holds. In particular, the
conditional projection $ \widehat\rho_t^N$ satisfies the same linear equation as the
canonical interacting state:  
\begin{equation}
\begin{aligned}
  \d\widehat\rho_t^N
  =
  \mathfrak L_N(\widehat\rho_t^N)\,\d t
  +
  \sqrt{\eta}\sum_{k=1}^m\sum_{j=1}^N
  \cC_{L_j^{(k)}}(\widehat\rho_t^N)\,\d Y_t^{k,j},\qquad 
  \widehat\rho_0^N
  =
  \gamma_0^{\otimes N}.
\end{aligned}
  \label{eq:eta-one-channel-linear}
\end{equation}
Recall
\begin{equation}
  Z_t^N
  =
  \Tr(\widehat\rho_t^N)>0,
  \qquad
  \rho_t^N
  =
  \frac{\widehat\rho_t^N}{Z_t^N}.
  \label{eq:eta-Z-rho-def}
\end{equation}
Then
\begin{equation}
  Z_t^N
  =
  \E_{\Q^{\mathrm{ext}}}
  \left[M_t^N\mid\cY_t^N\right].
  \label{eq:eta-Z-as-conditional-M}
\end{equation}
The physical law on observed paths is the pushforward of the fully
observed law,  i.e., 
\begin{equation}
  (\operatorname{pr}_Y)_\#\Q^{N,\mathrm{full}}=\Q^{N,\eta}.
  \label{eq:eta-Qeta-def}
\end{equation}
Moreover, it holds that
\begin{equation}
  \rho_t^N
  =
  \E_{\Q^{N,\mathrm{full}}}
  \left[R_t^N\mid\cY_t^N\right].
  \label{eq:eta-normalized-conditional}
\end{equation}
The Brownian motions $W^{k,j}$ and the equation for $\rho^N$
are those already constructed in Subsection~\ref{subsec:common-canonical}.
\end{proposition}

\begin{proof}
We first justify the projection of the stochastic integrals. If \(F\)
is a square-integrable
\((\cF_t^{N,\mathrm{ext}})\)-predictable process, then,
coordinatewise in \(\cL(\HH^N)\),
\begin{equation}
  \E_{\Q^{\mathrm{ext}}}
  \left[
    \int_0^t F_s\,\d U_s^{k,j}
    \mathrel{\bigg|}\cY_t^N
  \right]
  =
  0.
  \label{eq:eta-U-projection}
\end{equation}
Indeed, for an elementary predictable process and a bounded
\(\cY_t^N\)-measurable random variable \(H\), each expectation is a sum
of terms of the form
\[
  \E_{\Q^{\mathrm{ext}}}
  \left[
    H F_{t_\ell}
    \bigl(U_{t_{\ell+1}}^{k,j}-U_{t_\ell}^{k,j}\bigr)
  \right],
\]
which vanish because the displayed increment is centered and independent
of the entire \(Y\)-path and of the joint past at time \(t_\ell\).
The general case follows by the It\^o isometry. Notice that the integrand
may depend on both \(Y\) and \(U\).

Similarly,
\begin{equation}
  \E_{\Q^{\mathrm{ext}}}
  \left[
    \int_0^t F_s\,\d Y_s^{k,j}
    \mathrel{\bigg|}\cY_t^N
  \right]
  =
  \int_0^t
  \E_{\Q^{\mathrm{ext}}}
  \left[F_s\mid\cY_s^N\right]\d Y_s^{k,j}.
  \label{eq:eta-Y-projection}
\end{equation}
For \(s\le t\), the future increments of \(Y\) after time \(s\) are
independent of \(\cF_s^{N,\mathrm{ext}}\), so
\[
  \E_{\Q^{\mathrm{ext}}}
  \left[\widehat R_s^N\mid\cY_t^N\right]
  =
  \E_{\Q^{\mathrm{ext}}}
  \left[\widehat R_s^N\mid\cY_s^N\right].
\]
Applying these identities to the integral form of
\eqref{eq:eta-full-linear-N}, and using the linearity of all coefficients,
gives \eqref{eq:eta-one-channel-linear}.  Uniqueness for \eqref{eq:reference-linear} identifies this conditional
projection with $\widehat\rho^N$.

Equation \eqref{eq:eta-Z-as-conditional-M} follows by taking the trace.
For every \(A\in\cY_T^N\),
\[
\begin{aligned}
  \Q^{N,\mathrm{full}}
  \bigl(\operatorname{pr}_Y^{-1}(A)\bigr)=
  \E_{\Q^{\mathrm{ext}}}[\1_A M_T^N]=
  \E_{\Q^{\mathrm{ext}}}
  \left[
    \1_A
    \E_{\Q^{\mathrm{ext}}}
    [M_T^N\mid\cY_T^N]
  \right]=
  \E_{\Q^{\mathrm{ext}}}[\1_A Z_T^N],
\end{aligned}
\]
which proves the pushforward identity in \eqref{eq:eta-Qeta-def}.

Bayes' formula gives
\[
\begin{aligned}
  \E_{\Q^{N,\mathrm{full}}}
  \left[R_t^N\mid\cY_t^N\right]
  =
  \frac{
    \E_{\Q^{\mathrm{ext}}}
    \left[M_T^N R_t^N\mid\cY_t^N\right]
  }{
    \E_{\Q^{\mathrm{ext}}}
    \left[M_T^N\mid\cY_t^N\right]
  }=
  \frac{
    \E_{\Q^{\mathrm{ext}}}
    \left[M_t^N R_t^N\mid\cY_t^N\right]
  }{Z_t^N}
  =
  \rho_t^N,
\end{aligned}
\]
because \(R_t^N\) is
\(\cF_t^{N,\mathrm{ext}}\)-measurable and \(M^N\) is a
\(\Q^{\mathrm{ext}}\)-martingale.
\end{proof}

For \(1\le n\le N\), set
\begin{equation}
  \mathbf X_t^n
  :=
  X_t^1\otimes\cdots\otimes X_t^n,
  \qquad
  \Theta_t^{N,n}
  :=
  \E_{\Q^{N,\mathrm{full}}}
  \left[\mathbf X_t^n\mid\cY_t^N\right],
  \label{eq:eta-product-Theta-def}
\end{equation}
where we recall that \(X^1,\ldots,X^N\) are defined by
\eqref{eq:eta-full-MF}.
Since partial trace commutes with conditional expectation, one has
\[
  \rho_t^{N:n}
  =
  \E_{\Q^{N,\mathrm{full}}}
  \left[R_t^{N:n}\mid\cY_t^N\right].
\]
Conditional Jensen's inequality for the convex function $X\longmapsto\|X\|_1^2$
and the pushforward identity \eqref{eq:eta-Qeta-def} give
\begin{equation}
\begin{aligned}
  &\sup_{0\le t\le T}
  \E_{\Q^{N,\eta}}\left[
    \|\rho_t^{N:n}-\Theta_t^{N,n}\|_1^2
  \right]  \le
  \sup_{0\le t\le T}
  \E_{\Q^{N,\mathrm{full}}}\left[
    \|R_t^{N:n}-\mathbf X_t^n\|_1^2
  \right]
  \le C_{T,n}^{\mathrm{full}}N^{-1/2}.
\end{aligned}
  \label{eq:eta-first-term-bound}
\end{equation}
The last inequality is \eqref{eq:eta-full-chaos-square}. Thus conditioning
preserves the squared error bound. It remains to compare the conditional
expectation  $\Theta_t^{N,n}$ with a product of one-particle filters.

\subsection{Relative entropy and conditional factorization}
\label{subsec:eta-conditional-factorization}

For $0\le t\le T$,  we define
\[
  \mathsf X_t:=C([0,t],\cS(\HH)),
\]
and use the observation-path space $\mathsf Y_t$ from
Subsection~\ref{subsec:common-canonical}.

Let \(\mathbf q_T\) be the law on \(\mathsf X_T\times\mathsf Y_T\) of a
pair \((X,Y)\) satisfying
\begin{equation}
\begin{aligned}
  \d X_t
  &=
  \mathfrak l_t(X_t)\,\d t
  +
  \sqrt{\eta}\sum_{k=1}^m
  \cE_{L^{(k)}}(X_t)\,\d\beta_t^{k,1}
  +
  \sqrt{1-\eta}\sum_{k=1}^m
  \cE_{L^{(k)}}(X_t)\,\d\beta_t^{k,2},
  \qquad
  X_0=\gamma_0,\\
  \d Y_t^k
  &=
  \sqrt{\eta}\,b^{(k)}(X_t)\,\d t
  +
  \d\beta_t^{k,1},
  \qquad
  Y_0^k=0,
\end{aligned}
  \label{eq:eta-ideal-pair}
\end{equation}
where \(\{\beta^{k,a}:1\le k\le m,\ a=1,2\}\) are independent Brownian
motions.  The equation for $X$ is the same as \eqref{eq:eta-full-MF}, but the
observation drift is $\sqrt{\eta}\,b^{(k)}(X_t)$ rather than
$\sqrt{\eta}\,b_j^{N,k}(R_t^N)$ in \eqref{eq:eta-B1-full}.

Let \(\mathbf P_{N,T}\) denote the joint law under
\(\Q^{N,\mathrm{full}}\) of
\(
  (X^1,Y^1),\ldots,(X^N,Y^N)
\)
with $Y^j=(Y^{1,j},\ldots,Y^{m,j})$. Define
\begin{equation}
  \Delta_t^{k,j}
  :=
  b_j^{N,k}(R_t^N)-b^{(k)}(X_t^j)
  \label{eq:eta-Delta-def}
\end{equation}
and
\begin{equation}
  u_t^{k,j,1}
  :=
  \sqrt{\eta}\,\Delta_t^{k,j},
  \qquad
  u_t^{k,j,2}
  :=
  -\frac{\eta}{\sqrt{1-\eta}}\,\Delta_t^{k,j}.
  \label{eq:eta-controls-def}
\end{equation}
As the controls are bounded and predictable, define
$\widetilde\Q^N$ by
\begin{equation}
\begin{aligned}
  \frac{\d\widetilde\Q^N}{\d\Q^{N,\mathrm{full}}}
  \bigg|_{\cF_T^{N,\mathrm{ext}}}
  :=
  \exp\Bigg(
    -
    \sum_{k=1}^m\sum_{j=1}^N\sum_{a=1}^2
    \int_0^T u_s^{k,j,a}\,\d B_s^{k,j,a} -
    \frac12\sum_{k=1}^m\sum_{j=1}^N\sum_{a=1}^2
    \int_0^T|u_s^{k,j,a}|^2\,\d s
  \Bigg).
\end{aligned}
  \label{eq:eta-compensating-density}
\end{equation}
Novikov's condition always  holds. Hence, under \(\widetilde\Q^N\),
the processes
\begin{equation}
  \widetilde B_t^{k,j,a}
  :=
  B_t^{k,j,a}
  +
  \int_0^t u_s^{k,j,a}\,\d s
  \label{eq:eta-Btilde-def}
\end{equation}
form a family of independent Brownian motions, and a straightforward
verification gives
\[
  \d Y_t^{k,j}
  =
  \sqrt{\eta}\,b^{(k)}(X_t^j)\,\d t
  +
  \d\widetilde B_t^{k,j,1}.
\]
If we replace \(B^{k,j,a}\) by \(\widetilde B^{k,j,a}\) in the SDE for
\(X^j\), the additional drift corresponding to the \(k\)-th operator becomes
\[
  -\cE_{L^{(k)}}(X_t^j)
  \left(
    \sqrt{\eta}\,u_t^{k,j,1}
    +
    \sqrt{1-\eta}\,u_t^{k,j,2}
  \right)
  =
  0.
\]
Hence, by pathwise uniqueness,
\begin{equation}
  \Law_{\widetilde\Q^N}
  \bigl((X^1,Y^1),\ldots,(X^N,Y^N)\bigr)
  =
  \mathbf q_T^{\otimes N}.
  \label{eq:eta-product-law-under-tilde}
\end{equation}
For probability measures \(\mu,\nu\), we define the relative entropy by
\[
  \Ent(\mu\mid\nu)
  :=
  \int\log\left(\frac{\d\mu}{\d\nu}\right)\d\mu
\]
when \(\mu\ll\nu\), and set the entropy equal to \(+\infty\) otherwise.
Set
\begin{equation}
  \Lambda_L
  :=
  \sum_{k=1}^m\|A^{(k)}\|^2.
  \label{eq:eta-LambdaL-def}
\end{equation}
Then we have the following control.

\begin{proposition}
\label{prop:eta-entropy}
For every \(T>0\), it holds,  for some $C_{\eta,T}>0$,  that
\begin{equation}
\begin{aligned}
  \frac1N\Ent\left(\mathbf P_{N,T}\mid\mathbf q_T^{\otimes N}\right)
  \le \frac{\eta\Lambda_L}{2(1-\eta)}
  \int_0^T\E_{\Q^{N,\mathrm{full}}}\left[
    \|R_t^{N:(1)}-X_t^1\|_1^2
  \right]\d t \le C_{\eta,T}N^{-1/2}. 
\end{aligned}
  \label{eq:eta-entropy-bound}
\end{equation}
\end{proposition}

\begin{proof}
By data processing and \eqref{eq:eta-product-law-under-tilde},
it holds that
\[
  \Ent\left(\mathbf P_{N,T}\mid\mathbf q_T^{\otimes N}\right)
  \le
  \Ent\left(\Q^{N,\mathrm{full}}\mid\widetilde\Q^N\right),
\]
which gives by Girsanov's entropy identity 
\[
\begin{aligned}
  \Ent\left(\Q^{N,\mathrm{full}}\mid\widetilde\Q^N\right)
  =
  \frac12\sum_{k=1}^m\sum_{j=1}^N
  \E_{\Q^{N,\mathrm{full}}}
  \int_0^T
  \left(|u_t^{k,j,1}|^2+|u_t^{k,j,2}|^2\right)\d t=
  \frac{\eta}{2(1-\eta)}
  \sum_{k=1}^m\sum_{j=1}^N
  \E_{\Q^{N,\mathrm{full}}}
  \int_0^T|\Delta_t^{k,j}|^2\,\d t.
\end{aligned}
\]
By Lemma~\ref{lem:eta-full-exchangeability}, division by \(N\) reduces
the sum over \(j\) to \(j=1\). Moreover,
\[
  |\Delta_t^{k,1}|^2
  \le
  \|A^{(k)}\|^2\left\|R_t^{N:(1)}-X_t^1\right\|_1^2.
\]
Summing over $k$ gives the first inequality in
\eqref{eq:eta-entropy-bound}. The second follows from the squared
estimate \eqref{eq:eta-full-chaos-square} with $n=1$.
\end{proof}

We next recall two standard entropy facts. If \(\mu\) and \(\nu\) are
probability measures on a product \(E\times F\) of Polish spaces,  with $\mu\ll\nu$,  then
\begin{equation}
  \Ent(\mu\mid\nu)
  =
  \Ent(\mu_F\mid\nu_F)
  +
  \int_F
  \Ent\left(\mu^{E\mid F=y}\mid\nu^{E\mid F=y}\right)\mu_F(\d y).
  \label{eq:eta-entropy-chain-rule}
\end{equation}
Furthermore, if \(\mu\) is a probability measure on \(E^N\),
\(\nu_1,\ldots,\nu_N\) are probability measures on \(E\), and
\(I_1,\ldots,I_r\subset[N]\) are pairwise disjoint, then
\begin{equation}
  \sum_{\ell=1}^r
  \Ent\left(
    \mu_{I_\ell}
    \mathrel{\bigg|}
    \bigotimes_{j\in I_\ell}\nu_j
  \right)
  \le
  \Ent\left(
    \mu
    \mathrel{\bigg|}
    \bigotimes_{j=1}^N\nu_j
  \right).
  \label{eq:eta-block-superadditivity}
\end{equation}
The latter follows by applying the entropy chain rule successively
to the blocks and discarding the nonnegative mutual-information terms.

For \(t\le T\), let \(\mathbf P_{N,t}\) and \(\mathbf q_t\) be the
restrictions of \(\mathbf P_{N,T}\) and \(\mathbf q_T\) to \([0,t]\).  By Lemma~\ref{lem:eta-barycenter-identification} below, the $Y$-marginal
of $\mathbf q_t$ is the time-$t$ restriction $\mathbf q_t^Y$ of the
observation law already defined in Subsection~\ref{subsec:common-canonical}.

In the conditional laws below, all paths are restricted to $[0,t]$.
For $1\le n\le N$, write $X^{1:n}:=(X^1,\ldots,X^n)$ and
$Y^{1:n}:=(Y^1,\ldots,Y^n)$, and denote by $\mathbf P_{N,t}^Y$
the law of $Y^{1:N}$ under $\mathbf P_{N,t}$.

For $y\in\mathsf Y_t$, let $\mathbf q_t^{X\mid Y=y}$ be a regular
conditional distribution on $\mathsf X_t$ of the one-particle state
path $X$ given the observation path $Y=y$ under $\mathbf q_t$.
The notation $\mathbf q_t^{X\mid Y=Y^j}$ means that this kernel is
evaluated at the random path $Y^j$.

For $y^{1:N}:=(y^1,\ldots,y^N)\in\mathsf Y_t^N$, let
$\mathbf P_{N,t}^{X^{1:N}\mid Y^{1:N}=y^{1:N}}$ be a regular conditional
distribution on $\mathsf X_t^N$ of $X^{1:N}$ given $Y^{1:N}=y^{1:N}$
under $\mathbf P_{N,t}$. Its first $n$ state-path marginal is denoted
by $\mathbf P_{N,t}^{X^{1:n}\mid Y^{1:N}=y^{1:N}}$.

The chain rule and monotonicity of entropy under
restriction lead to
\begin{equation}
\begin{aligned}
  \E_{\mathbf P_{N,t}^Y}
  \Ent\left(
    \mathbf P_{N,t}^{X^{1:N}\mid Y^{1:N}}
    \mathrel{\bigg|}
    \bigotimes_{j=1}^N\mathbf q_t^{X\mid Y=Y^j}
  \right)
  \le
  \Ent\left(\mathbf P_{N,T}\mid\mathbf q_T^{\otimes N}\right).
\end{aligned}
  \label{eq:eta-total-conditional-entropy}
\end{equation}
Fix \(n\ge1\) and $N\ge n$,  set
\(r_N:=\lfloor N/n\rfloor\), and partition the first
\(nr_N\) coordinates into \(r_N\) consecutive blocks of size \(n\).
Applying \eqref{eq:eta-block-superadditivity} to the conditional law
\(\mathbf P_{N,t}^{X^{1:N}\mid Y^{1:N}=y^{1:N}}\), with reference measure
\(\bigotimes_{j=1}^N\mathbf q_t^{X\mid Y=y^j}\),   and then integrating with respect to $\mathbf P_{N,t}^Y$
gives a sum of $r_N$ block entropies.  The joint law \(\mathbf P_{N,t}\) is invariant under simultaneous
permutations of the pairs \((X^j,Y^j)\). Hence, the expectations of
these block entropies are equal. Therefore,
\begin{equation}
\begin{aligned}
  \E_{\mathbf P_{N,t}^Y}
  \Ent\left(
    \mathbf P_{N,t}^{X^{1:n}\mid Y^{1:N}}
    \mathrel{\bigg|}
    \bigotimes_{j=1}^n\mathbf q_t^{X\mid Y=Y^j}
  \right)
  \le
  \frac1{r_N}
  \Ent\left(\mathbf P_{N,T}\mid\mathbf q_T^{\otimes N}\right).
\end{aligned}
  \label{eq:eta-block-conditional-entropy}
\end{equation}
For $y\in\mathsf Y_t$, define the conditional mean state
\begin{equation}
  \overline\gamma_t(y):=\int_{\mathsf X_t}x_t\,\mathbf q_t^{X\mid Y=y}(\d x)\in\cS(\HH),
  \label{eq:eta-bar-gamma-barycenter}
\end{equation}
where $x=(x_s)_{0\le s\le t}$ is a state path and $x_t$ is its value
at time $t$. Thus $\overline\gamma_t(Y)=\E_{\mathbf q_t}[X_t\mid Y]$.

Lemma~\ref{lem:eta-barycenter-identification} below identifies this
barycenter with the map $\Pi$ already constructed. Consequently, the
product reference of Subsection~\ref{subsec:common-canonical} satisfies
\begin{equation}
  \overline\Gamma_t^{N,n}
  =
  \overline\gamma_t(Y^1)\otimes\cdots\otimes\overline\gamma_t(Y^n).
  \label{eq:eta-bar-Gamma-def}
\end{equation}
The conditional kernels may be chosen arbitrarily on a
\(\mathbf q_t^Y\)-null set. This causes no ambiguity here because
\(\mathbf P_{N,t}^Y\ll(\mathbf q_t^Y)^{\otimes N}\), as follows from
the finite entropy in Proposition~\ref{prop:eta-entropy}.
A canonical version will be selected in
Lemma~\ref{lem:eta-barycenter-identification} below.

\begin{proposition}
\label{prop:eta-conditional-factorization}
For every fixed \(n\ge1\) and \(T>0\), there exists
$C_{\eta,T,n}>0$ so that it holds for $N\ge n$
\begin{equation}
\begin{aligned}
  \sup_{0\le t\le T}\E_{\Q^{N,\eta}}\left[
    \|\Theta_t^{N,n}-\overline\Gamma_t^{N,n}\|_1^2
  \right]
  \le
  \frac{\eta N\Lambda_L}{(1-\eta)r_N}
  \int_0^T\E_{\Q^{N,\mathrm{full}}}\left[
    \|R_s^{N:(1)}-X_s^1\|_1^2
  \right]\d s
  \le C_{\eta,T,n}N^{-1/2}.
\end{aligned}
  \label{eq:eta-conditional-factorization-bound}
\end{equation}
\end{proposition}

\begin{proof}
We denote 
$$d_{\mathrm{TV}}(\mu,\nu):=\sup_A|\mu(A)-\nu(A)|,$$ where the
supremum is over measurable sets. Conditionally on $Y^{1:N}$, trace
duality and Pinsker's inequality give
\[
\begin{aligned}
  \left\|\Theta_t^{N,n}-\overline\Gamma_t^{N,n}\right\|_1^2
  \le
  4d_{\mathrm{TV}}
  \left(
    \mathbf P_{N,t}^{X^{1:n}\mid Y^{1:N}},
    \bigotimes_{j=1}^n\mathbf q_t^{X\mid Y=Y^j}
  \right)^2\le
  2\Ent\left(
    \mathbf P_{N,t}^{X^{1:n}\mid Y^{1:N}}
    \mathrel{\bigg|}
    \bigotimes_{j=1}^n\mathbf q_t^{X\mid Y=Y^j}
  \right).
\end{aligned}
\]
Indeed, the evaluation map
\[
  (x^1,\ldots,x^n)
  \longmapsto
  x_t^1\otimes\cdots\otimes x_t^n
\]
takes values in density matrices and therefore has trace norm one.  Taking expectations and applying
\eqref{eq:eta-block-conditional-entropy}, we obtain
\[
  \E_{\Q^{N,\eta}}\left[
    \|\Theta_t^{N,n}-\overline\Gamma_t^{N,n}\|_1^2
  \right]
  \le \frac{2}{r_N}
  \Ent\left(\mathbf P_{N,T}\mid\mathbf q_T^{\otimes N}\right).
\]
Here the expectation under $\Q^{N,\eta}$ equals the expectation of the
corresponding observation-path function under $\mathbf P_{N,t}^Y$,
by \eqref{eq:eta-Qeta-def}. Proposition~\ref{prop:eta-entropy} now gives
\eqref{eq:eta-conditional-factorization-bound}, since $N/r_N\le2n$
for $N\ge n$.
\end{proof}

Now,  the first two comparisons are controlled in the same squared norm.
Combining \eqref{eq:eta-first-term-bound} and
\eqref{eq:eta-conditional-factorization-bound} with
$\|A+B\|_1^2\le2\|A\|_1^2+2\|B\|_1^2$ gives
\begin{equation}
  \sup_{0\le t\le T}
  \E_{\Q^{N,\eta}}\left[
    \|\rho_t^{N:n}-\overline\Gamma_t^{N,n}\|_1^2
  \right]
  \le C_{\eta,T,n}N^{-1/2}.
  \label{eq:eta-intermediate-chaos}
\end{equation}
It remains to replace the filters evaluated on common observations by
the limiting filters driven by the interacting innovations.

\subsection{Identification of the limiting filter and conclusion}
\label{subsec:eta-return-to-target}
\label{subsec:local-comparison}

We now identify the conditional barycenter $\overline\gamma_t$
defined in \eqref{eq:eta-bar-gamma-barycenter} with the canonical
filter $\Pi_t$ defined in \eqref{eq:eta-Pi-def}.  Its linear equation,
trace likelihood, and observation-form equation have already been
constructed there for every efficiency. The following lemma relates that
construction to the auxiliary pair law used in the preceding subsection.

\begin{lemma}
\label{lem:eta-barycenter-identification}
The $\mathsf Y_T$-marginal of $\mathbf q_T$ is the law
$\mathbf q_T^Y$ defined in Subsection~\ref{subsec:common-canonical}.
In particular, it is equivalent to $\mathbb W_T$, and the version in
\eqref{eq:eta-bar-gamma-barycenter} may be chosen so that
\begin{equation}
  \overline\gamma_t(Y)
  =
  \Pi_t(Y)
  =
  \E_{\mathbf q_T}
  \left[
    X_t\mid\sigma(Y_s:0\le s\le t)
  \right].
  \label{eq:eta-barycenter-identification}
\end{equation}
The observation-form equation for this canonical version is
\eqref{eq:eta-local-observation-SDE}.
\end{lemma}

\begin{proof}
Work on $\Omega_1^{\mathrm{ext}}:=\mathsf Y_T\times\mathsf Y_T$
with Wiener law $\Q^{1,\mathrm{ext}}:=\mathbb W_T\otimes\mathbb W_T$
and coordinate processes $(Y^k,U^k)_{k=1}^m$.
Let $(\cF_t^{1,\mathrm{ext}})_{0\le t\le T}$ be the usual augmentation
under $\Q^{1,\mathrm{ext}}$ of
$\sigma(Y_s^k,U_s^k:0\le s\le t,\ 1\le k\le m)$.
We identify the observed filtration $\cY_t$ on $\mathsf Y_T$
with its pullback under $(y,u)\longmapsto y$.

Under $\Q^{1,\mathrm{ext}}$, let $\widehat R$ solve the linear Zakai
equation with initial condition $\widehat R_0=\gamma_0$:
\[
  \d\widehat R_t
  =\mathfrak l_t(\widehat R_t)\,\d t
  +\sqrt{\eta}\sum_{k=1}^m\cC_{L^{(k)}}(\widehat R_t)\,\d Y_t^k
  +\sqrt{1-\eta}\sum_{k=1}^m\cC_{L^{(k)}}(\widehat R_t)\,\d U_t^k.
\]
Here $\mathfrak l_t$ is defined in \eqref{eq:eta-little-l-def}.  Set $M_t:=\Tr(\widehat R_t)$ and $R_t:=\widehat R_t/M_t$.
Lemma~\ref{lem:full-linear-reference}, applied with Hamiltonian
$h+V^{\xi_t}$, gives $\widehat R_t\ge0$ and $M_t>0$.
The trace equation and the boundedness of $b^{(k)}$ on $\cS(\HH)$
show that $M$ is a mean-one martingale. Define the probability measure
$\overline\Q^{\mathrm{full}}$ by
\[
  \left.\frac{\d\overline\Q^{\mathrm{full}}}
  {\d\Q^{1,\mathrm{ext}}}\right|_{\cF_T^{1,\mathrm{ext}}}:=M_T.
\]
The Girsanov and normalization argument of
Subsection~\ref{subsec:full-dilation} shows that $(R,Y)$ satisfies
\eqref{eq:eta-ideal-pair} under $\overline\Q^{\mathrm{full}}$,
with $X$ replaced by $R$. Uniqueness in law therefore gives
$\Law_{\overline\Q^{\mathrm{full}}}(R,Y)=\mathbf q_T$.

The conditional-projection argument in
Proposition~\ref{prop:eta-conditional-reduction}, now with drift
$\mathfrak l_t$, and uniqueness for \eqref{eq:eta-local-linear-SDE} give
\[
  \E_{\Q^{1,\mathrm{ext}}}[\widehat R_t\mid\cY_t]
  =\widehat{\overline\gamma}_t(Y),
  \qquad
  \E_{\Q^{1,\mathrm{ext}}}[M_t\mid\cY_t]=\overline Z_t(Y).
\]
Consequently, for every bounded Borel function $F$ on $\mathsf Y_T$,
\[
  \E_{\mathbf q_T}[F(Y)]
  =\E_{\Q^{1,\mathrm{ext}}}[F(Y)M_T]
  =\E_{\mathbb W_T}[F(Y)\overline Z_T(Y)].
\]
By \eqref{eq:eta-Pi-def}, this identifies the $Y$-marginal of
$\mathbf q_T$ with $\mathbf q_T^Y$, which is equivalent to
$\mathbb W_T$ since $\overline Z_T>0$.

Since $M$ is a $\Q^{1,\mathrm{ext}}$-martingale and $R_t$ is
$\cF_t^{1,\mathrm{ext}}$-measurable, Bayes' formula yields
\[
  \E_{\overline\Q^{\mathrm{full}}}[R_t\mid\cY_t]
  =\frac{\E_{\Q^{1,\mathrm{ext}}}[M_tR_t\mid\cY_t]}
  {\E_{\Q^{1,\mathrm{ext}}}[M_t\mid\cY_t]}
  =\Pi_t(Y).
\]
Using $\Law_{\overline\Q^{\mathrm{full}}}(R,Y)=\mathbf q_T$,
we obtain
$\E_{\mathbf q_T}[X_t\mid\sigma(Y_s:0\le s\le t)]=\Pi_t(Y)$.
Thus the conditional barycenter in \eqref{eq:eta-bar-gamma-barycenter}
may be chosen as $\overline\gamma_t(\cdot)=\Pi_t(\cdot)$.
Its observation-form equation is \eqref{eq:eta-local-observation-SDE}.
\end{proof}

Return to the $N$-particle observation space
$(\Omega_N,\cY_T^N,\Q^{N,\eta})$. We use
$\overline\gamma_t^{N,j}=\Pi_t(Y^j)$ and
$\overline\Gamma_t^{N,n}$ as defined in
Subsection~\ref{subsec:common-canonical}. The next comparison is valid for
every $0<\eta\le1$ and for arbitrary interacting initial states, since it
only uses the common observation representation and the fixed one-particle
coefficients.

Using \eqref{eq:physical-innovation-family} with $\theta_0^N=\rho_0^N$
in \eqref{eq:eta-local-observation-SDE}, we obtain, under $\Q^{N,\eta}$,
\begin{equation}
\begin{aligned}
  \d\overline\gamma_t^{N,j}
  &=
  \mathfrak l_t(\overline\gamma_t^{N,j})\,\d t
  +
  \sqrt{\eta}\sum_{k=1}^m
  \cE_{L^{(k)}}(\overline\gamma_t^{N,j})\,\d W_t^{k,j}+
  \eta\sum_{k=1}^m
  \cE_{L^{(k)}}(\overline\gamma_t^{N,j})
  \left(
    b_j^{N,k}(\rho_t^N)
    -
    b^{(k)}(\overline\gamma_t^{N,j})
  \right)\d t.
\end{aligned}
  \label{eq:eta-bar-gamma-under-Qeta}
\end{equation}
On this canonical space, the target processes $\gamma^j$ of
Section~\ref{sec:framework} solve, with the same interacting innovations,
\begin{equation}
\begin{aligned}
  \d\gamma_t^j
  &=
  \mathfrak l_t(\gamma_t^j)\,\d t
  +
  \sqrt{\eta}\sum_{k=1}^m
  \cE_{L^{(k)}}(\gamma_t^j)\,\d W_t^{k,j},
  \qquad
  \gamma_0^j=\gamma_0.
\end{aligned}
  \label{eq:eta-MF-target}
\end{equation}
The processes \(\gamma^1,\ldots,\gamma^N\)
are independent and identically distributed under \(\Q^{N,\eta}\). Taking expectations in \eqref{eq:eta-MF-target} shows that
$\E_{\Q^{N,\eta}}[\gamma_t^j]$ and $\xi_t$ solve the same linear
equation $\d x_t=\mathfrak l_t(x_t) \d t$ with initial state $\gamma_0$.
Uniqueness therefore gives
\[
  \E_{\Q^{N,\eta}}[\gamma_t^j]=\xi_t,
\]
so \eqref{eq:eta-MF-target} is precisely the limiting equation
\eqref{eq:MF-Belavkin}.

\begin{lemma}
\label{lem:eta-projected-exchangeability}
Let $0<\eta\le1$ and assume that $\rho_0^N$ is permutation invariant.
Then the family
\[
  \bigl(
    \rho^N,
    (\overline\gamma^{N,j},\gamma^j,Y^j)_{j=1}^N
  \bigr)
\]
is exchangeable under $\Q^{N,\eta}$.
\end{lemma}

\begin{proof}
Fix $\pi\in\mathfrak S_N$ and define
$T_\pi(y^1,\ldots,y^N):=(y^{\pi^{-1}(1)},\ldots,y^{\pi^{-1}(N)})$.
Since $U_\pi\rho_0^NU_\pi^*=\rho_0^N$, conjugating
\eqref{eq:reference-linear} by $U_\pi$ and relabeling its coordinates
gives, by uniqueness,
\[
  \widehat\rho_t^N\circ T_\pi=U_\pi\widehat\rho_t^NU_\pi^*,\qquad Z_t^N\circ T_\pi=Z_t^N.
\]
The measure $\Q^N$ is invariant under $T_\pi$, so
\eqref{eq:physical-likelihood-family} implies that $\Q^{N,\eta}$
is also invariant under $T_\pi$.
Normalization gives
$\rho_t^N\circ T_\pi=U_\pi\rho_t^NU_\pi^*$.
Moreover, the definitions of $\overline\gamma^{N,j}$ and $W^{k,j}$ give
\[
  \overline\gamma_t^{N,j}\circ T_\pi=\overline\gamma_t^{N,\pi^{-1}(j)},\qquad W_t^{k,j}\circ T_\pi=W_t^{k,\pi^{-1}(j)}.
\]
Pathwise uniqueness for \eqref{eq:MF-Belavkin}, with the same initial
state $\gamma_0$ for every particle, then gives
$\gamma_t^j\circ T_\pi=\gamma_t^{\pi^{-1}(j)}$.
These identities prove the asserted exchangeability.
\end{proof}

\begin{lemma}
\label{lem:eta-local-comparison}
For every \(T>0\), there exists \(C_{\eta,T}<\infty\),  independent of $N$, $j$, $\rho_0^N$, and $\gamma_0$,
such that
\begin{equation}
\begin{aligned}
  \E_{\Q^{N,\eta}}
  \left[
    \sup_{0\le t\le T}
    \left\|
      \overline\gamma_t^{N,j}-\gamma_t^j
    \right\|_1^2
  \right] \le
  C_{\eta,T}\int_0^T
  \E_{\Q^{N,\eta}}
  \left[
    \left\|
      \rho_s^{N:(j)}-\overline\gamma_s^{N,j}
    \right\|_1^2
  \right]\d s.
\end{aligned}
  \label{eq:eta-local-comparison}
\end{equation}
\end{lemma}

\begin{proof}
Set
\(
  \Delta\gamma_t^j
  :=
  \overline\gamma_t^{N,j}-\gamma_t^j.
\)
Subtracting \eqref{eq:eta-MF-target} from
\eqref{eq:eta-bar-gamma-under-Qeta} gives a linear Lipschitz drift,
the martingale term
\[
  \sqrt{\eta}\sum_{k=1}^m
  \left(
    \cE_{L^{(k)}}(\overline\gamma_t^{N,j})
    -
    \cE_{L^{(k)}}(\gamma_t^j)
  \right)\d W_t^{k,j},
\]
and the additional drift
\[
  \eta\sum_{k=1}^m
  \cE_{L^{(k)}}(\overline\gamma_t^{N,j})
  \left(
    b_j^{N,k}(\rho_t^N)
    -
    b^{(k)}(\overline\gamma_t^{N,j})
  \right).
\]
On the compact state space \(\cS(\HH)\), every map
\(\cE_{L^{(k)}}\) is bounded and Lipschitz.
The Burkholder--Davis--Gundy inequality in Hilbert--Schmidt norm,
followed by Gronwall's lemma, therefore gives
\[
\begin{aligned}
  &\E_{\Q^{N,\eta}}
  \left[
    \sup_{0\le t\le T}\|\Delta\gamma_t^j\|_2^2
  \right]\le
  C_{\eta,T}\sum_{k=1}^m\int_0^T
  \E_{\Q^{N,\eta}}
  \left|
    b_j^{N,k}(\rho_s^N)
    -
    b^{(k)}(\overline\gamma_s^{N,j})
  \right|^2\d s.
\end{aligned}
\]
For every \(k\), it holds that
\[
\begin{aligned}
  \left|
    b_j^{N,k}(\rho_s^N)
    -
    b^{(k)}(\overline\gamma_s^{N,j})
  \right|
  &\le
  \|A^{(k)}\|
  \left\|
    \rho_s^{N:(j)}-\overline\gamma_s^{N,j}
  \right\|_1.
\end{aligned}
\]
Equivalence between the norms $\|\cdot\|_1$ and $\|\cdot\|_2$ on
\(\cL(\HH)\) proves \eqref{eq:eta-local-comparison}.
\end{proof}

We are now ready to complete the proof of
Proposition~\ref{prop:main-exact-eta-less-one}.

\begin{proof}[Proof of Proposition~\ref{prop:main-exact-eta-less-one}]
By \eqref{eq:eta-intermediate-chaos} with $n=1$ and
Lemma~\ref{lem:eta-projected-exchangeability},
\[
  \int_0^T\E_{\Q^{N,\eta}}\left[
    \|\rho_s^{N:(j)}-\overline\gamma_s^{N,j}\|_1^2
  \right]\d s
  \le C_{\eta,T}N^{-1/2}.
\]
Lemma~\ref{lem:eta-local-comparison} therefore yields
\begin{equation}
  \E_{\Q^{N,\eta}}\left[
    \sup_{0\le t\le T}
    \|\overline\gamma_t^{N,j}-\gamma_t^j\|_1^2
  \right]
  \le C_{\eta,T}N^{-1/2}.
  \label{eq:eta-local-process-rate}
\end{equation}
For the tensor products, telescoping and Cauchy--Schwarz give, pathwise,
\[
  \|\overline\Gamma_t^{N,n}-\Gamma_t^n\|_1^2
  \le
  \left(\sum_{j=1}^n
    \|\overline\gamma_t^{N,j}-\gamma_t^j\|_1\right)^2
  \le n\sum_{j=1}^n
    \|\overline\gamma_t^{N,j}-\gamma_t^j\|_1^2.
\]
Hence
\[
  \E_{\Q^{N,\eta}}\left[
    \sup_{0\le t\le T}
    \|\overline\Gamma_t^{N,n}-\Gamma_t^n\|_1^2
  \right]
  \le C_{\eta,T,n}N^{-1/2}.
\]
Combining this estimate with \eqref{eq:eta-intermediate-chaos} gives
\[
\begin{aligned}
  &\sup_{0\le t\le T}\E_{\Q^{N,\eta}}\left[
    \|\rho_t^{N:n}-\Gamma_t^n\|_1^2
  \right] \le
  2\sup_{0\le t\le T}\E_{\Q^{N,\eta}}\left[
    \|\rho_t^{N:n}-\overline\Gamma_t^{N,n}\|_1^2
  \right]
  +2\E_{\Q^{N,\eta}}\left[
    \sup_{0\le t\le T}
    \|\overline\Gamma_t^{N,n}-\Gamma_t^n\|_1^2
  \right]\le C_{\eta,T,n}N^{-1/2}.
\end{aligned}
\]
This is \eqref{eq:eta-main-square} in the canonical realization.  The coupled family $(\rho^N,\gamma^1,\ldots,\gamma^n)$ is a weak
realization of \eqref{eq:N-Belavkin} and \eqref{eq:MF-Belavkin} driven
by the same innovations. Uniqueness in law identifies its joint law with
that in Definition~\ref{def:propagation-of-chaos}, proving
\eqref{eq:eta-main-square} and \eqref{eq:eta-main-rate} on the original
probability space.
\end{proof}

\begin{remark}
{\rm (i)} The second, unobserved channel has two roles. It restores
perfect total observation, so that the result of
Subsection~\ref{sec:purification} applies, and it supplies
the compensating Brownian drift in \eqref{eq:eta-controls-def}.
The latter changes the drift of the observed record from the interacting
quantity \(b_j^{N,k}(R^N)\) to the one-particle quantity \(b^{(k)}(X^j)\)
without changing the SDE for \(X^j\).

\medskip

\noindent
{\rm (ii)} However, this analysis yields an entropy constant containing
the factor \((1-\eta)^{-1}\), and the resulting estimates are not uniform as
\(\eta\uparrow1\). The endpoint \(\eta=1\) must be treated directly by
purification in Subsection~\ref{sec:purification}.
\end{remark}

Propositions~\ref{prop:main-exact-eta-one} and
\ref{prop:main-exact-eta-less-one} prove
\eqref{eq:exact-product-square} and \eqref{eq:exact-product-goal}
under Hypothesis~\ref{hyp:perfect-efficiency-tensorization-purity}(3),
for every fixed $0<\eta\le1$. The exponent $1/4$ is the same for
perfect and inefficient observation, while the inefficient-observation
argument retains constants depending on $(1-\eta)^{-1}$.
The passage to approximately tensorized initial states is carried out
in Subsection~\ref{subsec:main-proofs}.

\section{Propagation of chaos with a faithful reference}
\label{sec:entropy}

Fix $0<\eta\le1$ and $T>0$. Except in
Subsection~\ref{subsec:local-entropy}, we work under
Hypothesis~\ref{hyp:perfect-efficiency-tensorization-purity}(4):
$\rho_0^N$ is permutation invariant and $\gamma_0>0$.
This section aims to control the propagation-of-chaos error
by the initial relative entropy per particle. The argument uses the
canonical construction of Subsection~\ref{subsec:common-canonical} and
Lemma~\ref{lem:eta-local-comparison}.

For density operators $\rho$ and $\sigma$ on the same finite-dimensional
Hilbert space, define
\[
  D(\rho\mid\sigma):=\Tr\bigl(\rho(\log(\rho)-\log(\sigma))\bigr),
  \qquad S(\rho):=-\Tr\bigl(\rho\log(\rho)\bigr),
\]
where we use $0\log(0):=0$ and natural logarithms. The relative entropy is
$+\infty$ unless $\operatorname{supp}(\rho)\subset\operatorname{supp}(\sigma)$;
when this inclusion holds, the logarithms are evaluated on their supports.

\begin{proposition}
\label{prop:faithful-entropy}
For every fixed $n\ge1$, there is a constant $C_{T,n,\gamma_0}$,
independent of $N\ge n$ and of the permutation-invariant state $\rho_0^N$,
such that
\begin{equation}
  \sup_{0\le t\le T}\E_{\Q^{N,\eta}}\left[\|\rho_t^{N:n}-\Gamma_t^n\|_1\right]
  \le C_{T,n,\gamma_0}
  \left(\frac{D(\rho_0^N\mid\gamma_0^{\otimes N})+1}{N}\right)^{1/2}.
  \label{eq:faithful-proposition}
\end{equation}
The constant may depend on $d,h,V,L^{(1)},\ldots,L^{(m)},\eta,T,n$
and the smallest eigenvalue of $\gamma_0$.
\end{proposition}

To prove Proposition~\ref{prop:faithful-entropy}, we first control
$\E_{\Q^{N,\eta}}[D(\rho_t^N\mid\overline\Gamma_t^{N,N})]$
in terms of $D(\rho_0^N\mid\gamma_0^{\otimes N})$.
Subsection~\ref{subsec:entropy-roadmap} introduces
$\recordentropy_N(t)$ in \eqref{eq:joint-entropy-definition} and proves that
\[
  \E_{\Q^{N,\eta}}[D(\rho_t^N\mid\overline\Gamma_t^{N,N})]
  \le\recordentropy_N(t),
  \qquad \mbox{with }
  \recordentropy_N(0)=D(\rho_0^N\mid\gamma_0^{\otimes N}).
\]
Subsection~\ref{subsec:entropy-production} estimates the increase of
$\recordentropy_N(t)$, and
Subsection~\ref{subsec:entropy-exponential} establishes the estimates
needed to bound the integral term in \eqref{eq:entropy-production}.
In Subsection~\ref{subsec:faithful-completion}, we combine these estimates
with Gronwall's inequality to obtain
\[
  \sup_{0\le t\le T}\recordentropy_N(t)
  \le C_{T,\gamma_0}\bigl(D(\rho_0^N\mid\gamma_0^{\otimes N})+1\bigr),
\]
where $C_{T,\gamma_0}$ is independent of $N$.
Permutation invariance and Pinsker's inequality then give a bound for
$$\E_{\Q^{N,\eta}}[\|\rho_t^{N:n}-\overline\Gamma_t^{N,n}\|_1^2].$$
Lemma~\ref{lem:eta-local-comparison} controls
$\overline\gamma^{N,j}-\gamma^j$ and hence
$\overline\Gamma^{N,n}-\Gamma^n$, completing the proof.

Subsection~\ref{subsec:local-entropy} gives a local version of the
entropy estimate with explicit dependence on a lower bound for the
eigenvalues of the one-particle matrix at the starting time.
This estimate is used in Section~\ref{sec:extensions} to treat
possibly singular $\gamma_0$.

\subsection{Relative entropy on the observation space}
\label{subsec:entropy-roadmap}

All $N$-particle processes in this section are defined on
$(\Omega_N,\cY_T^N)=(\mathsf Y_T^N,\cY_T^N)$ from
Subsection~\ref{subsec:common-canonical}. Recall that
$\Q^N=\mathbb W_T^{\otimes N}$ is the Wiener measure on this space,
$\Q^{N,\eta}$ is the interacting observation law, and
$\overline\Q^N=(\mathbf q_T^Y)^{\otimes N}$ is the product reference law.
For $0\le t\le T$, their densities satisfy
\[
  \left.\frac{\d\Q^{N,\eta}}{\d\Q^N}\right|_{\cY_t^N}=Z_t^N,
  \qquad
  \left.\frac{\d\overline\Q^N}{\d\Q^N}\right|_{\cY_t^N}=\overline Z_t^N.
\]
Thus, for any integrable $\cY_t^N$-measurable random variable $F$,
\begin{equation}
  \E_{\Q^{N,\eta}}[F]=\E_{\Q^N}[Z_t^N F],
  \qquad
  \E_{\overline\Q^N}[F]=\E_{\Q^N}[\overline Z_t^N F].
  \label{eq:entropy-change-of-measure}
\end{equation}
The matrices to be compared are the solutions  constructed in
Subsection~\ref{subsec:common-canonical}:
\[
  \widehat\rho_t^N=Z_t^N\rho_t^N,
  \qquad
  \widehat{\overline\Gamma}_t^N=\overline Z_t^N\overline\Gamma_t^{N,N},
  \qquad
  \overline\Gamma_t^{N,N}=\Pi_t(Y^1)\otimes\cdots\otimes\Pi_t(Y^N).
\]
Under $\overline\Q^N$, the factors $\Pi_t(Y^j)$ are independent and have
mean $\xi_t$ by \eqref{eq:reference-Pi-mean}. Under $\Q^{N,\eta}$ they
are evaluated along the interacting observations and need not be
independent. In both cases, they are positive definite by
Lemma~\ref{lem:CP-reference} and equivalence of these measures with $\Q^N$.

For positive matrices on the same Hilbert space, define
\begin{equation}
  \grel(X\mid Z):=\Tr\bigl(X(\log(X)-\log(Z))-X+Z\bigr),
  \label{eq:generalized-entropy}
\end{equation}
with the same support convention as for $D$. Klein's inequality gives
$\grel\ge0$.  Set
\begin{equation}
  \recordentropy_N(t):=\E_{\Q^N}
  \left[\grel(\widehat\rho_t^N\mid\widehat{\overline\Gamma}_t^N)\right].
  \label{eq:joint-entropy-definition}
\end{equation}
Separating the scalar factors $Z_t^N$ and $\overline Z_t^N$ inside the
logarithms, and using \eqref{eq:entropy-change-of-measure}, gives
\begin{equation}
  \recordentropy_N(t)
  =\Ent\left(\left.\Q^{N,\eta}\right|_{\cY_t^N}
       \,\middle|\,\left.\overline\Q^N\right|_{\cY_t^N}\right)
   +\E_{\Q^{N,\eta}}\left[D(\rho_t^N\mid\overline\Gamma_t^{N,N})\right],
  \label{eq:joint-entropy-chain-rule}
\end{equation}
where the classical entropy
$\Ent$ is the one defined in Subsection~\ref{subsec:eta-conditional-factorization}.
In particular,
$\recordentropy_N(0)=D(\rho_0^N\mid\gamma_0^{\otimes N})$.
The second term in \eqref{eq:joint-entropy-chain-rule} compares the
states conditional on the entire observation filtration $\cY_t^N$.

For each fixed $N$, we have
\[
  \E_{\Q^N}\left[
    \sup_{0\le t\le T}
    \grel(\widehat\rho_t^N\mid\widehat{\overline\Gamma}_t^N)
  \right]<\infty.
\]
Indeed, all eigenvalues of $\rho_0^N$ belong to $[0,1]$, so
$0\le\rho_0^N\le I_{\HH^N}$. By linearity and positivity of
$\Phi_t^{N,\eta}$,
$\Phi_t^{N,\eta}(I_{\HH^N})-\widehat\rho_t^N
=\Phi_t^{N,\eta}(I_{\HH^N}-\rho_0^N)\ge0$.
Consequently,
$\|\widehat\rho_t^N\|\le\|\Phi_t^{N,\eta}(I_{\HH^N})\|$,
even when $\rho_0^N$ is singular.

For $X\ge0$ and $Z>0$ on $\HH^N$, the inequality
$x\log(x)\le x^2$ gives
$\Tr(X\log(X))\le\Tr(X^2)$, while
$|\Tr(X\log(Z))|\le\Tr(X)\|\log(Z)\|$.
Dropping the nonpositive term $-\Tr(X)$ in
\eqref{eq:generalized-entropy} and using $\dim(\HH^N)=d^N$, we obtain
\[
  0\le\grel(X\mid Z)
  \le d^N\left(\|X\|^2+\|X\|\|\log(Z)\|+\|Z\|\right).
\]
The spectral decomposition of $Z$ and
$|\log(x)|\le x+x^{-1}$ imply
$\|\log(Z)\|\le\|Z\|+\|Z^{-1}\|$.
Applying $2ab\le a^2+b^2$ to the two resulting products, and
$2\|Z\|\le1+\|Z\|^2$, therefore yields
\[
  0\le\grel(X\mid Z)
  \le 2d^N\left(1+\|X\|^2+\|Z\|^2+\|Z^{-1}\|^2\right).
\]
The claimed expectation bound follows from
Lemma~\ref{lem:CP-reference}, applied with initial condition
$I_{\HH^N}$ and to the one-particle equation
\eqref{eq:eta-local-linear-SDE} with initial condition $\gamma_0>0$,
together with \eqref{eq:product-reference-definitions} and
$\Q^N=\mathbb W_T^{\otimes N}$.

With this notation in place, we now turn to the proof of
Proposition~\ref{prop:faithful-entropy}. We begin by estimating
$\recordentropy_N(t)-\recordentropy_N(s)$ for $0\le s\le t\le T$.

\subsection{Evolution of the relative entropy}
\label{subsec:entropy-production}

We next estimate $\recordentropy_N(t)-\recordentropy_N(s)$.
The estimate will involve the difference between $H^N$ and
$\sum_{j=1}^N(h+V^{\xi_t})_j$. Set
\begin{equation}
  \overline H_t^N:=\sum_{j=1}^N(h+V^{\xi_t})_j,
  \qquad \Delta H_t^N:=H^N-\overline H_t^N,
  \qquad K_N(t):=\i[\log(\overline\Gamma_t^{N,N}),\Delta H_t^N].
  \label{eq:H-defect}
\end{equation}
Since $\overline\Gamma_t^{N,N}>0$, the matrix $K_N(t)$ is well defined
and belongs to $\cL^*(\HH^N)$.

\begin{proposition}
\label{prop:H-production}
For $0\le s\le t\le T$,
\begin{equation}
  \recordentropy_N(t)\le\recordentropy_N(s)+\int_s^t\E_{\Q^{N,\eta}}[\Tr(\rho_u^N K_N(u))]\,\d u.
  \label{eq:entropy-production}
\end{equation}
This assertion does not require permutation invariance of $\rho_0^N$.
\end{proposition}

\begin{proof}
We proceed as follows.  First,  
we assume $\rho_0^N>0$. All stochastic equations below
are considered under $\Q^N$, and write for simplicity 
$\sum_{k,j}:=\sum_{k=1}^m\sum_{j=1}^N$.
Define, for $X\in\cL(\HH^N)$,
\[
  \overline{\mathfrak L}_{N,u}(X):=-\i[\overline H_u^N,X]+\sum_{k,j}\cD_{L_j^{(k)}}(X).
\]
Then
$\mathfrak L_N(X)=\overline{\mathfrak L}_{N,u}(X)-\i[\Delta H_u^N,X]$.
Equation~\eqref{eq:reference-linear} therefore gives
\[
  \d\widehat\rho_u^N=\bigl(\overline{\mathfrak L}_{N,u}(\widehat\rho_u^N)-\i[\Delta H_u^N,\widehat\rho_u^N]\bigr)\,\d u+\sqrt\eta\sum_{k,j}\cC_{L_j^{(k)}}(\widehat\rho_u^N)\,\d Y_u^{k,j}.
\]
Moreover, applying the product rule to
\eqref{eq:product-reference-definitions}, with
$\d[Y^{k,i},Y^{\ell,j}]_u=\delta_{k\ell}\delta_{ij}\,\d u$, yields
\[
  \d\widehat{\overline\Gamma}_u^N=\overline{\mathfrak L}_{N,u}(\widehat{\overline\Gamma}_u^N)\,\d u+\sqrt\eta\sum_{k,j}\cC_{L_j^{(k)}}(\widehat{\overline\Gamma}_u^N)\,\d Y_u^{k,j}.
\]
Put $f(X,Z):=\grel(X\mid Z)$ for $X,Z>0$.
This is a smooth real-valued function on the open set of positive
definite pairs in
$\cL^*(\HH^N)\times\cL^*(\HH^N)$.
For self-adjoint matrices $U,V,U',V'$, we use the derivatives
\[
  \mathrm D_Xf(X,Z)[U]:=\left.\frac{\d}{\d\varepsilon}f(X+\varepsilon U,Z)\right|_{\varepsilon=0},\qquad
  \mathrm D_Zf(X,Z)[V]:=\left.\frac{\d}{\d\varepsilon}f(X,Z+\varepsilon V)\right|_{\varepsilon=0}.
\]
The derivative with respect to both arguments is
\[
  \mathrm Df(X,Z)[(U,V)]:=\left.\frac{\d}{\d\varepsilon}f(X+\varepsilon U,Z+\varepsilon V)\right|_{\varepsilon=0}=\mathrm D_Xf(X,Z)[U]+\mathrm D_Zf(X,Z)[V].
\]
Its second derivative is the bilinear form
\[
  \mathrm D^2f(X,Z)[(U,V),(U',V')]:=\left.\frac{\partial^2}{\partial a\,\partial b}f(X+aU+bU',Z+aV+bV')\right|_{a=b=0}.
\]
We equip the space of pairs with
$\|(U,V)\|:=\|U\|+\|V\|$ and use the induced operator norms
for $\mathrm Df$ and $\mathrm D^2f$.

For $1\le k\le m$ and $1\le j\le N$, set
\[
  \mathbf C_{k,j}(X,Z):=\bigl(\cC_{L_j^{(k)}}(X),\cC_{L_j^{(k)}}(Z)\bigr).
\]
Define
\begin{equation}
  \mathcal A_u^{\rm c}f(X,Z):=\mathrm Df(X,Z)[(\overline{\mathfrak L}_{N,u}(X),\overline{\mathfrak L}_{N,u}(Z))]+\frac\eta2\sum_{k,j}\mathrm D^2f(X,Z)[\mathbf C_{k,j}(X,Z),\mathbf C_{k,j}(X,Z)].
  \label{eq:entropy-common-generator}
\end{equation}
Our proof consists of first proving the identity
\begin{equation}
  \recordentropy_N(t)-\recordentropy_N(s)=\int_s^t\E_{\Q^N}\!\left[\mathcal A_u^{\rm c}f(\widehat\rho_u^N,\widehat{\overline\Gamma}_u^N)+\Tr(\widehat\rho_u^N K_N(u))\right]\,\d u,
  \label{eq:entropy-ito-identity}
\end{equation}
and then showing 
$\mathcal A_u^{\rm c}f(X,Z)\le0$ for $X,Z>0$.

For the function $f$, differentiation in the first argument gives
\[
  \mathrm D_Xf(X,Z)[U]=\Tr\bigl(U(\log(X)-\log(Z))\bigr).
\]
Since $X$ commutes with $\log(X)$,
$\Tr([\Delta H_u^N,X]\log(X))=0$.
Cyclicity of the trace therefore implies
\[
  \mathrm D_Xf(X,Z)[-\i[\Delta H_u^N,X]]=\i\Tr([\Delta H_u^N,X]\log(Z))=\Tr\bigl(X\,\i[\log(Z),\Delta H_u^N]\bigr).
\]
By
$\log(\widehat{\overline\Gamma}_u^N)
=\log(\overline Z_u^N)I_{\HH^N}+\log(\overline\Gamma_u^{N,N})$,
we have
$K_N(u)=\i[\log(\widehat{\overline\Gamma}_u^N),\Delta H_u^N]$.
Consequently,
\[
  \mathrm D_Xf(\widehat\rho_u^N,\widehat{\overline\Gamma}_u^N)[-\i[\Delta H_u^N,\widehat\rho_u^N]]=\Tr(\widehat\rho_u^N K_N(u)).
\]
To take expectations in It\^o's formula, we use the following bounds.
Applying Lemma~\ref{lem:CP-reference} with initial conditions
$\rho_0^N$ and $\gamma_0^{\otimes N}$ and Hamiltonians
$H^N$ and $\overline H^N$, respectively, gives
\[
  \E_{\Q^N}\!\left[\sup_{0\le u\le T}\left(\|\widehat\rho_u^N\|^8+\|(\widehat\rho_u^N)^{-1}\|^8+\|\widehat{\overline\Gamma}_u^N\|^8+\|(\widehat{\overline\Gamma}_u^N)^{-1}\|^8\right)\right]<\infty.
\]
The derivative formulas in Appendix~\ref{app:generator} imply
\[
  |f(X,Z)|+\|\mathrm Df(X,Z)\|+\|\mathrm D^2f(X,Z)\|\le C_N\bigl(1+\|X\|+\|Z\|+\|X^{-1}\|+\|Z^{-1}\|\bigr)^3.
\]
Since the coefficients in the two SDEs are linear maps whose
operator norms are bounded on $[0,T]$, these estimates yield
\[
  \E_{\Q^N}\int_0^T\left(|\mathcal A_u^{\rm c}f(\widehat\rho_u^N,\widehat{\overline\Gamma}_u^N)|+|\Tr(\widehat\rho_u^N K_N(u))|\right)\,\d u +
  \eta\sum_{k,j}\E_{\Q^N}\int_0^T\left|\mathrm Df(\widehat\rho_u^N,\widehat{\overline\Gamma}_u^N)[\mathbf C_{k,j}(\widehat\rho_u^N,\widehat{\overline\Gamma}_u^N)]\right|^2\,\d u<\infty.
\]
Thus the stochastic integral in It\^o's formula is a square-integrable
martingale and has expectation zero. Applying It\^o's formula on
$[s,t]$, taking expectations, and using
\eqref{eq:joint-entropy-definition} proves
\eqref{eq:entropy-ito-identity}.

We next prove
\begin{equation}
  \mathcal A_u^{\rm c}f(X,Z)\le0,\qquad X,Z>0,\quad 0\le u\le T.
  \label{eq:generator-sign}
\end{equation}
Fix $\tau\in[0,T]$. For $B\in\cL(\HH^N)$, let
$\Phi_r^{\rm c}(B)$ solve the following equation under $\Q^N$,
with initial condition $\Phi_0^{\rm c}(B)=B$:
\[
  \d\Phi_r^{\rm c}(B)=\overline{\mathfrak L}_{N,\tau}(\Phi_r^{\rm c}(B))\,\d r+\sqrt\eta\sum_{k,j}\cC_{L_j^{(k)}}(\Phi_r^{\rm c}(B))\,\d Y_r^{k,j}.
\]
Here $\tau$ is fixed and $0\le r\le T$.
Lemma~\ref{lem:CP-reference} shows that $\Phi_r^{\rm c}$ is completely
positive and
$\E_{\Q^N}[\Tr(\Phi_r^{\rm c}(B))]=\Tr(B)$.
Lemma~\ref{lem:instrument-DPI} therefore gives
\[
  \E_{\Q^N}[f(\Phi_r^{\rm c}(X),\Phi_r^{\rm c}(Z))]\le f(X,Z),\qquad X,Z>0.
\]
The preceding moment and derivative estimates also apply to
$(\Phi_r^{\rm c}(X),\Phi_r^{\rm c}(Z))$.
Hence, for $0<\varepsilon\le T$, It\^o's formula and the last inequality give
\[
  \int_0^\varepsilon\E_{\Q^N}[\mathcal A_\tau^{\rm c}f(\Phi_r^{\rm c}(X),\Phi_r^{\rm c}(Z))]\,\d r=\E_{\Q^N}[f(\Phi_\varepsilon^{\rm c}(X),\Phi_\varepsilon^{\rm c}(Z))]-f(X,Z)\le0.
\]
Moreover,
\[
  \E_{\Q^N}\!\left[\sup_{0\le r\le T}|\mathcal A_\tau^{\rm c}f(\Phi_r^{\rm c}(X),\Phi_r^{\rm c}(Z))|\right]<\infty.
\]
Since $\Phi_r^{\rm c}(X)\to X$ and $\Phi_r^{\rm c}(Z)\to Z$ as
$r\downarrow0$, continuity and dominated convergence imply
\[
  \mathcal A_\tau^{\rm c}f(X,Z)=\lim_{\varepsilon\downarrow0}\frac1\varepsilon\int_0^\varepsilon\E_{\Q^N}[\mathcal A_\tau^{\rm c}f(\Phi_r^{\rm c}(X),\Phi_r^{\rm c}(Z))]\,\d r\le0.
\]
This proves \eqref{eq:generator-sign}. Substituting it into
\eqref{eq:entropy-ito-identity}, we obtain,  for $\rho_0^N>0$,
\[
  \recordentropy_N(t)-\recordentropy_N(s)\le\int_s^t\E_{\Q^N}[\Tr(\widehat\rho_u^N K_N(u))]\,\d u.
\]
It remains to remove the assumption $\rho_0^N>0$.
For $0<\delta<1$, set
\[
  \rho_0^{N,\delta}:=(1-\delta)\rho_0^N+\delta I_{\HH^N}/d^N,\qquad \widehat\rho_u^{N,\delta}:=\Phi_u^{N,\eta}(\rho_0^{N,\delta}).
\]
We keep $\gamma_0$ fixed, so
$\widehat{\overline\Gamma}^N$ and $K_N$ do not depend on $\delta$.
Define
\[
  \recordentropy_N^\delta(u):=\E_{\Q^N}[\grel(\widehat\rho_u^{N,\delta}\mid\widehat{\overline\Gamma}_u^N)].
\]
Since $\rho_0^{N,\delta}>0$, the inequality already proved gives
\[
  \recordentropy_N^\delta(t)-\recordentropy_N^\delta(s)\le\int_s^t\E_{\Q^N}[\Tr(\widehat\rho_u^{N,\delta}K_N(u))]\,\d u.
\]
Linearity and positivity of $\Phi_u^{N,\eta}$ imply
\[
  \widehat\rho_u^{N,\delta}=(1-\delta)\widehat\rho_u^N+\delta\Phi_u^{N,\eta}(I_{\HH^N}/d^N),\qquad 0\le\widehat\rho_u^{N,\delta}\le\Phi_u^{N,\eta}(I_{\HH^N}).
\]
The bound proved in Subsection~\ref{subsec:entropy-roadmap} therefore yields
\[
  0\le\grel(\widehat\rho_u^{N,\delta}\mid\widehat{\overline\Gamma}_u^N)\le2d^N\left(1+\|\Phi_u^{N,\eta}(I_{\HH^N})\|^2+\|\widehat{\overline\Gamma}_u^N\|^2+\|(\widehat{\overline\Gamma}_u^N)^{-1}\|^2\right).
\]
By Lemma~\ref{lem:CP-reference},
\[
  \E_{\Q^N}\!\left[\sup_{0\le u\le T}\left(\|\Phi_u^{N,\eta}(I_{\HH^N})\|^2+\|\widehat{\overline\Gamma}_u^N\|^2+\|(\widehat{\overline\Gamma}_u^N)^{-1}\|^2\right)\right]<\infty.
\]
For fixed $Z>0$, the function $X\mapsto\grel(X\mid Z)$ is continuous
on $X\ge0$, with the convention $0\log(0)=0$.
Dominated convergence consequently gives
\[
  \recordentropy_N^\delta(u)\longrightarrow\recordentropy_N(u),\qquad 0\le u\le T.
\]
In particular, this convergence holds at both $s$ and $t$.

For the integral term, use
$K_N(u)=\i[\log(\widehat{\overline\Gamma}_u^N),\Delta H_u^N]$
and $\|\log(Z)\|\le\|Z\|+\|Z^{-1}\|$ to obtain
\[
  |\Tr(\widehat\rho_u^{N,\delta}K_N(u))|\le2d^N\|\Delta H_u^N\|\,\|\Phi_u^{N,\eta}(I_{\HH^N})\|\left(\|\widehat{\overline\Gamma}_u^N\|+\|(\widehat{\overline\Gamma}_u^N)^{-1}\|\right).
\]
Since $\sup_{0\le u\le T}\|\Delta H_u^N\|<\infty$,
Cauchy--Schwarz and the preceding second-moment bounds imply
\[
  \E_{\Q^N}\int_0^T\sup_{0<\delta<1}|\Tr(\widehat\rho_u^{N,\delta}K_N(u))|\,\d u<\infty.
\]
We may therefore let $\delta\downarrow0$ in the inequality for
$\recordentropy_N^\delta$, obtaining
\[
  \recordentropy_N(t)-\recordentropy_N(s)\le\int_s^t\E_{\Q^N}[\Tr(\widehat\rho_u^N K_N(u))]\,\d u.
\]
Finally, $\widehat\rho_u^N=Z_u^N\rho_u^N$ and
\eqref{eq:entropy-change-of-measure} give
\[
  \E_{\Q^N}[\Tr(\widehat\rho_u^N K_N(u))]=\E_{\Q^{N,\eta}}[\Tr(\rho_u^N K_N(u))].
\]
This proves \eqref{eq:entropy-production} for every
$\rho_0^N\in\cS(\HH^N)$.
\end{proof}

\subsection{Estimates for the growth of relative entropy}
\label{subsec:entropy-exponential}

We estimate the last term of \eqref{eq:entropy-production} through
the quantity
\begin{equation}
  \mathcal Z_N(\lambda,t):=\E_{\overline\Q^N}\left[\Tr\left(\exp\bigl(\log(\overline\Gamma_t^{N,N})+\lambda K_N(t)\bigr)\right)\right],\qquad \lambda>0.
  \label{eq:partition-function}
\end{equation}
For every $\lambda>0$ such that $\mathcal Z_N(\lambda,t)<\infty$,
we have
\begin{equation}
  \E_{\Q^{N,\eta}}[\Tr(\rho_t^N K_N(t))]\le\frac{\recordentropy_N(t)+\log(\mathcal Z_N(\lambda,t))}{\lambda}.
  \label{eq:entropy-exponential-reduction}
\end{equation}
Indeed, Lemma~\ref{lem:cq-gibbs} applies on $(\Omega_N,\Q^N)$ to
$\widehat\rho_t^N$ and $\widehat{\overline\Gamma}_t^N$.
The second matrix is positive definite, and
\[
  \E_{\Q^N}[\Tr(\widehat\rho_t^N)]=\E_{\Q^N}[\Tr(\widehat{\overline\Gamma}_t^N)]=1.
\]
The estimates in the proof of Proposition~\ref{prop:H-production}
also give
$\E_{\Q^N}[|\Tr(\widehat\rho_t^N K_N(t))|]<\infty$.
Moreover,
\[
  \exp\bigl(\log(\widehat{\overline\Gamma}_t^N)+\lambda K_N(t)\bigr)=\overline Z_t^N\exp\bigl(\log(\overline\Gamma_t^{N,N})+\lambda K_N(t)\bigr).
\]
Taking traces and using \eqref{eq:entropy-change-of-measure}, we obtain
\[
  \E_{\Q^N}\left[\Tr\left(\exp\bigl(\log(\widehat{\overline\Gamma}_t^N)+\lambda K_N(t)\bigr)\right)\right]=\mathcal Z_N(\lambda,t).
\]
Lemma~\ref{lem:cq-gibbs} therefore yields \(
  \lambda\E_{\Q^N}[\Tr(\widehat\rho_t^N K_N(t))]\le\recordentropy_N(t)+\log(\mathcal Z_N(\lambda,t))
\).  Since
$\E_{\Q^N}[\Tr(\widehat\rho_t^N K_N(t))]
=\E_{\Q^{N,\eta}}[\Tr(\rho_t^N K_N(t))]$,
this proves \eqref{eq:entropy-exponential-reduction}.

It therefore suffices to bound $\mathcal Z_N(\lambda,t)$ for a
positive $\lambda$ independent of $N$.
Lemma~\ref{lem:faithful-log-moment} below provides a constant
$\kappa_{T,\gamma_0}$ for which
\eqref{eq:scaled-log-moment} holds on $I=[0,T]$.
Proposition~\ref{prop:partition} then gives
\[
  \sup_{N\ge1,\,0\le t\le T}\mathcal Z_N(c_*/\kappa_{T,\gamma_0},t)\le2e.
\]
Substitution in \eqref{eq:entropy-exponential-reduction} yields
\[
  \E_{\Q^{N,\eta}}[\Tr(\rho_t^N K_N(t))]\le\frac{\kappa_{T,\gamma_0}}{c_*}\bigl(\recordentropy_N(t)+\log(2e)\bigr).
\]
This is the estimate needed to control the integral in
\eqref{eq:entropy-production}. We now prove the two results.

\begin{lemma}
\label{lem:faithful-log-moment}
There exists $\kappa_{T,\gamma_0}\ge1$, depending only on
$d,T,L^{(1)},\ldots,L^{(m)}$ and the smallest eigenvalue of
$\gamma_0$, such that
\[
  \E_{\mathbf q_T^Y}\left[\exp\left(\frac{\sup_{0\le t\le T}\|\log(\Pi_t(Y))\|^2}{\kappa_{T,\gamma_0}^2}\right)\right]\le2.
\]
\end{lemma}

\begin{proof}
We work on $(\mathsf Y_T,\cY_T,\mathbf q_T^Y)$ and write
$\Pi_t:=\Pi_t(Y)$.
Recall from Subsection~\ref{subsec:common-canonical} that
\[
  \overline W_t^k:=Y_t^k-\sqrt\eta\int_0^t b^{(k)}(\Pi_s)\,\d s,\qquad 1\le k\le m,
\]
are independent Brownian motions under $\mathbf q_T^Y$.
Equation~\eqref{eq:eta-local-observation-SDE} therefore becomes
\[
  \d\Pi_t=\mathfrak l_t(\Pi_t)\,\d t+\sqrt\eta\sum_{k=1}^m\cE_{L^{(k)}}(\Pi_t)\,\d\overline W_t^k.
\]
Since $\gamma_0>0$, Lemma~\ref{lem:CP-reference} gives $\Pi_t>0$
for all $t\in[0,T]$. Define $Q_t:=-\log(\det(\Pi_t))$.
Every eigenvalue of $\Pi_t$ belongs to $(0,1]$, so
\[
  0\le\|\log(\Pi_t)\|=-\log(\lambda_{\min}(\Pi_t))\le-\log(\det(\Pi_t))=Q_t.
\]
We shall bound $Q_t$ by a deterministic constant and a martingale
with bounded quadratic variation.

For $X>0$, the function $g(X):=-\log(\det(X))$ satisfies
\[
  \mathrm Dg(X)[U]=-\Tr(X^{-1}U),\qquad \mathrm D^2g(X)[U,U]=\Tr(X^{-1}UX^{-1}U).
\]
For each $k$, put
\[
  a_k:=\Tr(A^{(k)}),\qquad b_k(t):=b^{(k)}(\Pi_t),\qquad r_k(t):=\Tr\bigl(\Pi_t^{-1}L^{(k)}\Pi_t(L^{(k)})^*\bigr).
\]
In particular,
$r_k(t)=\|\Pi_t^{-1/2}L^{(k)}\Pi_t^{1/2}\|_2^2\ge0$.
Direct expansion of $\cD_{L^{(k)}}$ and $\cE_{L^{(k)}}$ gives
\[
  \Tr\bigl(\Pi_t^{-1}\cD_{L^{(k)}}(\Pi_t)\bigr)=r_k(t)-\Tr\bigl((L^{(k)})^*L^{(k)}\bigr),  \qquad 
  \Tr\bigl(\Pi_t^{-1}\cE_{L^{(k)}}(\Pi_t)\bigr)=a_k-d\,b_k(t)
\]
and
\[
  \Tr\left(\bigl(\Pi_t^{-1}\cE_{L^{(k)}}(\Pi_t)\bigr)^2\right)=2r_k(t)+2\operatorname{Re}\bigl(\Tr((L^{(k)})^2)\bigr)-2a_kb_k(t)+d\,b_k(t)^2.
\]
The Hamiltonian term vanishes because
$\Tr(\Pi_t^{-1}[h+V^{\xi_t},\Pi_t])=0$.

Define
\[
  c_t:=\sum_{k=1}^m\left(\Tr\bigl((L^{(k)})^*L^{(k)}\bigr)+\eta\operatorname{Re}\bigl(\Tr((L^{(k)})^2)\bigr)-\eta a_kb_k(t)+\frac{\eta d}{2}b_k(t)^2\right)
\]
and
\[
  \mathcal M_t:=-\sqrt\eta\sum_{k=1}^m\int_0^t(a_k-d\,b_k(u))\,\d\overline W_u^k.
\]
It\^o's formula now gives
\[
  \d Q_t=\left(c_t-(1-\eta)\sum_{k=1}^m r_k(t)\right)\,\d t+\d\mathcal M_t,\qquad Q_0=-\log(\det(\gamma_0)).
\]
These identities hold on $[0,T]$: continuity and positivity give
$\inf_{0\le t\le T}\lambda_{\min}(\Pi_t)>0$, so the usual
localization for $g$ exhausts this interval.  We have
\[
  |b_k(t)|\le\|A^{(k)}\|,\qquad |a_k|\le d\|A^{(k)}\|,\qquad |\Tr((L^{(k)})^2)|\le\Tr\bigl((L^{(k)})^*L^{(k)}\bigr).
\]
Consequently, the constants
\[
  C:=2\sum_{k=1}^m\Tr\bigl((L^{(k)})^*L^{(k)}\bigr)+\frac{3d}{2}\sum_{k=1}^m\|A^{(k)}\|^2,\qquad v:=4d^2\sum_{k=1}^m\|A^{(k)}\|^2
\]
satisfy $c_t\le C$ and
\[
  \langle\mathcal M\rangle_t=\eta\sum_{k=1}^m\int_0^t(a_k-d\,b_k(u))^2\,\d u\le vt.
\]
Since $r_k(t)\ge0$ and $0<\eta\le1$, integration of the equation
for $Q_t$ gives
\[
  Q_t\le Q_0+Ct+\mathcal M_t.
\]
Set
\[
  S_T:=\sup_{0\le t\le T}\mathcal M_t,\qquad q_*:=d\log(1/\lambda_{\min}(\gamma_0))+CT,\qquad \kappa:=4\sqrt{1+q_*^2+vT}.
\]
Then $S_T\ge0$, $\kappa\ge1$, and
\[
  \sup_{0\le t\le T}\|\log(\Pi_t)\|\le q_*+S_T.
\]
If $v=0$, then $\mathcal M=0$ and the desired expectation is at most
$\exp(q_*^2/\kappa^2)\le e^{1/16}<2$.
We henceforth assume $v>0$.

For $\theta,x>0$, stopping
$\exp(\theta\mathcal M_t-\theta^2\langle\mathcal M\rangle_t/2)$
at the first time $\mathcal M$ reaches $x$ gives
\[
  \mathbf q_T^Y(S_T>x)\le\exp(-\theta x+\theta^2vT/2).
\]
Taking $\theta=x/(vT)$ yields
\[
  \mathbf q_T^Y(S_T>x)\le\exp\left(-\frac{x^2}{2vT}\right).
\]
Tonelli's theorem therefore gives
\[
  \E_{\mathbf q_T^Y}\!\left[\exp\left(\frac{2S_T^2}{\kappa^2}\right)\right]=1+\frac4{\kappa^2}\int_0^\infty x\exp\left(\frac{2x^2}{\kappa^2}\right)\mathbf q_T^Y(S_T>x)\,\d x.
\]
Using the tail bound to evaluate the integral, and
$\kappa^2\ge16vT$, we obtain
\[
  \E_{\mathbf q_T^Y}\!\left[\exp\left(\frac{2S_T^2}{\kappa^2}\right)\right]\le\frac1{1-4vT/\kappa^2}\le\frac43.
\]
Finally, $(q_*+S_T)^2\le2q_*^2+2S_T^2$ implies
\[
  \E_{\mathbf q_T^Y}\!\left[\exp\left(\frac{\sup_{0\le t\le T}\|\log(\Pi_t)\|^2}{\kappa^2}\right)\right]\le e^{2q_*^2/\kappa^2}\E_{\mathbf q_T^Y}[e^{2S_T^2/\kappa^2}]\le\frac43e^{1/8}<2.
\]
The lemma follows with $\kappa_{T,\gamma_0}:=\kappa$.
\end{proof}

We state the next proposition for a general interval $I\subset[0,T]$, since its estimate will also be used on shorter time intervals in Subsection~\ref{subsec:local-entropy}.

\begin{proposition}
\label{prop:partition}
Suppose that $I\subset[0,T]$ and $\kappa\ge1$ satisfy
\begin{equation}
  \sup_{t\in I}\E_{\mathbf q_T^Y}\left[\exp\left(\frac{\|\log(\Pi_t(Y))\|^2}{\kappa^2}\right)\right]\le2.
  \label{eq:scaled-log-moment}
\end{equation}
There exists $c_*>0$, depending only on $d$ and $V$, such that
\begin{equation}
  \sup_{N\ge1,\,t\in I}\mathcal Z_N(c_*/\kappa,t)\le2e.
  \label{eq:partition-bound}
\end{equation}
More generally, the same conclusion holds for any measurable maps
$\Pi_t:\mathsf Y_T\to\cS(\HH)$ and probability measure
$\mathbf q_T^Y$ satisfying $\Pi_t>0$ and
\eqref{eq:scaled-log-moment}, with the corresponding products
$\overline\Q^N=(\mathbf q_T^Y)^{\otimes N}$ and
$\overline\Gamma_t^{N,N}=\Pi_t(Y^1)\otimes\cdots\otimes\Pi_t(Y^N)$.
In this case, $K_N(t)$ is defined using
$\overline H_t^N:=\sum_{j=1}^N(h+V^{\zeta_t})_j$,
where $\zeta_t:=\E_{\mathbf q_T^Y}[\Pi_t(Y)]$.
\end{proposition}

\begin{proof}
Fix $N\ge1$ and $t\in I$.
Choose a basis $(F_a)_{a=1}^{d^2}$ of $\cL^*(\HH)$ satisfying
$\Tr(F_aF_b)=\delta_{ab}$.
In particular, $\|F_a\|^2\le\Tr(F_a^2)=1$.
Since $V$ is self-adjoint and invariant under exchange of its
two tensor factors,
\begin{equation}
  V=\sum_{a,b=1}^{d^2}v_{ab}F_a\otimes F_b,\qquad v_{ab}:=\Tr\bigl(V(F_a\otimes F_b)\bigr)=v_{ba}\in\mathbb R.
  \label{eq:V-expansion}
\end{equation}
Define
\begin{equation}
  a_a(t):=\Tr(\xi_tF_a),\qquad A_a(t):=F_a-a_a(t)I_{\HH},\qquad B_a(t,y):=\i[\log(\Pi_t(y)),F_a].
  \label{eq:centered-fields}
\end{equation}
Both $A_a(t)$ and $B_a(t,y)$ are self-adjoint, and
\[
  \|A_a(t)\|\le2,\qquad \|B_a(t,y)\|\le2\|\log(\Pi_t(y))\|.
\]
Their centering properties are
\begin{equation}
  \E_{\mathbf q_T^Y}[\Tr(\Pi_t(Y)A_a(t))]=0,\qquad \Tr(\Pi_t(y)B_a(t,y))=0.
  \label{eq:centering}
\end{equation}
Indeed, the first left-hand side equals
$\Tr(\xi_tF_a)-a_a(t)=0$ by \eqref{eq:reference-Pi-mean}.
The second equals
$\i\Tr([\Pi_t(y),\log(\Pi_t(y))]F_a)=0$.

We write $A_{a,j}:=(A_a(t))_j$ and
$B_{a,j}:=(B_a(t,Y^j))_j$ for the matrices acting on the
$j$-th tensor factor, and set
\[
  S_{A_a}:=\sum_{j=1}^N A_{a,j},\qquad S_{B_a}:=\sum_{j=1}^N B_{a,j}.
\]
We first express $K_N(t)$ in terms of these sums.
The definitions of $V^{\xi_t}$ and $\overline\Gamma_t^{N,N}$ give
\[
  V^{\xi_t}=\sum_{a,b=1}^{d^2}v_{ab}a_b(t)F_a,\qquad \log(\overline\Gamma_t^{N,N})=\sum_{j=1}^N(\log(\Pi_t(Y^j)))_j.
\]
Consequently, for $i<j$,
\[
  \i[\log(\overline\Gamma_t^{N,N}),V_{ij}]=\sum_{a,b=1}^{d^2}v_{ab}(B_{a,i}F_{b,j}+F_{a,i}B_{b,j}),
\]
whereas
\[
  \i[\log(\overline\Gamma_t^{N,N}),(V^{\xi_t})_j]=\sum_{a,b=1}^{d^2}v_{ab}a_b(t)B_{a,j}.
\]
Using $v_{ab}=v_{ba}$ and commutation of matrices acting on
distinct tensor factors, we obtain
\[
  K_N(t)=\frac1N\sum_{a,b=1}^{d^2}v_{ab}\sum_{i\ne j}B_{a,i}F_{b,j}-\sum_{a,b=1}^{d^2}v_{ab}a_b(t)S_{B_a}.
\]
Moreover,
\[
  \frac12\{S_{B_a},S_{A_b}\}=\sum_{i\ne j}B_{a,i}F_{b,j}+\frac12\sum_{j=1}^N\{B_{a,j},F_{b,j}\}-Na_b(t)S_{B_a}.
\]
It follows that
\begin{equation}
  K_N(t)=\frac1{2N}\sum_{a,b=1}^{d^2}v_{ab}\{S_{B_a},S_{A_b}\}-\frac1{2N}\sum_{j=1}^N\sum_{a,b=1}^{d^2}v_{ab}\{B_{a,j},F_{b,j}\}.
  \label{eq:K-decomposition}
\end{equation}
The second sum accounts for the terms with $i=j$.

For self-adjoint $X,Z$, positivity of
$(\kappa^{-1/2}X\pm\kappa^{1/2}Z)^2$ gives
\[
  \pm\{X,Z\}\le\kappa^{-1}X^2+\kappa Z^2.
\]
Applying this inequality to both sums in
\eqref{eq:K-decomposition}, we obtain
\[
  K_N(t)\le\frac{\kappa}{2N}\sum_{a,b}|v_{ab}|(S_{A_b}^2+\kappa^{-2}S_{B_a}^2)+\frac{\kappa}{2N}\sum_{j,a,b}|v_{ab}|(F_{b,j}^2+\kappa^{-2}B_{a,j}^2).
\]
Here and below, the indices $a,b$ range from $1$ to $d^2$,
and $j$ ranges from $1$ to $N$.
By Cauchy--Schwarz and orthonormality of the basis,
\[
  \sum_{a,b}|v_{ab}|\le d^2\bigl(\Tr(V^2)\bigr)^{1/2}\le d^3\|V\|.
\]
Since $F_{b,j}^2\le I_{\HH^N}$, we may therefore take
$C_V:=1+d^3\|V\|$ in the bound
\begin{equation}
  K_N(t)\le C_V\kappa\left(I_{\HH^N}+\sum_{a=1}^{d^2}\left[\frac{S_{A_a}^2}{N}+\frac{S_{B_a}^2}{\kappa^2N}+\frac1{\kappa^2N}\sum_{j=1}^N B_{a,j}^2\right]\right).
  \label{eq:K-square-upper-bound}
\end{equation}
Put $J:=3d^2$ and define the positive matrices
\[
  Q_a:=\frac{S_{A_a}^2}{N},\qquad Q_{d^2+a}:=\frac{S_{B_a}^2}{\kappa^2N},\qquad Q_{2d^2+a}:=\frac1{\kappa^2N}\sum_{j=1}^N B_{a,j}^2,\qquad 1\le a\le d^2.
\]
We next prove that, with $u_0:=1/576$,
\begin{equation}
  \E_{\overline\Q^N}\left[\Tr\bigl(\overline\Gamma_t^{N,N}\exp(uQ_\ell)\bigr)\right]\le2,\qquad 0\le u\le u_0,\quad 1\le\ell\le J.
  \label{eq:individual-square-bound}
\end{equation}
For a measurable self-adjoint matrix $F(t,y)$ on $\HH$, define
\begin{equation}
  \nu_{F,t}(B):=\int_{\mathsf Y_T}\Tr\bigl(\Pi_t(y)\1_B(F(t,y))\bigr)\,\mathbf q_T^Y(\d y),\qquad B\subset\mathbb R\ \text{Borel}.
  \label{eq:spectral-law}
\end{equation}
If $F(t,y)=\sum_r z_rP_r$ is its spectral decomposition, then
$\1_B(F(t,y)):=\sum_{r:z_r\in B}P_r$.
For every orthogonal projection $P$,
$\Tr(\Pi_t(y)P)=\Tr(P\Pi_t(y)P)\ge0$.
Together with $\Tr(\Pi_t(y))=1$, this shows that $\nu_{F,t}$ is
a probability measure. Its definition gives, for every nonnegative
Borel function $g$,
\[
  \int_{\mathbb R}g(z)\,\nu_{F,t}(\d z)=\E_{\mathbf q_T^Y}[\Tr(\Pi_t(Y)g(F(t,Y)))].
\]
For $F=A_a$, the bound $\|A_a(t)\|\le2$ implies
\[
  \int_{\mathbb R}e^{z^2/9}\,\nu_{A_a,t}(\d z)=\E_{\mathbf q_T^Y}[\Tr(\Pi_t(Y)\exp(A_a(t)^2/9))]\le e^{4/9}<2.
\]
For $F=B_a/\kappa$, the bound on $B_a$ and
\eqref{eq:scaled-log-moment} give
\[
  \int_{\mathbb R}e^{z^2/9}\,\nu_{B_a/\kappa,t}(\d z)\le\E_{\mathbf q_T^Y}\left[\exp\left(\frac{4\|\log(\Pi_t(Y))\|^2}{9\kappa^2}\right)\right]\le2.
\]
Both measures thus have finite first moments, and
\eqref{eq:centering} yields
\[
  \int_{\mathbb R}z\,\nu_{A_a,t}(\d z)=0,\qquad \int_{\mathbb R}z\,\nu_{B_a/\kappa,t}(\d z)=0.
\]
We now relate these measures to the matrices on $\HH^N$.
For each $j$, write
$F(t,Y^j)=\sum_r z_r(Y^j)P_r(Y^j)$ and let
$F_j(t,Y^j)$ denote its action on the $j$-th tensor factor.
The product defining $\overline\Gamma_t^{N,N}$ gives
\[
  \Tr\bigl(\overline\Gamma_t^{N,N}(P_{r_1}(Y^1)\otimes\cdots\otimes P_{r_N}(Y^N))\bigr)=\prod_{j=1}^N\Tr(\Pi_t(Y^j)P_{r_j}(Y^j)).
\]
Expanding the exponentials in these projections and integrating
under $\overline\Q^N=(\mathbf q_T^Y)^{\otimes N}$ therefore yields
\[
  \E_{\overline\Q^N}\left[\Tr\left(\overline\Gamma_t^{N,N}\exp\left(\frac uN\Big(\sum_{j=1}^N F_j(t,Y^j)\Big)^2\right)\right)\right]=\int_{\mathbb R^N}\exp\left(\frac uN\Big(\sum_{j=1}^N z_j\Big)^2\right)\prod_{j=1}^N\nu_{F,t}(\d z_j)
\]
and
\[
  \E_{\overline\Q^N}\left[\Tr\left(\overline\Gamma_t^{N,N}\exp\left(\frac uN\sum_{j=1}^N F_j(t,Y^j)^2\right)\right)\right]=\int_{\mathbb R^N}\exp\left(\frac uN\sum_{j=1}^N z_j^2\right)\prod_{j=1}^N\nu_{F,t}(\d z_j).
\]
Only matrices acting on distinct tensor factors are commuted in
these identities.

Apply Lemma~\ref{lem:scalar} with $K=3$.
For $F=A_a$, the first identity bounds the expectation associated
with $Q_a$ by $\sqrt2$.
For $F=B_a/\kappa$, it gives the same bound for $Q_{d^2+a}$,
while the second identity bounds the expectation associated with
$Q_{2d^2+a}$ by $2$.
These estimates hold for
$0\le u\le(64\cdot3^2)^{-1}=u_0$,
which proves \eqref{eq:individual-square-bound}.

It remains to combine these estimates.
Choose
\[
  c_*:=\frac{u_0}{JC_V},\qquad \theta:=c_*C_V=\frac{u_0}{J}.
\]
Then $J\theta=u_0$ and $0<\theta\le1$.
By \eqref{eq:K-square-upper-bound} and monotonicity of the trace
of the exponential,
\[
  \Tr\left(\exp\bigl(\log(\overline\Gamma_t^{N,N})+(c_*/\kappa)K_N(t)\bigr)\right)\le e^\theta\Tr\left(\exp\left(\log(\overline\Gamma_t^{N,N})+\theta\sum_{\ell=1}^JQ_\ell\right)\right).
\]
The identity
\[
  \log(\overline\Gamma_t^{N,N})+\theta\sum_{\ell=1}^JQ_\ell=\frac1J\sum_{\ell=1}^J\bigl(\log(\overline\Gamma_t^{N,N})+J\theta Q_\ell\bigr)
\]
and convexity of $X\mapsto\Tr(\exp(X))$ imply
\[
  \Tr\left(\exp\left(\log(\overline\Gamma_t^{N,N})+\theta\sum_{\ell=1}^JQ_\ell\right)\right)\le\frac1J\sum_{\ell=1}^J\Tr\left(\exp\bigl(\log(\overline\Gamma_t^{N,N})+J\theta Q_\ell\bigr)\right).
\]
For each $\ell$, the Golden--Thompson inequality gives
\[
  \Tr\left(\exp\bigl(\log(\overline\Gamma_t^{N,N})+J\theta Q_\ell\bigr)\right)\le\Tr\bigl(\overline\Gamma_t^{N,N}\exp(J\theta Q_\ell)\bigr).
\]
Taking expectations and using
\eqref{eq:individual-square-bound}, with $J\theta=u_0$, proves
\[
  \mathcal Z_N(c_*/\kappa,t)\le\frac{e^\theta}{J}\sum_{\ell=1}^J\E_{\overline\Q^N}\left[\Tr\bigl(\overline\Gamma_t^{N,N}\exp(u_0Q_\ell)\bigr)\right]\le2e^\theta\le2e.
\]
Since $N\ge1$ and $t\in I$ were arbitrary, this is
\eqref{eq:partition-bound}.

For the final assertion, replace $\xi_t$ by $\zeta_t$ in
$a_a(t)$ and $\overline H_t^N$.
The identities \eqref{eq:centering} and
\eqref{eq:K-decomposition} remain valid because
$\zeta_t=\E_{\mathbf q_T^Y}[\Pi_t(Y)]$.
All subsequent estimates use only
$\Pi_t>0$, $\Tr(\Pi_t)=1$,
\eqref{eq:scaled-log-moment}, and the displayed product formulas.
The same constant $c_*$ therefore applies.
\end{proof}

\subsection{Propagation of chaos from relative entropy}
\label{subsec:faithful-completion}

We now prove Proposition~\ref{prop:faithful-entropy} and then estimate
$N^{-1}D(\rho_0^N\mid\gamma_0^{\otimes N})$ in terms of $\delta_N^2$
to obtain Theorem~\ref{thm:main-chaotic}(ii).
Put $c_0:=\log(2e)$.

\begin{proof}[Proof of Proposition~\ref{prop:faithful-entropy}]
Let $\lambda:=c_*/\kappa_{T,\gamma_0}$, where
$\kappa_{T,\gamma_0}$ is given by
Lemma~\ref{lem:faithful-log-moment}.
Proposition~\ref{prop:partition}, applied with $I=[0,T]$, gives
$\mathcal Z_N(\lambda,t)\le2e$ for every $N\ge1$ and $0\le t\le T$.
Thus \eqref{eq:entropy-exponential-reduction} yields
\begin{equation}
  \E_{\Q^{N,\eta}}\left[\Tr(\rho_t^N K_N(t))\right]\le\lambda^{-1}\bigl(\recordentropy_N(t)+c_0\bigr).
  \label{eq:gibbs-production-bound}
\end{equation}
Combining this inequality with Proposition~\ref{prop:H-production},
with $s=0$, gives
\[
  \recordentropy_N(t)+c_0\le\recordentropy_N(0)+c_0+\lambda^{-1}\int_0^t\bigl(\recordentropy_N(u)+c_0\bigr)\,\d u.
\]
Gronwall's inequality therefore implies
\begin{equation}
  \recordentropy_N(t)+c_0\le e^{t/\lambda}\bigl(\recordentropy_N(0)+c_0\bigr),\qquad 0\le t\le T.
  \label{eq:H-final}
\end{equation}
In particular,
\begin{equation}
  \sup_{0\le t\le T}\recordentropy_N(t)\le C_{T,\gamma_0}\bigl(\recordentropy_N(0)+1\bigr),
  \label{eq:entropy-E1}
\end{equation}
where the constant is independent of $N$ and $\rho_0^N$.

We next estimate the difference between $\rho_t^{N:n}$ and
$\overline\Gamma_t^{N,n}$.
Fix $1\le n\le N$, set $r_N:=\lfloor N/n\rfloor$, and define
$I_\ell:=\{(\ell-1)n+1,\ldots,\ell n\}$ for $1\le\ell\le r_N$.
Monotonicity of relative entropy under the partial trace gives
\[
  D(\rho_t^N\mid\overline\Gamma_t^{N,N})\ge D(\rho_t^{N:nr_N}\mid\overline\Gamma_t^{N,nr_N}).
\]
Using the tensor-product logarithm and
$S(\rho_t^{N:nr_N})\le\sum_{\ell=1}^{r_N}S(\rho_t^{N:I_\ell})$,
we further obtain
\[
  D(\rho_t^{N:nr_N}\mid\overline\Gamma_t^{N,nr_N})\ge\sum_{\ell=1}^{r_N}D\left(\rho_t^{N:I_\ell}\,\middle|\,\bigotimes_{j\in I_\ell}\overline\gamma_t^{N,j}\right).
\]
By Lemma~\ref{lem:eta-projected-exchangeability}, the family
$(\rho^N,(\overline\gamma^{N,j})_{j=1}^N)$ is exchangeable under $\Q^{N,\eta}$.  Consequently, for every $\ell$,
\[
  \E_{\Q^{N,\eta}}\left[D\left(\rho_t^{N:I_\ell}\,\middle|\,\bigotimes_{j\in I_\ell}\overline\gamma_t^{N,j}\right)\right]=\E_{\Q^{N,\eta}}\left[D(\rho_t^{N:n}\mid\overline\Gamma_t^{N,n})\right].
\]
Taking expectations in the preceding inequalities and using
\eqref{eq:joint-entropy-chain-rule}, together with the nonnegativity
of classical relative entropy, gives
\[
  r_N\E_{\Q^{N,\eta}}\left[D(\rho_t^{N:n}\mid\overline\Gamma_t^{N,n})\right]\le\E_{\Q^{N,\eta}}\left[D(\rho_t^N\mid\overline\Gamma_t^{N,N})\right]\le\recordentropy_N(t).
\]
Quantum Pinsker's inequality now yields
\begin{equation}
  \E_{\Q^{N,\eta}}\left[\|\rho_t^{N:n}-\overline\Gamma_t^{N,n}\|_1^2\right]\le\frac{2\recordentropy_N(t)}{r_N}.
  \label{eq:entropy-E2}
\end{equation}
It remains to compare $\overline\Gamma_t^{N,n}$ with $\Gamma_t^n$.
The case $n=1$ of \eqref{eq:entropy-E2}, together with permutation
invariance in law, gives, for every $j\in[N]$,
\[
  \E_{\Q^{N,\eta}}\left[\|\rho_t^{N:(j)}-\overline\gamma_t^{N,j}\|_1^2\right]\le\frac{2\recordentropy_N(t)}{N}.
\]
Let $C_{\eta,T}$ be the constant in
Lemma~\ref{lem:eta-local-comparison}.
Applying that lemma and substituting the last bound, we obtain
\begin{equation}
  \E_{\Q^{N,\eta}}\left[\sup_{0\le t\le T}\|\overline\gamma_t^{N,j}-\gamma_t^j\|_1^2\right]\le\frac{2C_{\eta,T}}{N}\int_0^T\recordentropy_N(u)\,\d u.
  \label{eq:entropy-E3}
\end{equation}
Every factor in $\overline\Gamma_t^{N,n}$ and $\Gamma_t^n$ has
trace norm one. Expanding the difference of these products one
factor at a time and using
$\|A\otimes B\|_1=\|A\|_1\|B\|_1$, we find
\[
  \|\overline\Gamma_t^{N,n}-\Gamma_t^n\|_1\le\sum_{j=1}^n\|\overline\gamma_t^{N,j}-\gamma_t^j\|_1.
\]
The triangle inequality, Cauchy--Schwarz, and
\eqref{eq:entropy-E2}--\eqref{eq:entropy-E3} therefore give
\[
  \E_{\Q^{N,\eta}}\left[\|\rho_t^{N:n}-\Gamma_t^n\|_1\right]\le\left(\frac{2\recordentropy_N(t)}{r_N}\right)^{1/2}+n\left(\frac{2C_{\eta,T}}{N}\int_0^T\recordentropy_N(u)\,\d u\right)^{1/2}.
\]
Since $r_N\ge N/(2n)$, it follows that
\begin{equation}
  \sup_{0\le t\le T}\E_{\Q^{N,\eta}}\left[\|\rho_t^{N:n}-\Gamma_t^n\|_1\right]\le C_{\eta,T,n}\left(\frac{\sup_{0\le u\le T}\recordentropy_N(u)}{N}\right)^{1/2}.
  \label{eq:entropy-to-target}
\end{equation}
Combining \eqref{eq:entropy-to-target} with
\eqref{eq:entropy-E1} and
$\recordentropy_N(0)=D(\rho_0^N\mid\gamma_0^{\otimes N})$
proves \eqref{eq:faithful-proposition}.

Finally, uniqueness in law of the jointly coupled equations
\eqref{eq:N-Belavkin} and \eqref{eq:MF-Belavkin}, driven by the
same Brownian motions, identifies this expectation with the one
in Definition~\ref{def:propagation-of-chaos}.
\end{proof}

To deduce Theorem~\ref{thm:main-chaotic}(ii), we estimate the initial
relative entropy in terms of
$\delta_N^2=\|\rho_0^{N:2}-\gamma_0^{\otimes2}\|_1$.

\begin{lemma}
\label{lem:faithful-initial-conversion}
Assume that $\gamma_0>0$.
For every $N\ge2$ and every permutation-invariant
$\rho_0^N\in\cS(\HH^N)$,
\begin{equation}
  \frac1N D(\rho_0^N\mid\gamma_0^{\otimes N})\le C_{d,\gamma_0}\left(\delta_N^2+\frac1N\right)^{1/2}.
  \label{eq:faithful-initial-conversion}
\end{equation}
The constant depends only on $d$ and the smallest eigenvalue
of $\gamma_0$.
\end{lemma}

\begin{proof}
Let $r=(r_1,\ldots,r_d)$ be the decreasing eigenvalue vector of
$\gamma_0$, and put $a:=\min_i r_i>0$.
For a probability vector $p=(p_1,\ldots,p_d)$, define
$H(p):=-\sum_{i=1}^d p_i\log(p_i)$.  We first show that
\[
  |H(p)-H(r)|\le C_{d,a}|p-r|
\]
for every probability vector $p$, where $|\cdot|$ denotes the
Euclidean norm.

If $|p-r|\le a/2$, every coordinate on the segment from $r$ to $p$
is at least $a/2$. Since
$\nabla H(q)=(-1-\log(q_i))_{i=1}^d$ for $q_i>0$, the gradient norm
on this segment is bounded by $\sqrt d(1+\log(2/a))$.
The mean value theorem gives the required inequality.
If $|p-r|>a/2$, the bound $0\le H\le\log(d)$ gives
$|H(p)-H(r)|\le2\log(d)|p-r|/a$.
This proves the claim in both cases.

Recall the Schur--Weyl decomposition \eqref{eq:Schur-state} in
Appendix~\ref{app:initial}. Its indices are the integer vectors
$\lambda=(\lambda_1,\ldots,\lambda_d)$ satisfying
$\lambda_1\ge\cdots\ge\lambda_d\ge0$ and $\sum_i\lambda_i=N$.
If $P_\lambda$ denotes the orthogonal projection onto the
$\lambda$-summand, then
$p_N(\lambda):=\Tr(P_\lambda\rho_0^N)$ satisfies
$p_N(\lambda)\ge0$ and $\sum_\lambda p_N(\lambda)=1$.
Equations~\eqref{eq:entropy-shape} and
\eqref{eq:shape-concentration} give
\[
  \left|\frac1N S(\rho_0^N)-\sum_\lambda p_N(\lambda)H(\lambda/N)\right|\le C_d\frac{\log(N+1)}{N}
\]
and
\[
  \sum_\lambda p_N(\lambda)|\lambda/N-r|^2\le\delta_N^2+2\delta_N^1+\frac{2d-1}{N}.
\]
Since $S(\gamma_0)=H(r)$, the preceding bound on $H$ implies
\[
  \left|\frac1N S(\rho_0^N)-S(\gamma_0)\right|\le C_{d,a}\sum_\lambda p_N(\lambda)|\lambda/N-r|+C_d\frac{\log(N+1)}{N}.
\]
Cauchy--Schwarz and the concentration estimate therefore yield
\[
  \left|\frac1N S(\rho_0^N)-S(\gamma_0)\right|\le C_{d,a}\left(\delta_N^2+2\delta_N^1+\frac{2d-1}{N}\right)^{1/2}+C_d\frac{\log(N+1)}{N}.
\]
The partial trace is contractive in trace norm, so
$\delta_N^1\le\delta_N^2$.
Together with $\log(N+1)/N\le N^{-1/2}$, this gives
\[
  \left|\frac1N S(\rho_0^N)-S(\gamma_0)\right|\le C_{d,a}\left(\delta_N^2+\frac1N\right)^{1/2}.
\]
Permutation invariance and
$\log(\gamma_0^{\otimes N})=\sum_{j=1}^N(\log(\gamma_0))_j$
give
\[
  \frac1N D(\rho_0^N\mid\gamma_0^{\otimes N})=S(\gamma_0)-\frac1N S(\rho_0^N)-\Tr\bigl((\rho_0^{N:1}-\gamma_0)\log(\gamma_0)\bigr).
\]
Moreover, $\|\log(\gamma_0)\|=\log(1/a)$ and
$\delta_N^1\le\delta_N^2\le2$, whence
\[
  \left|\Tr\bigl((\rho_0^{N:1}-\gamma_0)\log(\gamma_0)\bigr)\right|\le\delta_N^1\log(1/a)\le\sqrt{2}\log(1/a)\left(\delta_N^2+\frac1N\right)^{1/2}.
\]
Combining the last three displays proves
\eqref{eq:faithful-initial-conversion}.
\end{proof}

We now combine Proposition~\ref{prop:faithful-entropy} with
Lemma~\ref{lem:faithful-initial-conversion} to prove
Theorem~\ref{thm:main-chaotic}(ii).
The proposition estimates the propagation-of-chaos error in terms
of the initial relative entropy, while the lemma bounds this
entropy in terms of the two-particle error $\delta_N^2$.

\begin{proof}[Proof of Theorem~\ref{thm:main-chaotic}(ii)]
By Lemma~\ref{lem:faithful-initial-conversion} and
$N^{-1}\le(\delta_N^2+N^{-1})^{1/2}$,
\[
  \frac{D(\rho_0^N\mid\gamma_0^{\otimes N})+1}{N}\le C_{d,\gamma_0}\left(\delta_N^2+\frac1N\right)^{1/2}.
\]
Substituting this bound into \eqref{eq:faithful-proposition}
proves \eqref{eq:main-chaotic-faithful-bound}.
\end{proof}

In particular,  for $\rho_0^N=\gamma_0^{\otimes N}$,   Proposition~\ref{prop:faithful-entropy} gives
\[
  \sup_{0\le t\le T}\E_{\Q^{N,\eta}}\left[\|\rho_t^{N:n}-\Gamma_t^n\|_1\right]\le C_{T,n,\gamma_0}N^{-1/2}.
\]
To obtain Theorem~\ref{thm:main-approx}(ii) for general
$\rho_0^N$, which need not be permutation invariant, we will apply
Proposition~\ref{prop:uniform-transfer-final} in
Section~\ref{sec:extensions}. 

The next subsection prepares the proof of
Theorem~\ref{thm:main-chaotic}(i), where $\gamma_0$ need not be
faithful.

\subsection{Local entropy bounds under regularization}
\label{subsec:local-entropy}

Now we allow arbitrary $\gamma_0\in\cS(\HH)$ and
$\rho_0^N\in\cS(\HH^N)$.
To apply the preceding entropy estimates, we will compare $\rho^N$
with tensor products of positive definite one-particle matrices.
For $R\in\cS(\HH)$ and $0<a\le1/2$, the matrix
$R_a:=(1-a)R+aI_{\HH}/d$ satisfies
\[
  R_a\ge\frac a d I_{\HH},\qquad \|R_a-R\|_1\le2a.
\]
This operation can be applied to $\gamma_0$ even when it is singular.
Its later application at a deterministic time $s$ requires allowing
the matrix at that time to depend on $\cY_s$.
We therefore establish an entropy estimate starting from a random
matrix at time $s$, with explicit dependence on the lower eigenvalue
bound $a/d$.

Fix $0\le s<T$ and $0<a\le1/2$.
On $(\mathsf Y_T,\cY_T,\mathbb W_T)$, let
$\widehat{\overline\gamma}_s^s>0$ be $\cY_s$-measurable, and set
\[
  \overline Z_s^s:=\Tr(\widehat{\overline\gamma}_s^s),\qquad \Pi_s^s:=\frac{\widehat{\overline\gamma}_s^s}{\overline Z_s^s}.
\]
The superscript $s$ labels the construction starting at time $s$.
Assume that
\[
  \E_{\mathbb W_T}[\overline Z_s^s]=1,\qquad \E_{\mathbb W_T}\left[(\overline Z_s^s)^p+(\overline Z_s^s)^{-p}\right]<\infty\quad\text{for every }p>0,
\]
and
\begin{equation}
  \Pi_s^s\ge\frac a d I_{\HH}.
  \label{eq:local-reference-floor}
\end{equation}
Put $\zeta_s:=\E_{\mathbb W_T}[\widehat{\overline\gamma}_s^s]$,
which belongs to $\cS(\HH)$, and let $\zeta$ solve
\eqref{eq:mean-state-equation} on $[s,T]$ with this initial state.
Define
\[
  \mathfrak l_t^s(X):=-\i[h+V^{\zeta_t},X]+\sum_{k=1}^m\cD_{L^{(k)}}(X),\qquad X\in\cL(\HH),\quad s\le t\le T.
\]
Let $\widehat{\overline\gamma}^s(Y)$ be the solution under
$\mathbb W_T$ on $[s,T]$ to
\begin{equation}
  \d\widehat{\overline\gamma}_t^s=\mathfrak l_t^s(\widehat{\overline\gamma}_t^s)\,\d t+\sqrt\eta\sum_{k=1}^m\cC_{L^{(k)}}(\widehat{\overline\gamma}_t^s)\,\d Y_t^k,
  \label{eq:local-reference-linear}
\end{equation}
with the initial matrix specified above.
Lemma~\ref{lem:CP-reference}, applied to the linear evolution
on $[s,T]$, gives $\widehat{\overline\gamma}_t^s>0$.
For $s\le t\le T$, set
\[
  \overline Z_t^s:=\Tr(\widehat{\overline\gamma}_t^s),\qquad \Pi_t^s(Y):=\frac{\widehat{\overline\gamma}_t^s(Y)}{\overline Z_t^s(Y)}.
\]
Taking the trace in \eqref{eq:local-reference-linear} gives
\[
  \d\overline Z_t^s=\sqrt\eta\,\overline Z_t^s\sum_{k=1}^m b^{(k)}(\Pi_t^s)\,\d Y_t^k.
\]
Since $|b^{(k)}(\Pi_t^s)|\le\|A^{(k)}\|$,
Novikov's condition applies to $\overline Z_t^s/\overline Z_s^s$
on $[s,T]$. Consequently,
\[
  \E_{\mathbb W_T}[\overline Z_t^s\mid\cY_u]=\overline Z_u^s,\qquad s\le u\le t\le T.
\]
In particular, $\E_{\mathbb W_T}[\overline Z_t^s]=1$.

Define the probability measures $\mathbf q_T^{Y,s}$ on
$(\mathsf Y_T,\cY_T)$ and $\overline\Q^{N,s}$ on
$(\Omega_N,\cY_T^N)$ by
\[
  \left.\frac{\d\mathbf q_T^{Y,s}}{\d\mathbb W_T}\right|_{\cY_T}:=\overline Z_T^s,\qquad \overline\Q^{N,s}:=(\mathbf q_T^{Y,s})^{\otimes N}.
\]
They are equivalent to $\mathbb W_T$ and $\Q^N$, respectively.
For $s\le t\le T$,
\[
  \left.\frac{\d\mathbf q_T^{Y,s}}{\d\mathbb W_T}\right|_{\cY_t}=\overline Z_t^s.
\]
In particular, the density on $\cY_s$ is $\overline Z_s^s$.

On $\Omega_N=\mathsf Y_T^N$, define
\[
  \widehat{\overline\Gamma}_t^{N,s}:=\widehat{\overline\gamma}_t^s(Y^1)\otimes\cdots\otimes\widehat{\overline\gamma}_t^s(Y^N),\qquad \overline Z_t^{N,s}:=\prod_{j=1}^N\overline Z_t^s(Y^j),
\]
and
\[
  \overline\Gamma_t^{N,s,N}:=\Pi_t^s(Y^1)\otimes\cdots\otimes\Pi_t^s(Y^N).
\]
Thus
\[
  \widehat{\overline\Gamma}_t^{N,s}=\overline Z_t^{N,s}\overline\Gamma_t^{N,s,N},\qquad \left.\frac{\d\overline\Q^{N,s}}{\d\Q^N}\right|_{\cY_t^N}=\overline Z_t^{N,s},\qquad s\le t\le T.
\]
Recall that
$\widehat\rho_t^N=\Phi_t^{N,\eta}(\rho_0^N)$.
We compare these matrices through
\[
  \recordentropy_N^s(t):=\E_{\Q^N}\left[\grel(\widehat\rho_t^N\mid\widehat{\overline\Gamma}_t^{N,s})\right],\qquad s\le t\le T.
\]
Put $T_s:=\min\{s+1,T\}$ and recall that $c_0=\log(2e)$.

\begin{proposition}
\label{prop:local-entropy}
Under the preceding assumptions, there exists $C_*>0$, depending
only on $d,h,V,L^{(1)},\ldots,L^{(m)}$ and $\eta$, such that,
for every $N\ge1$,
\begin{equation}
  \recordentropy_N^s(t)+c_0\le\exp\bigl(C_*(1+|\log(a)|)(t-s)\bigr)\bigl(\recordentropy_N^s(s)+c_0\bigr),\qquad s\le t\le T_s.
  \label{eq:local-entropy-growth}
\end{equation}
The constant is independent of $N,s,T,a,\rho_0^N$ and the choice
of $\widehat{\overline\gamma}_s^s$.
\end{proposition}

\begin{proof}
We first establish the moment bounds needed for the entropy
calculation.
Let $G^s$ solve \eqref{eq:local-reference-linear} with initial
condition $G_s^s=I_{\HH}$.
The coefficients are deterministic, so $G^s$ depends only on
the Wiener increments after $s$ and is independent of $\cY_s$
under $\mathbb W_T$.
The initial bounds
\[
  \frac a d\,\overline Z_s^s I_{\HH}\le\widehat{\overline\gamma}_s^s\le\overline Z_s^s I_{\HH}
\]
and positivity and linearity of the evolution imply
\[
  \frac a d\,\overline Z_s^s G_t^s\le\widehat{\overline\gamma}_t^s\le\overline Z_s^s G_t^s,\qquad s\le t\le T.
\]
Since $G_t^s>0$, inversion reverses the operator order and gives
\[
  \|\widehat{\overline\gamma}_t^s\|\le\overline Z_s^s\|G_t^s\|,\qquad \|(\widehat{\overline\gamma}_t^s)^{-1}\|\le\frac{d}{a\overline Z_s^s}\|(G_t^s)^{-1}\|.
\]
Lemma~\ref{lem:CP-reference}, independence from $\cY_s$, and the
assumed moments of $\overline Z_s^s$ therefore yield, for every $p\ge1$,
\[
  \E_{\mathbb W_T}\left[\sup_{s\le t\le T}\left(\|\widehat{\overline\gamma}_t^s\|^p+\|(\widehat{\overline\gamma}_t^s)^{-1}\|^p\right)\right]<\infty.
\]
Using $\Q^N=\mathbb W_T^{\otimes N}$ and multiplicativity of
the operator norm on tensor products, we obtain
\begin{equation}
  \E_{\Q^N}\left[\sup_{s\le t\le T}\left(\|\widehat{\overline\Gamma}_t^{N,s}\|^p+\|(\widehat{\overline\Gamma}_t^{N,s})^{-1}\|^p\right)\right]<\infty.
  \label{eq:local-reference-moments}
\end{equation}
We next identify the equation and mean of $\Pi^s$ under
$\mathbf q_T^{Y,s}$.
For every bounded $\cY_T$-measurable random variable $F$,
Bayes' formula gives
\[
  \E_{\mathbf q_T^{Y,s}}[F\mid\cY_s]=\E_{\mathbb W_T}\left[\frac{\overline Z_T^s}{\overline Z_s^s}F\,\middle|\,\cY_s\right].
\]
Define
\[
  \overline W_t^{s,k}:=Y_t^k-Y_s^k-\sqrt\eta\int_s^t b^{(k)}(\Pi_u^s)\,\d u,\qquad s\le t\le T.
\]
By Girsanov's theorem,
$(\overline W_{s+r}^{s,k})_{0\le r\le T-s}$, $1\le k\le m$,
are independent Brownian motions under $\mathbf q_T^{Y,s}$
with respect to $(\cY_{s+r})_{0\le r\le T-s}$.
Applying It\^o's formula to the quotient defining $\Pi_t^s$ gives
\[
  \d\Pi_t^s=\mathfrak l_t^s(\Pi_t^s)\,\d t+\sqrt\eta\sum_{k=1}^m\cE_{L^{(k)}}(\Pi_t^s)\,\d\overline W_t^{s,k}.
\]
Moreover,
\begin{equation}
  \E_{\mathbf q_T^{Y,s}}[\Pi_t^s]=\E_{\mathbb W_T}[\overline Z_t^s\Pi_t^s]=\E_{\mathbb W_T}[\widehat{\overline\gamma}_t^s]=\zeta_t,\qquad s\le t\le T.
  \label{eq:local-reference-mean}
\end{equation}
For the last equality, the moment bounds above justify taking
expectations in \eqref{eq:local-reference-linear}.
Both $t\mapsto\E_{\mathbb W_T}[\widehat{\overline\gamma}_t^s]$
and $\zeta_t$ solve $\dot x_t=\mathfrak l_t^s(x_t)$ with initial
value $\zeta_s$, so uniqueness gives the equality.

We now estimate $\|\log(\Pi_t^s)\|$ on $[s,T_s]$.
Put $Q_t:=-\log(\det(\Pi_t^s))$ and define
\[
  \mathcal M_t:=-\sqrt\eta\sum_{k=1}^m\int_s^t\left(\Tr(A^{(k)})-d\,b^{(k)}(\Pi_u^s)\right)\,\d\overline W_u^{s,k}.
\]
The calculation in the proof of
Lemma~\ref{lem:faithful-log-moment} gives constants $C,v\ge0$,
depending only on $d,L^{(1)},\ldots,L^{(m)}$, such that
\[
  Q_t\le Q_s+C(t-s)+\mathcal M_t,\qquad \langle\mathcal M\rangle_t\le v(t-s),\qquad s\le t\le T_s.
\]
These bounds hold for every deterministic Hamiltonian
$h+V^{\zeta_t}$.
Since $\Pi_t^s$ is a positive definite density matrix,
$\|\log(\Pi_t^s)\|\le Q_t$, and
\eqref{eq:local-reference-floor} gives $Q_s\le d\log(d/a)$.
Consequently, with $q_a:=d\log(d/a)+C$ and
$S_s:=\sup_{s\le u\le T_s}\mathcal M_u$,
\[
  \sup_{s\le u\le T_s}\|\log(\Pi_u^s)\|\le q_a+S_s.
\]
Choose
\[
  C_1:=4\sqrt{1+\bigl(d(1+\log(d))+C\bigr)^2+v},\qquad \kappa_a:=C_1(1+|\log(a)|).
\]
Then $\kappa_a^2\ge16(1+q_a^2+v)$. The exponential-martingale estimate in the proof of Lemma~\ref{lem:faithful-log-moment}, applied on $[s,T_s]$ with $\langle\mathcal M\rangle_{T_s}\le v$, gives
\[
  \E_{\mathbf q_T^{Y,s}}[e^{2S_s^2/\kappa_a^2}]\le(1-4v/\kappa_a^2)^{-1}\le\frac43.
\]
For $v=0$, this follows directly from $\mathcal M=0$.
Together with $\sup_{s\le u\le T_s}\|\log(\Pi_u^s)\|\le q_a+S_s$ and $(q_a+S_s)^2\le2q_a^2+2S_s^2$, this proves
\begin{equation}
  \E_{\mathbf q_T^{Y,s}}\left[\exp\left(\frac{\sup_{s\le u\le T_s}\|\log(\Pi_u^s)\|^2}{\kappa_a^2}\right)\right]\le\frac43e^{1/8}<2.
  \label{eq:local-log-moment}
\end{equation}
To estimate $\recordentropy_N^s$, set
$\overline H_t^{N,s}:=\sum_{j=1}^N(h+V^{\zeta_t})_j$ and
$\Delta H_t^{N,s}:=H^N-\overline H_t^{N,s}$, and define
\[
  K_N^s(t):=\i[\log(\overline\Gamma_t^{N,s,N}),\Delta H_t^{N,s}].
\]
For $X\in\cL(\HH^N)$, put
\[
  \overline{\mathfrak L}_{N,t}^s(X):=-\i[\overline H_t^{N,s},X]+\sum_{k=1}^m\sum_{j=1}^N\cD_{L_j^{(k)}}(X).
\]
Under $\Q^N$, the product rule in the definition of
$\widehat{\overline\Gamma}^{N,s}$ gives
\[
  \d\widehat{\overline\Gamma}_t^{N,s}=\overline{\mathfrak L}_{N,t}^s(\widehat{\overline\Gamma}_t^{N,s})\,\d t+\sqrt\eta\sum_{k=1}^m\sum_{j=1}^N\cC_{L_j^{(k)}}(\widehat{\overline\Gamma}_t^{N,s})\,\d Y_t^{k,j}.
\]
Here the quadratic covariations between coordinates with different
particle labels vanish.
The interacting drift satisfies
$\mathfrak L_N(X)=\overline{\mathfrak L}_{N,t}^s(X)
-\i[\Delta H_t^{N,s},X]$.
Thus the two equations have the form used in the proof of
Proposition~\ref{prop:H-production}, with
$\overline H_t^{N,s}$ in place of $\overline H_t^N$.

For $\rho_0^N>0$, Lemma~\ref{lem:CP-reference} and \eqref{eq:local-reference-moments} provide the norm and inverse-norm moments used in the proof of Proposition~\ref{prop:H-production}.
The inequality \eqref{eq:generator-sign} holds for every pair of positive definite matrices and applies with $\overline H^{N,s}$ in place of $\overline H^N$.
The same eighth-moment estimates therefore justify taking expectations in It\^o's formula on $[s,t]$.

For a possibly singular $\rho_0^N$, use $\rho_0^{N,\delta}:=(1-\delta)\rho_0^N+\delta I_{\HH^N}/d^N$ and $\widehat\rho_t^{N,\delta}:=\Phi_t^{N,\eta}(\rho_0^{N,\delta})$, as in that proof.
The bound $0\le\widehat\rho_t^{N,\delta}\le\Phi_t^{N,\eta}(I_{\HH^N})$, together with \eqref{eq:local-reference-moments}, gives the same dominating bounds for the entropy at $s,t$ and for the integral containing $K_N^s$.
Passing to the limit $\delta\downarrow0$ and using \eqref{eq:entropy-change-of-measure}, we obtain
\begin{equation}
  \recordentropy_N^s(t)\le\recordentropy_N^s(s)+\int_s^t\E_{\Q^{N,\eta}}\left[\Tr(\rho_u^N K_N^s(u))\right]\,\d u.
  \label{eq:local-entropy-production}
\end{equation}
The bound from Subsection~\ref{subsec:entropy-roadmap} and the same moment estimates also give
\[
  \E_{\Q^N}\left[\sup_{s\le t\le T}\grel(\widehat\rho_t^N\mid\widehat{\overline\Gamma}_t^{N,s})\right]<\infty.
\]
For $\lambda>0$ and $s\le t\le T$, define
\[
  \mathcal Z_N^s(\lambda,t):=\E_{\overline\Q^{N,s}}\left[\Tr\exp\bigl(\log(\overline\Gamma_t^{N,s,N})+\lambda K_N^s(t)\bigr)\right].
\]
Set $\lambda_a:=c_*/\kappa_a$.
By \eqref{eq:local-reference-mean},
$\E_{\mathbf q_T^{Y,s}}[\Pi_t^s]=\zeta_t$, and
$\overline\Q^{N,s}=(\mathbf q_T^{Y,s})^{\otimes N}$.
The last assertion of Proposition~\ref{prop:partition} therefore
applies with $I=[s,T_s]$, $\Pi^s$ and $\mathbf q_T^{Y,s}$,
using \eqref{eq:local-log-moment}.
It gives
\[
  \sup_{N\ge1}\sup_{s\le t\le T_s}\mathcal Z_N^s(\lambda_a,t)\le2e.
\]
Both $\widehat\rho_t^N$ and
$\widehat{\overline\Gamma}_t^{N,s}$ have expected trace one
under $\Q^N$.
Their density identities and Lemma~\ref{lem:cq-gibbs},
applied under $\Q^N$ as in
\eqref{eq:entropy-exponential-reduction}, yield
\[
  \E_{\Q^{N,\eta}}\left[\Tr(\rho_t^N K_N^s(t))\right]\le\lambda_a^{-1}\left(\recordentropy_N^s(t)+\log(\mathcal Z_N^s(\lambda_a,t))\right).
\]
Consequently,
\[
  \E_{\Q^{N,\eta}}\left[\Tr(\rho_t^N K_N^s(t))\right]\le\frac{C_1(1+|\log(a)|)}{c_*}\bigl(\recordentropy_N^s(t)+c_0\bigr),\qquad s\le t\le T_s.
\]
Substituting this bound into
\eqref{eq:local-entropy-production} gives
\[
  \recordentropy_N^s(t)+c_0\le\recordentropy_N^s(s)+c_0+\frac{C_1(1+|\log(a)|)}{c_*}\int_s^t\bigl(\recordentropy_N^s(u)+c_0\bigr)\,\d u.
\]
Gronwall's inequality proves \eqref{eq:local-entropy-growth}
with $C_*:=C_1/c_*$.
\end{proof}

Choose
\begin{equation}
  \tau_*:=\min\{1,(2C_*)^{-1}\}.
  \label{eq:tau-choice}
\end{equation}
Since
$\exp(C_*(1+|\log(a)|)\tau_*)\le e^{1/2}a^{-1/2}$,
Proposition~\ref{prop:local-entropy} gives
\begin{equation}
  \sup_{s\le t\le\min\{s+\tau_*,T\}}\bigl(\recordentropy_N^s(t)+c_0\bigr)\le e^{1/2}a^{-1/2}\bigl(\recordentropy_N^s(s)+c_0\bigr).
  \label{eq:local-half-power}
\end{equation}
The length $\tau_*$ is independent of $a,N,s$ and
$\widehat{\overline\gamma}_s^s$.
In particular, a bound
$(\recordentropy_N^s(s)+c_0)/N\le Ca$
implies a bound $Ce^{1/2}a^{1/2}$ for
$(\recordentropy_N^s(t)+c_0)/N$ throughout the next interval.

In Subsection~\ref{subsec:reference-regularization}, we will apply
\eqref{eq:local-half-power} on successive intervals.
At their initial times, the regularization will be implemented
on unnormalized matrices by
$X\mapsto(1-a)X+a\Tr(X)I_{\HH}/d$, which preserves the trace
and gives the normalized lower bound $aI_{\HH}/d$.
We will estimate the initial relative entropy and its increase
at each regularization, and choose the parameters as functions
of $N$ so that the entropy per particle tends to zero.
To complete Theorem~\ref{thm:main-chaotic}(i), we will then use
the marginal comparison from Subsection~\ref{subsec:faithful-completion}
and estimate the difference between the resulting one-particle
processes and the original $\gamma^j$.

\section{Proofs of the main theorems}
\label{sec:extensions}\label{sec:initial-perturbation}

We now prove Theorem~\ref{thm:main-approx} and
Theorem~\ref{thm:main-chaotic}(i).
Subsection~\ref{subsec:initial-stability} establishes a stability estimate
that transfers the exact-product bounds of Sections~\ref{sec:exact}
and~\ref{sec:entropy} to initial states close to
$\gamma_0^{\otimes N}$.
For permutation-invariant initial states and arbitrary $\gamma_0$,
Subsection~\ref{subsec:reference-regularization} applies the local entropy
estimate of Subsection~\ref{subsec:local-entropy} after successive
regularizations, and controls the resulting error.
The proofs are completed in Subsection~\ref{subsec:main-proofs}.
Theorem~\ref{thm:main-chaotic}(ii) has already been proved in
Subsection~\ref{subsec:faithful-completion}.

\subsection{Uniform stability under changes of the interacting initial state}
\label{subsec:initial-stability}

Fix $T>0$, $0<\eta\le1$, and $\gamma_0\in\cS(\HH)$.
For $\alpha_0^N,\beta_0^N\in\cS(\HH^N)$, use the canonical construction
of Subsection~\ref{subsec:common-canonical} with
$\theta\in\{\alpha,\beta\}$.
In particular, the two linear equations use the same Wiener coordinates,
whereas their physical laws are $\Q^{N,\alpha}$ and $\Q^{N,\beta}$.
Set
\[
  d_N:=\|\alpha_0^N-\beta_0^N\|_1,\qquad \rho_t^{N:n,\theta}:=\Tr_{[N]\setminus[n]}(\rho_t^{N,\theta}),\qquad 1\le n\le N.
\]
The limiting initial state $\gamma_0$, and hence its deterministic mean
$\xi$, are the same for both comparisons.
No symmetry or faithfulness is assumed in this subsection.

To avoid a particle-number-dependent estimate for the nonlinear equation,
we compare the two initial conditions through the linear Zakai equation.
Its solution map is positive and preserves the trace in expectation.
These two structural properties yield an expected trace-norm contraction
with constant one, independently of the number of particles, the number
of noise terms and the dimension of $\HH^N$.

\begin{lemma}
\label{lem:linear-contraction-final}
For every $t\in[0,T]$,
\begin{equation}
  \E_{\Q^N}\left[\|\widehat\rho_t^{N,\alpha}-\widehat\rho_t^{N,\beta}\|_1\right]\le d_N.
  \label{eq:linear-contraction-final}
\end{equation}
Moreover, if $B,C\ge0$ are nonzero, with $b:=\Tr(B)$ and $c:=\Tr(C)$,
then
\begin{equation}
  b\left\|\frac{B}{b}-\frac{C}{c}\right\|_1\le2\|B-C\|_1.
  \label{eq:normalization-inequality}
\end{equation}
\end{lemma}

\begin{proof}
Let $\alpha_0^N-\beta_0^N=D_+-D_-$ be the Jordan decomposition, where
$D_+,D_-\ge0$ and $D_+D_-=0$.
By linearity,
\[
  \widehat\rho_t^{N,\alpha}-\widehat\rho_t^{N,\beta}=\Phi_t^{N,\eta}(D_+)-\Phi_t^{N,\eta}(D_-).
\]
The positivity of $\Phi_t^{N,\eta}$ thus gives
\[
  \|\widehat\rho_t^{N,\alpha}-\widehat\rho_t^{N,\beta}\|_1\le\Tr\bigl(\Phi_t^{N,\eta}(D_+)\bigr)+\Tr\bigl(\Phi_t^{N,\eta}(D_-)\bigr).
\]
Taking expectations and applying \eqref{eq:reference-mean-trace},
we obtain
\[
  \E_{\Q^N}\left[\|\widehat\rho_t^{N,\alpha}-\widehat\rho_t^{N,\beta}\|_1\right]\le\Tr(D_+)+\Tr(D_-)=d_N.
\]
For the normalization inequality, we have
\[
  b\left\|\frac{B}{b}-\frac{C}{c}\right\|_1=\left\|B-\frac{b}{c}C\right\|_1\le\|B-C\|_1+\left|1-\frac{b}{c}\right|\|C\|_1=\|B-C\|_1+|b-c|.
\]
Since $|b-c|\le\|B-C\|_1$, the result follows.
\end{proof}

Apply \eqref{eq:normalization-inequality} with
$B=\widehat\rho_t^{N,\alpha}$ and $C=\widehat\rho_t^{N,\beta}$.
The time-$t$ likelihood identity gives directly
\[
  \E_{\Q^{N,\alpha}}\left[\|\rho_t^{N,\alpha}-\rho_t^{N,\beta}\|_1\right]=\E_{\Q^N}\left[Z_t^{N,\alpha}\|\rho_t^{N,\alpha}-\rho_t^{N,\beta}\|_1\right]\le2\E_{\Q^N}\left[\|\widehat\rho_t^{N,\alpha}-\widehat\rho_t^{N,\beta}\|_1\right]\le2d_N.
\]
Together with partial-trace contraction, this proves, for
$n=1,\ldots,N$,
\begin{equation}
  \E_{\Q^{N,\alpha}}\left[\|\rho_t^{N,\alpha}-\rho_t^{N,\beta}\|_1\right]\le2d_N,\qquad \E_{\Q^{N,\alpha}}\left[\|\rho_t^{N:n,\alpha}-\rho_t^{N:n,\beta}\|_1\right]\le2d_N.
  \label{eq:uniform-N-stability}
\end{equation}
Under $\Q^{N,\alpha}$, the process $\rho^{N,\beta}$ in this estimate
is the normalized solution of the same linear reference equation,
evaluated along the same observation paths, with initial condition
$\beta_0^N$.
Its innovations are those associated with $\beta$, rather than those
associated with $\alpha$.

Since $Z^{N,\theta}$ is a martingale, the density formula for total
variation gives
\begin{equation}
  d_{\mathrm{TV}}\left(\left.\Q^{N,\alpha}\right|_{\cY_t^N},\left.\Q^{N,\beta}\right|_{\cY_t^N}\right)=\frac12\E_{\Q^N}\left[|Z_t^{N,\alpha}-Z_t^{N,\beta}|\right]\le\frac12d_N.
  \label{eq:perturbation-TV-bound}
\end{equation}
Indeed,
\[
  |Z_t^{N,\alpha}-Z_t^{N,\beta}|=\left|\Tr\bigl(\widehat\rho_t^{N,\alpha}-\widehat\rho_t^{N,\beta}\bigr)\right|\le\|\widehat\rho_t^{N,\alpha}-\widehat\rho_t^{N,\beta}\|_1,
\]
and Lemma~\ref{lem:linear-contraction-final} applies.
We use the convention for $d_{\mathrm{TV}}$ from
Subsection~\ref{subsec:eta-conditional-factorization}.
In particular, for a $\cY_t^N$-measurable real function $F$ with
$a\le F\le b$,
\[
  |\E_{\Q^{N,\alpha}}[F]-\E_{\Q^{N,\beta}}[F]|\le(b-a)\,d_{\mathrm{TV}}\left(\left.\Q^{N,\alpha}\right|_{\cY_t^N},\left.\Q^{N,\beta}\right|_{\cY_t^N}\right).
\]
The estimates \eqref{eq:uniform-N-stability} and
\eqref{eq:perturbation-TV-bound} contain no factor depending on
$\dim(\HH^N)$, $mN$, or the norm of the $N$-particle generator.

\medskip

To transfer the propagation-of-chaos error, we must also compare the two
limiting products.
Each product is driven by the innovations of its own interacting system.
For $\theta\in\{\alpha,\beta\}$, let $\gamma^{N,\theta,j}$ solve the
limiting equation under $\Q^{N,\theta}$ with initial state $\gamma_0$
and Brownian motions $W^{N,\theta,k,j}$ from
\eqref{eq:physical-innovation-family}.
Equivalently, on the common observation space,
\begin{equation}
  \d\gamma_t^{N,\theta,j}=\mathfrak l_t(\gamma_t^{N,\theta,j})\,\d t+\sqrt\eta\sum_{k=1}^m\cE_{L^{(k)}}(\gamma_t^{N,\theta,j})\left(\d Y_t^{k,j}-\sqrt\eta\,b_j^{N,k}(\rho_t^{N,\theta})\,\d t\right),\qquad \gamma_0^{N,\theta,j}=\gamma_0.
  \label{eq:perturbation-gamma-observation-form}
\end{equation}
Set
\[
  \Gamma_t^{N,\theta,n}:=\bigotimes_{j=1}^n\gamma_t^{N,\theta,j}.
\]
Under $\Q^{N,\theta}$, the processes
$\gamma^{N,\theta,1},\ldots,\gamma^{N,\theta,N}$ are independent copies
of the limiting one-particle equation and have mean $\xi_t$.
The two physical laws are equivalent on $\cY_T^N$.
Consequently these processes can be evaluated under either law,
remain state valued, and satisfy the same observation-form identities
there.

\begin{lemma}
\label{lem:MF-comparison-final}
For every $T>0$, there exists $C_T<\infty$, independent of $N$, $j$,
$\eta\in(0,1]$, and the initial states, such that
\begin{equation}
  \sup_{0\le t\le T}\E_{\Q^{N,\alpha}}\left[\|\gamma_t^{N,\alpha,j}-\gamma_t^{N,\beta,j}\|_1\right]\le C_Td_N.
  \label{eq:MF-comparison-final}
\end{equation}
\end{lemma}

\begin{proof}
For $1\le k\le m$, set
\[
  \delta b_t^{N,k,j}:=b_j^{N,k}(\rho_t^{N,\alpha})-b_j^{N,k}(\rho_t^{N,\beta}).
\]
The innovation processes satisfy
\[
  \d W_t^{N,\beta,k,j}=\d W_t^{N,\alpha,k,j}+\sqrt\eta\,\delta b_t^{N,k,j}\,\d t.
\]
Thus, under $\Q^{N,\alpha}$, the equation for $\gamma^{N,\beta,j}$
has the additional drift
$\eta\sum_k\cE_{L^{(k)}}(\gamma_t^{N,\beta,j})\delta b_t^{N,k,j}$.
Let
$\Delta\gamma_t^j:=\gamma_t^{N,\alpha,j}-\gamma_t^{N,\beta,j}$.
Subtracting the two equations gives
\[
  \d\Delta\gamma_t^j=\mathfrak l_t(\Delta\gamma_t^j)\,\d t+\sqrt\eta\sum_{k=1}^m\left(\cE_{L^{(k)}}(\gamma_t^{N,\alpha,j})-\cE_{L^{(k)}}(\gamma_t^{N,\beta,j})\right)\d W_t^{N,\alpha,k,j}-\eta\sum_{k=1}^m\cE_{L^{(k)}}(\gamma_t^{N,\beta,j})\delta b_t^{N,k,j}\,\d t.
\]
Since $\xi_t\in\cS(\HH)$, the linear maps $\mathfrak l_t$ are uniformly
bounded.
Moreover, the maps $\cE_{L^{(k)}}$ are bounded and Lipschitz on
the compact state space $\cS(\HH)$, with constants depending only on
the fixed one-particle coefficients and $d$.

For $\varepsilon>0$, define
$f_\varepsilon(D):=(\|D\|_2^2+\varepsilon^2)^{1/2}$ on the real Hilbert
space of self-adjoint matrices with the Hilbert--Schmidt inner product.
Its gradient has norm at most one and
\[
  \mathrm D^2f_\varepsilon(D)[U,U]=\frac{\|U\|_2^2}{f_\varepsilon(D)}-\frac{\langle D,U\rangle_{\mathrm{HS}}^2}{f_\varepsilon(D)^3}\le\frac{\|U\|_2^2}{f_\varepsilon(D)}.
\]
The diffusion difference is Lipschitz in $D$, and
$\|D\|_2^2/f_\varepsilon(D)\le f_\varepsilon(D)$.
It\^o's formula, followed by expectation, therefore gives
\[
  \E_{\Q^{N,\alpha}}[f_\varepsilon(\Delta\gamma_t^j)]\le\varepsilon+C\int_0^t\E_{\Q^{N,\alpha}}[f_\varepsilon(\Delta\gamma_s^j)]\,\d s+C\sum_{k=1}^m\int_0^t\E_{\Q^{N,\alpha}}[|\delta b_s^{N,k,j}|]\,\d s.
\]
By \eqref{eq:uniform-N-stability} and partial-trace contraction,
\[
  \E_{\Q^{N,\alpha}}[|\delta b_s^{N,k,j}|]\le2\|A^{(k)}\|d_N.
\]
Gronwall's lemma yields
\[
  \sup_{0\le t\le T}\E_{\Q^{N,\alpha}}[f_\varepsilon(\Delta\gamma_t^j)]\le C_T(\varepsilon+d_N).
\]
Letting $\varepsilon\downarrow0$ and using
$\|D\|_1\le\sqrt d\,\|D\|_2$ proves the result.
\end{proof}

We can now formulate the transfer estimate.

\begin{proposition}
\label{prop:uniform-transfer-final}
For every fixed $n\ge1$ and $T>0$, there exists $C_{T,n}<\infty$,
independent of $N\ge n$, the initial states, and $\eta\in(0,1]$,
such that
\begin{equation}
  \sup_{0\le t\le T}\E_{\Q^{N,\alpha}}\left[\|\rho_t^{N:n,\alpha}-\Gamma_t^{N,\alpha,n}\|_1\right]\le\sup_{0\le t\le T}\E_{\Q^{N,\beta}}\left[\|\rho_t^{N:n,\beta}-\Gamma_t^{N,\beta,n}\|_1\right]+C_{T,n}d_N.
  \label{eq:uniform-transfer-final}
\end{equation}
\end{proposition}

\begin{proof}
For every $t\in[0,T]$, the triangle inequality gives
\[
  \|\rho_t^{N:n,\alpha}-\Gamma_t^{N,\alpha,n}\|_1\le\|\rho_t^{N:n,\alpha}-\rho_t^{N:n,\beta}\|_1+\|\rho_t^{N:n,\beta}-\Gamma_t^{N,\beta,n}\|_1+\|\Gamma_t^{N,\beta,n}-\Gamma_t^{N,\alpha,n}\|_1.
\]
Under $\Q^{N,\alpha}$, the first term has expectation at most $2d_N$
by \eqref{eq:uniform-N-stability}.
The second term is $\cY_t^N$-measurable and belongs to $[0,2]$.
Changing its expectation from $\Q^{N,\alpha}$ to $\Q^{N,\beta}$
therefore costs at most $d_N$ by \eqref{eq:perturbation-TV-bound}.
Finally, tensor telescoping gives
\[
  \|\Gamma_t^{N,\beta,n}-\Gamma_t^{N,\alpha,n}\|_1\le\sum_{j=1}^n\|\gamma_t^{N,\beta,j}-\gamma_t^{N,\alpha,j}\|_1,
\]
whose expectation is at most $nC_Td_N$ by
Lemma~\ref{lem:MF-comparison-final}.
Taking the supremum over $t\in[0,T]$ proves
\eqref{eq:uniform-transfer-final}.
\end{proof}

The transfer estimate keeps the dependence on $d_N$ linear.
It can be applied to any available bound for the comparison system.
In particular, choosing $\beta_0^N=\gamma_0^{\otimes N}$ will allow
the two exact-product estimates from Sections~\ref{sec:exact}
and~\ref{sec:entropy} to be used without changing this argument.

\subsection{Regularization and entropy estimates}
\label{subsec:reference-regularization}

Fix $T>0$ and $0<\eta\le1$.
Throughout this subsection, $N\ge2$, $\rho_0^N$ is permutation
invariant, and $\gamma_0\in\cS(\HH)$ is arbitrary.
We apply Proposition~\ref{prop:local-entropy} on successive intervals,
regularizing the one-particle matrix at their initial times.
The first proposition below controls the initial entropy, its increase
at each regularization, and its growth between regularizations.
The second converts these bounds into an estimate for
$\E_{\Q^{N,\eta}}[\|\rho_t^{N:n}-\Gamma_t^n\|_1]$.

For $0<\varepsilon\le1/2$, define
\begin{equation}
  \mathscr D_\varepsilon(X):=(1-\varepsilon)X+\varepsilon\Tr(X)I_{\HH}/d,\qquad X\in\cL(\HH).
  \label{eq:depolarizing-map}
\end{equation}
This map is linear, completely positive, and trace preserving.
For $R\in\cS(\HH)$,
$\mathscr D_\varepsilon(R)\ge\varepsilon I_{\HH}/d$ and
$\|\mathscr D_\varepsilon(R)-R\|_1\le2\varepsilon$.

Let $\tau_*$ be defined in \eqref{eq:tau-choice}, and put
$M:=\lceil T/\tau_*\rceil$ and
$t_\ell:=\min\{\ell\tau_*,T\}$ for $0\le\ell\le M$.
Set $I_\ell:=[t_\ell,t_{\ell+1})$ for $0\le\ell<M-1$ and
$I_{M-1}:=[t_{M-1},T]$.
Choose deterministic parameters
$\varepsilon_0,\ldots,\varepsilon_{M-1}\in(0,1/2]$.
The superscript $\varepsilon$ denotes this family of parameters.

Let $\xi^\varepsilon$ solve \eqref{eq:mean-state-equation}
between the times $t_\ell$, with
$\xi_0^\varepsilon:=\mathscr D_{\varepsilon_0}(\gamma_0)$ and
\begin{equation}
  \xi_{t_\ell}^\varepsilon:=\mathscr D_{\varepsilon_\ell}(\xi_{t_\ell-}^\varepsilon),\qquad 1\le\ell<M.
  \label{eq:updated-mean}
\end{equation}
Define
\[
  \mathfrak l_t^\varepsilon(X):=-\i[h+V^{\xi_t^\varepsilon},X]+\sum_{k=1}^m\cD_{L^{(k)}}(X),\qquad X\in\cL(\HH).
\]
On $(\mathsf Y_T,\cY_T,\mathbb W_T)$, let
$\widehat{\overline\gamma}^\varepsilon(Y)$ solve, between the same
times,
\begin{equation}
  \d\widehat{\overline\gamma}_t^\varepsilon=\mathfrak l_t^\varepsilon(\widehat{\overline\gamma}_t^\varepsilon)\,\d t+\sqrt\eta\sum_{k=1}^m\cC_{L^{(k)}}(\widehat{\overline\gamma}_t^\varepsilon)\,\d Y_t^k,
  \label{eq:updated-reference-SDE}
\end{equation}
with $\widehat{\overline\gamma}_0^\varepsilon=\xi_0^\varepsilon$ and
\begin{equation}
  \widehat{\overline\gamma}_{t_\ell}^\varepsilon:=\mathscr D_{\varepsilon_\ell}(\widehat{\overline\gamma}_{t_\ell-}^\varepsilon),\qquad 1\le\ell<M.
  \label{eq:updated-reference-jump}
\end{equation}
Lemma~\ref{lem:CP-reference}, applied on each interval, gives
$\widehat{\overline\gamma}_t^\varepsilon>0$.
Set
\[
  \overline Z_t^\varepsilon(Y):=\Tr(\widehat{\overline\gamma}_t^\varepsilon(Y)),\qquad \Pi_t^\varepsilon(Y):=\frac{\widehat{\overline\gamma}_t^\varepsilon(Y)}{\overline Z_t^\varepsilon(Y)}.
\]
On $\Omega_N$, define
\[
  \widehat{\overline\Gamma}_t^{N,\varepsilon}:=\widehat{\overline\gamma}_t^\varepsilon(Y^1)\otimes\cdots\otimes\widehat{\overline\gamma}_t^\varepsilon(Y^N),
\]
and
\begin{equation}
  \recordentropy_N^\varepsilon(t):=\E_{\Q^N}\left[\grel(\widehat\rho_t^N\mid\widehat{\overline\Gamma}_t^{N,\varepsilon})\right].
  \label{eq:updated-joint-entropy}
\end{equation}
Recall that $c_0=\log(2e)$.

\begin{proposition}
\label{prop:updated-entropy}
For every $N\ge2$,
\begin{equation}
  \frac{\recordentropy_N^\varepsilon(0)}N\le C_d\left(\delta_N^2+\frac1N\right)^{1/4}+2\varepsilon_0+\delta_N^1\log(d/\varepsilon_0).
  \label{eq:regularized-initial-entropy}
\end{equation}
For $1\le\ell<M$,
\begin{equation}
  \recordentropy_N^\varepsilon(t_\ell)\le\recordentropy_N^\varepsilon(t_\ell-)-N\log(1-\varepsilon_\ell)\le\recordentropy_N^\varepsilon(t_\ell-)+2N\varepsilon_\ell.
  \label{eq:reference-update-cost}
\end{equation}
For $0\le\ell<M$,
\begin{equation}
  \sup_{t\in I_\ell}\bigl(\recordentropy_N^\varepsilon(t)+c_0\bigr)\le e^{1/2}\varepsilon_\ell^{-1/2}\bigl(\recordentropy_N^\varepsilon(t_\ell)+c_0\bigr).
  \label{eq:updated-local-growth}
\end{equation}
Here $\delta_N^1,\delta_N^2$ are defined in
\eqref{eq:weak-tensorization}, and $C_d$ depends only on $d$.
\end{proposition}

\begin{proof}
We first verify the hypotheses needed to apply
Proposition~\ref{prop:local-entropy}.
Since each update preserves the trace,
$\overline Z^\varepsilon$ is continuous and satisfies
\[
  \d\overline Z_t^\varepsilon=\sqrt\eta\,\overline Z_t^\varepsilon\sum_{k=1}^m b^{(k)}(\Pi_{t-}^\varepsilon)\,\d Y_t^k,\qquad \overline Z_0^\varepsilon=1.
\]
The bound $|b^{(k)}(\Pi_{t-}^\varepsilon)|\le\|A^{(k)}\|$
and the exponential-martingale argument of
Subsection~\ref{subsec:common-canonical} give, for every $p>0$,
\[
  \E_{\mathbb W_T}[\overline Z_t^\varepsilon]=1,\qquad \E_{\mathbb W_T}\left[\sup_{0\le t\le T}\left((\overline Z_t^\varepsilon)^p+(\overline Z_t^\varepsilon)^{-p}\right)\right]<\infty.
\]
Taking expectations in the linear equation and its linear updates
also gives
\begin{equation}
  \E_{\mathbb W_T}[\widehat{\overline\gamma}_t^\varepsilon]=\xi_t^\varepsilon,\qquad 0\le t\le T.
  \label{eq:updated-reference-mean}
\end{equation}
Indeed, between updates both sides solve
$\dot x_t=\mathfrak l_t^\varepsilon(x_t)$, and their initial values
and updates agree.
Furthermore,
$\Pi_{t_\ell}^\varepsilon\ge\varepsilon_\ell I_{\HH}/d$
for $0\le\ell<M$, because normalization commutes with
$\mathscr D_{\varepsilon_\ell}$.

The moment argument in the proof of
Proposition~\ref{prop:local-entropy} therefore applies on every
interval.
Together with the bound for $\grel$ in
Subsection~\ref{subsec:entropy-roadmap}, it gives
\[
  \E_{\Q^N}\left[\sup_{0\le t\le T}\grel(\widehat\rho_t^N\mid\widehat{\overline\Gamma}_t^{N,\varepsilon})\right]<\infty.
\]
Thus $\recordentropy_N^\varepsilon$ is continuous between updates,
and its left limits are obtained by taking the corresponding
left limits inside the expectation.

Fix $0\le\ell<M$.
Apply the construction of Subsection~\ref{subsec:local-entropy}
with $s=t_\ell$, $a=\varepsilon_\ell$, and
$\widehat{\overline\gamma}_s^s
=\widehat{\overline\gamma}_{t_\ell}^\varepsilon$.
Equation \eqref{eq:updated-reference-mean} gives
$\zeta_s=\xi_{t_\ell}^\varepsilon$.
Continuing from $s$ without further updates, uniqueness of the mean
equation and of the linear equation gives
$\widehat{\overline\gamma}_t^s
=\widehat{\overline\gamma}_t^\varepsilon$ on $I_\ell$.
Hence $\recordentropy_N^s(t)=\recordentropy_N^\varepsilon(t)$
there.
Since $t_{\ell+1}-t_\ell\le\tau_*$,
\eqref{eq:local-half-power} proves
\eqref{eq:updated-local-growth}.
The same bound holds at $t_{\ell+1}-$ when $\ell<M-1$;
the last interval includes $T$.

At time zero,
$\recordentropy_N^\varepsilon(0)
=D(\rho_0^N\mid(\xi_0^\varepsilon)^{\otimes N})$.
Permutation invariance and the tensor-product logarithm give
\[
  \frac{\recordentropy_N^\varepsilon(0)}N=D(\gamma_0\mid\xi_0^\varepsilon)+S(\gamma_0)-\frac{S(\rho_0^N)}N-\Tr\bigl((\rho_0^{N:1}-\gamma_0)\log(\xi_0^\varepsilon)\bigr).
\]
Since $\xi_0^\varepsilon\ge\varepsilon_0 I_{\HH}/d$,
$\|\log(\xi_0^\varepsilon)\|\le\log(d/\varepsilon_0)$.
Moreover, $\gamma_0$ and $\xi_0^\varepsilon$ commute, and comparison
of their eigenvalues on the support of $\gamma_0$ gives
\[
  D(\gamma_0\mid\xi_0^\varepsilon)\le-\log(1-\varepsilon_0)\le2\varepsilon_0.
\]
Lemma~\ref{lem:initial-entropy} now proves
\eqref{eq:regularized-initial-entropy}.

Finally, fix $1\le\ell<M$ and write
$B_-:=\widehat{\overline\Gamma}_{t_\ell-}^{N,\varepsilon}$ and
$B_+:=\widehat{\overline\Gamma}_{t_\ell}^{N,\varepsilon}$.
The definition of $\mathscr D_{\varepsilon_\ell}$ gives
\[
  \Tr(B_+)=\Tr(B_-),\qquad B_+\ge(1-\varepsilon_\ell)^N B_-.
\]
Operator monotonicity of the logarithm therefore implies
\[
  \grel(\widehat\rho_{t_\ell}^N\mid B_+)-\grel(\widehat\rho_{t_\ell}^N\mid B_-)\le-N\log(1-\varepsilon_\ell)Z_{t_\ell}^N.
\]
Taking expectations, using
$\E_{\Q^N}[Z_{t_\ell}^N]=1$ and the continuity of
$\widehat\rho^N$, proves \eqref{eq:reference-update-cost}.
\end{proof}

We next compare the normalized matrices $\Pi_t^\varepsilon(Y^j)$
directly with the original processes $\gamma_t^j$.
This accounts for both the change from $\xi$ to $\xi^\varepsilon$
and the jumps at the regularization times.

\begin{proposition}
\label{prop:regularized-comparison}
For every fixed $n\ge1$ and every $N\ge\max\{2,n\}$,
\begin{equation}
  \sup_{0\le t\le T}\E_{\Q^{N,\eta}}\left[\|\rho_t^{N:n}-\Gamma_t^n\|_1\right]\le C_{T,n}\left[\left(\frac{\sup_{0\le t\le T}\recordentropy_N^\varepsilon(t)}N\right)^{1/2}+\sum_{\ell=0}^{M-1}\varepsilon_\ell\right].
  \label{eq:regularized-comparison}
\end{equation}
The constant may depend on
$d,h,V,L^{(1)},\ldots,L^{(m)},\eta,T,n$,
but is independent of $N$, the parameters $\varepsilon_\ell$,
and the initial states.
\end{proposition}

\begin{proof}
On $(\Omega_N,\cY_T^N,\Q^{N,\eta})$, set
$\overline\gamma_t^{N,\varepsilon,j}:=\Pi_t^\varepsilon(Y^j)$ and
\[
  \overline\Gamma_t^{N,\varepsilon,n}:=\overline\gamma_t^{N,\varepsilon,1}\otimes\cdots\otimes\overline\gamma_t^{N,\varepsilon,n},\qquad 1\le n\le N.
\]
For $a,b>0$ and states $R,S$ with $S>0$,
\[
  \grel(aR\mid bS)=aD(R\mid S)+a\log(a/b)-a+b\ge aD(R\mid S).
\]
Apply this identity with $a=Z_t^N$,
$b=\Tr(\widehat{\overline\Gamma}_t^{N,\varepsilon})$,
$R=\rho_t^N$, and $S=\overline\Gamma_t^{N,\varepsilon,N}$.
Taking expectation under $\Q^N$ gives
\[
  \E_{\Q^{N,\eta}}\left[D(\rho_t^N\mid\overline\Gamma_t^{N,\varepsilon,N})\right]\le\recordentropy_N^\varepsilon(t).
\]
The same map $\Pi^\varepsilon$ and the same deterministic updates
are used for every coordinate.
Thus the permutation argument and the block estimate in the proof
of \eqref{eq:entropy-E2} apply.
Quantum Pinsker's inequality yields
\begin{equation}
  \E_{\Q^{N,\eta}}\left[\|\rho_t^{N:n}-\overline\Gamma_t^{N,\varepsilon,n}\|_1^2\right]\le\frac{2\recordentropy_N^\varepsilon(t)}{\lfloor N/n\rfloor}.
  \label{eq:updated-block-bound}
\end{equation}

Write $S_\varepsilon:=\sum_{\ell=0}^{M-1}\varepsilon_\ell$.
The vector field in \eqref{eq:mean-state-equation} is Lipschitz
on $\cS(\HH)$, since
$\|V^x-V^y\|\le\|V\|\|x-y\|_1$.
The initial difference between $\xi^\varepsilon$ and $\xi$
has trace norm at most $2\varepsilon_0$, and the jump at
$t_\ell$ has trace norm at most $2\varepsilon_\ell$.
Their integral equations and Gronwall's inequality therefore give
\begin{equation}
  \sup_{0\le t\le T}\|\xi_t^\varepsilon-\xi_t\|_1\le C_T S_\varepsilon.
  \label{eq:updated-mean-comparison}
\end{equation}

Fix $j\in[N]$ and define
\[
  e_t^j:=\overline\gamma_t^{N,\varepsilon,j}-\gamma_t^j,\qquad \Delta b_t^{j,k}:=b_j^{N,k}(\rho_t^N)-b^{(k)}(\overline\gamma_t^{N,\varepsilon,j}).
\]
The case $n=1$ of \eqref{eq:updated-block-bound}, together with
permutation invariance in law and trace duality, gives
\[
  \E_{\Q^{N,\eta}}\left[\sum_{k=1}^m|\Delta b_t^{j,k}|^2\right]\le C\,\E_{\Q^{N,\eta}}\left[\|\rho_t^{N:(j)}-\overline\gamma_t^{N,\varepsilon,j}\|_1^2\right]\le\frac{C}{N}\recordentropy_N^\varepsilon(t).
\]

Normalizing \eqref{eq:updated-reference-SDE} and using
$\d Y_t^{k,j}=\d W_t^{k,j}
+\sqrt\eta\,b_j^{N,k}(\rho_t^N)\,\d t$
shows that, between updates,
\[
  \d e_t^j=a_t^j\,\d t+\sum_{k=1}^mG_t^{j,k}\,\d W_t^{k,j},
\]
where
\[
  a_t^j:=\mathfrak l_t^\varepsilon(e_t^j)-\i[V^{\xi_t^\varepsilon}-V^{\xi_t},\gamma_t^j]+\eta\sum_{k=1}^m\cE_{L^{(k)}}(\overline\gamma_t^{N,\varepsilon,j})\Delta b_t^{j,k},
\]
and
\[
  G_t^{j,k}:=\sqrt\eta\left(\cE_{L^{(k)}}(\overline\gamma_t^{N,\varepsilon,j})-\cE_{L^{(k)}}(\gamma_t^j)\right).
\]
The coefficient estimates in the proof of
Lemma~\ref{lem:eta-local-comparison} imply
\[
  \|a_t^j\|_2^2+\sum_{k=1}^m\|G_t^{j,k}\|_2^2\le C\left(\|e_t^j\|_2^2+\sum_{k=1}^m|\Delta b_t^{j,k}|^2+\|\xi_t^\varepsilon-\xi_t\|_1^2\right).
\]
Here the additional mean term is controlled by
$\|[V^{\xi_t^\varepsilon}-V^{\xi_t},\gamma_t^j]\|_2
\le2\|V\|\|\xi_t^\varepsilon-\xi_t\|_1$.
The constants are independent of $N$ and the regularization
parameters.

The initial error satisfies $\|e_0^j\|_1\le2\varepsilon_0$.
Since $\gamma^j$ is continuous, the jumps are
\[
  e_{t_\ell}^j-e_{t_\ell-}^j=\varepsilon_\ell\left(I_{\HH}/d-\overline\gamma_{t_\ell-}^{N,\varepsilon,j}\right),\qquad 1\le\ell<M.
\]
Their total trace norm is at most
$2\sum_{\ell=1}^{M-1}\varepsilon_\ell$.

Set
$F_j(t):=\E_{\Q^{N,\eta}}[\sup_{0\le u\le t}\|e_u^j\|_2^2]$.
Apply Cauchy--Schwarz to the drift integral and the $L^2$
martingale maximal inequality to the stochastic integral, as in
the proof of Lemma~\ref{lem:eta-local-comparison}.
The integral equation includes the preceding jumps, and stochastic
integrands are evaluated at their left limits.
Using the bounds on $\Delta b^{j,k}$,
\eqref{eq:updated-mean-comparison}, and the total jump contribution,
we obtain
\[
  F_j(t)\le C_T\left(S_\varepsilon^2+\int_0^tF_j(u)\,\d u+\frac1N\int_0^t\recordentropy_N^\varepsilon(u)\,\d u\right).
\]
Gronwall's inequality and $\|A\|_1\le\sqrt d\,\|A\|_2$
therefore give
\begin{equation}
  \E_{\Q^{N,\eta}}\left[\sup_{0\le t\le T}\|\overline\gamma_t^{N,\varepsilon,j}-\gamma_t^j\|_1^2\right]\le C_T\left(S_\varepsilon^2+\frac1N\int_0^T\recordentropy_N^\varepsilon(u)\,\d u\right).
  \label{eq:updated-local-comparison}
\end{equation}

Finally,
$\|\overline\Gamma_t^{N,\varepsilon,n}-\Gamma_t^n\|_1
\le\sum_{j=1}^n\|e_t^j\|_1$.
Combining this inequality with
\eqref{eq:updated-block-bound} and
\eqref{eq:updated-local-comparison}, and using Cauchy--Schwarz
and $\lfloor N/n\rfloor\ge N/(2n)$, proves
\eqref{eq:regularized-comparison}.
\end{proof}

In Subsection~\ref{subsec:main-proofs}, we choose the parameters
$\varepsilon_\ell$ as functions of $\delta_N^2+N^{-1}$.
Iterating \eqref{eq:updated-local-growth} with the update cost
\eqref{eq:reference-update-cost}, and then applying
\eqref{eq:regularized-comparison}, will prove
Theorem~\ref{thm:main-chaotic}(i).

\subsection{Remaining proofs}
\label{subsec:main-proofs}
\label{subsec:combined-transfer}

We use the canonical realization of the coupled equations from
Subsection~\ref{subsec:common-canonical}.

\begin{proof}[Proof of Theorem~\ref{thm:main-approx}]
Apply Proposition~\ref{prop:uniform-transfer-final} with
$\alpha_0^N:=\rho_0^N$ and
$\beta_0^N:=\gamma_0^{\otimes N}$.
Then $d_N=\delta_N$, and the construction for $\alpha$ is the
one denoted by $\rho^N$, $\Gamma^n$, and $\Q^{N,\eta}$.
Thus
\[
  \sup_{0\le t\le T}\E_{\Q^{N,\eta}}\left[\|\rho_t^{N:n}-\Gamma_t^n\|_1\right]\le C_{T,n}\delta_N+\sup_{0\le t\le T}\E_{\Q^{N,\beta}}\left[\|\rho_t^{N:n,\beta}-\Gamma_t^{N,\beta,n}\|_1\right].
\]

For part \textup{(i)}, Propositions~\ref{prop:main-exact-eta-one}
and~\ref{prop:main-exact-eta-less-one}, applied to the system
with initial state $\beta_0^N$, give
\[
  \sup_{0\le t\le T}\E_{\Q^{N,\beta}}\left[\|\rho_t^{N:n,\beta}-\Gamma_t^{N,\beta,n}\|_1\right]\le C_{\eta,T,n}N^{-1/4}.
\]
Substitution proves \eqref{eq:main-approx-bound}.

For part \textup{(ii)}, assume $\gamma_0>0$.
The state $\beta_0^N=\gamma_0^{\otimes N}$ is permutation invariant
and satisfies $D(\beta_0^N\mid\gamma_0^{\otimes N})=0$.
Proposition~\ref{prop:faithful-entropy} therefore gives
\[
  \sup_{0\le t\le T}\E_{\Q^{N,\beta}}\left[\|\rho_t^{N:n,\beta}-\Gamma_t^{N,\beta,n}\|_1\right]\le C_{T,n,\gamma_0}N^{-1/2}.
\]
Substituting this estimate proves
\eqref{eq:main-approx-faithful-bound}.
\end{proof}

\begin{proof}[Proof of Theorem~\ref{thm:main-chaotic}(i)]
Set $r_N:=\delta_N^2+N^{-1}$ and
$M:=\lceil T/\tau_*\rceil$.
First suppose that
$r_N\le2^{-2^{M+1}}$, and choose
\begin{equation}
  \varepsilon_\ell:=r_N^{\,2^{-\ell}/4},\qquad 0\le\ell<M.
  \label{eq:regularization-parameters}
\end{equation}
Then $0<\varepsilon_\ell\le1/2$ for every $\ell$, and
$\varepsilon_\ell=\varepsilon_{\ell-1}^{1/2}$ for $1\le\ell<M$.

Partial-trace contraction gives
$\delta_N^1\le\delta_N^2\le r_N$, and $N^{-1}\le r_N$.
Since $r|\log(r)|\le Cr^{1/4}$ for $0<r\le1$,
\eqref{eq:regularized-initial-entropy} implies
\[
  \frac{\recordentropy_N^\varepsilon(0)+c_0}{N}\le C\varepsilon_0.
\]
For $0\le\ell<M$, define
\[
  \mathfrak a_\ell:=\sup_{t\in I_\ell}\frac{\recordentropy_N^\varepsilon(t)+c_0}{N}.
\]
Equations \eqref{eq:reference-update-cost} and
\eqref{eq:updated-local-growth} yield
\[
  \mathfrak a_0\le C\varepsilon_0^{1/2},\qquad \mathfrak a_\ell\le e^{1/2}\varepsilon_\ell^{-1/2}\bigl(\mathfrak a_{\ell-1}+2\varepsilon_\ell\bigr),\qquad 1\le\ell<M.
\]
In using the second inequality, the bound defining
$\mathfrak a_{\ell-1}$ also applies to the left limit at
$t_\ell$.

We claim that
$\mathfrak a_\ell\le C_\ell\varepsilon_\ell^{1/2}$.
The case $\ell=0$ was just proved.
If it holds at $\ell-1$, then
$\varepsilon_{\ell-1}^{1/2}=\varepsilon_\ell$, so
\[
  \mathfrak a_\ell\le e^{1/2}(C_{\ell-1}+2)\varepsilon_\ell^{1/2}.
\]
Induction proves the claim.
Because $0<r_N\le1$ and there are $M$ intervals, it follows that
\begin{equation}
  \sup_{0\le t\le T}\frac{\recordentropy_N^\varepsilon(t)+c_0}{N}\le C_T r_N^{\,2^{-M}/4}.
  \label{eq:regularization-iteration}
\end{equation}

Substituting \eqref{eq:regularization-parameters} and
\eqref{eq:regularization-iteration} into
Proposition~\ref{prop:regularized-comparison} gives
\[
  \sup_{0\le t\le T}\E_{\Q^{N,\eta}}\left[\|\rho_t^{N:n}-\Gamma_t^n\|_1\right]\le C_{T,n}\left(r_N^{\,2^{-M}/8}+\sum_{\ell=0}^{M-1}r_N^{\,2^{-\ell}/4}\right).
\]
For $0\le\ell<M$,
$2^{-\ell}/4\ge2^{-M}/8$.
Hence, after increasing $C_{T,n}$,
\begin{equation}
  \sup_{0\le t\le T}\E_{\Q^{N,\eta}}\left[\|\rho_t^{N:n}-\Gamma_t^n\|_1\right]\le C_{T,n}r_N^{\,2^{-M}/8}.
  \label{eq:regularized-main-bound}
\end{equation}
Since
\[
  2^{-M}/8=2^{-3-\lceil T/\tau_*\rceil}=\vartheta(T),
\]
this is \eqref{eq:main-chaotic-bound} when
$r_N\le2^{-2^{M+1}}$.

For $r_N>2^{-2^{M+1}}$, the left-hand side of
\eqref{eq:regularized-main-bound} is at most $2$.
The threshold depends only on $T$ and the fixed coefficients,
so enlarging $C_{T,n}$ proves the same bound in this case.
Thus \eqref{eq:main-chaotic-bound} holds for every
$N\ge\max\{2,n\}$.
Its constants are independent of the initial states.
Finally, $\delta_N^2\to0$ implies $r_N\to0$, and
$\vartheta(T)>0$ for every fixed $T$, proving propagation of chaos.
\end{proof}

\appendix

\section{Linear and entropy tools}
\label{app:standard}

We prove Lemma~\ref{lem:CP-reference} and establish the entropy
inequalities and scalar exponential estimates used in
Section~\ref{sec:entropy}.

\subsection{Complete positivity and fixed-dimensional moments}
\label{app:zakai}

\begin{proof}[Proof of Lemma~\ref{lem:CP-reference}]
We work on a fixed finite-dimensional Hilbert space $\KK$.
Let $Y^1,\ldots,Y^r$ be independent Brownian motions under a
probability measure $\mathbb P$, let $H_t=H_t^*$ be bounded,
deterministic and continuous, and let
$L_1,\ldots,L_r\in\cL(\KK)$.
Fix $0<\eta\le1$.
For $B\in\cL(\KK)$, consider
\[
  \d X_t=\left(-\i[H_t,X_t]+\sum_{a=1}^r\cD_{L_a}(X_t)\right)\,\d t+\sqrt\eta\sum_{a=1}^r\cC_{L_a}(X_t)\,\d Y_t^a,\qquad X_0=B.
\]
This linear equation has bounded coefficients and therefore admits
a unique continuous strong solution.

Put $C_t:=-\i H_t-\frac12\sum_{a=1}^rL_a^*L_a$, and let
$G$ and $J$ solve
\[
  \d G_t=C_tG_t\,\d t+\sqrt\eta\sum_{a=1}^rL_aG_t\,\d Y_t^a,\qquad G_0=I_{\KK},
\]
and
\[
  \d J_t=J_t\left(-C_t+\eta\sum_{a=1}^rL_a^2\right)\,\d t-\sqrt\eta\sum_{a=1}^rJ_tL_a\,\d Y_t^a,\qquad J_0=I_{\KK}.
\]
It\^o's product rule gives $\d(J_tG_t)=0$.
Indeed, the quadratic covariation contributes
$-\eta\sum_aJ_tL_a^2G_t\,\d t$, which cancels the remaining
drift after the terms containing $C_t$ have canceled.
Thus $J_tG_t=I_{\KK}$ for all $t\in[0,T]$.
Since these are square matrices, $G_t$ is invertible and
$J_t=G_t^{-1}$.
Neither process depends on $B$.

We record the moment bounds for these linear equations.
For $p\ge2$ and $U\in\{G,J,X\}$, write
$f_U(t):=\E_{\mathbb P}[\sup_{0\le s\le t}\|U_s\|_2^p]$.
The integral equation, the Burkholder--Davis--Gundy inequality,
and H\"older's inequality give
\[
  f_U(t)\le C_{p,T}\|U_0\|_2^p+C_{p,T}\int_0^t f_U(s)\,\d s.
\]
These estimates may first be applied to bounded stopped solutions,
with constants independent of the stopping level.
Gronwall's inequality and Fatou's lemma then give the bound
for the original solution.
Smaller positive orders follow by H\"older's inequality.
Consequently,
\[
  \E_{\mathbb P}\left[\sup_{0\le t\le T}\left(\|G_t\|^p+\|J_t\|^p+\|X_t\|^p\right)\right]<\infty,\qquad p>0.
\]
Fix a path on which $G,J$ are continuous and $JG=I_{\KK}$
on $[0,T]$.
Set $M_a(t):=G_t^{-1}L_aG_t$ and consider
\[
  \dot R_t=(1-\eta)\sum_{a=1}^rM_a(t)R_tM_a(t)^*,\qquad R_0=B.
\]
Define $R_t^{(0)}:=B$ and, recursively,
\[
  R_t^{(q+1)}:=(1-\eta)\sum_{a=1}^r\int_0^tM_a(s)R_s^{(q)}M_a(s)^*\,\d s,\qquad q\ge0.
\]
The constant
$K_T:=(1-\eta)\sum_a\sup_{0\le s\le T}\|M_a(s)\|^2$
is finite on the chosen path.
Induction gives
\[
  \sup_{0\le s\le t}\|R_s^{(q)}\|\le\|B\|\frac{(K_Tt)^q}{q!}.
\]
Thus $\sum_{q\ge0}R_t^{(q)}$ converges uniformly on $[0,T]$
in operator norm and solves the equation for $R$.
For every $q$, the map $B\mapsto R_t^{(q)}$ is completely positive:
this follows inductively from conjugation, composition, and
integration of completely positive maps.
The series converges in the norm of linear maps as well, by the
same estimate uniformly over $\|B\|\le1$.
Hence $B\mapsto R_t$ is completely positive.
If $B\ge0$, every term of the series is nonnegative, so $R_t\ge B$.

It\^o's product rule shows that $G_tR_tG_t^*$ solves the equation
for $X_t$.
The quadratic covariation of $G$ and $G^*$ contributes
$\eta\sum_aL_a(G_tR_tG_t^*)L_a^*\,\d t$, while
$G_t\dot R_tG_t^*$ contributes
$(1-\eta)\sum_aL_a(G_tR_tG_t^*)L_a^*$.
Together with $C_tG_tR_tG_t^*+G_tR_tG_t^*C_t^*$,
these are exactly the required drift terms.
The diffusion terms also agree.
Uniqueness therefore gives
$X_t=G_tR_tG_t^*$, proving that $B\mapsto X_t$ is completely
positive.

For $B\ge0$, the inequality $R_t\ge B$ implies
\[
  X_t\ge G_tBG_t^*,\qquad \Tr(X_t)\ge\frac{\Tr(B)}{\|J_t\|^2}.
\]
In particular, $\Tr(X_t)>0$ for all $t\in[0,T]$ whenever
$B\ge0$ and $B\ne0$.
If $B>0$, then $X_t>0$.
Inversion reverses the operator inequality, and hence
\[
  \|X_t^{-1}\|\le\|B^{-1}\|\|J_t\|^2.
\]
The moments of $J$ give
\[
  \E_{\mathbb P}\left[\sup_{0\le t\le T}\|X_t^{-1}\|^p\right]\le\|B^{-1}\|^p\E_{\mathbb P}\left[\sup_{0\le t\le T}\|J_t\|^{2p}\right]<\infty,\qquad p>0.
\]
Finally, taking the trace in the equation for $X$ gives
\[
  \Tr(X_t)=\Tr(B)+\sqrt\eta\sum_{a=1}^r\int_0^t\Tr\bigl((L_a+L_a^*)X_s\bigr)\,\d Y_s^a.
\]
The moment bound with $p=2$ implies
\[
  \E_{\mathbb P}\left[\int_0^T\sum_{a=1}^r|\Tr\bigl((L_a+L_a^*)X_s\bigr)|^2\,\d s\right]<\infty.
\]
The stochastic integrals therefore have expectation zero, so
$\E_{\mathbb P}[\Tr(X_t)]=\Tr(B)$ for every $B\in\cL(\KK)$,
including non-self-adjoint $B$.

Taking $\KK=\HH^N$, the family $(L_a)$ to be
$(L_j^{(k)})_{k,j}$, and $\mathbb P=\Q^N$ proves the lemma,
including the assertion for a deterministic time-dependent
Hamiltonian.
The constants in the moment bounds may depend on $N$.
\end{proof}

\subsection{Entropy and trace inequalities}
\label{app:standard-entropy}

Let $\KK$ be a finite-dimensional Hilbert space.
We use $\grel$ with the support convention in
\eqref{eq:generalized-entropy}.
In particular, $\grel(X\mid Z)\ge0$, and $\grel$ is lower
semicontinuous on pairs of positive semidefinite matrices.
We also use quantum data processing for completely positive
trace-preserving maps; see \cite{Lindblad1975,watrous2018}.

\begin{lemma}
\label{lem:cq-gibbs}
Let $(E,\mathcal A,\nu)$ be a probability space and let
$a,b:E\to\cL^+(\KK)$ be measurable, with $b(x)>0$ for
$\nu$-almost every $x$ and
\[
  \int_E\Tr(a(x))\,\nu(\d x)=\int_E\Tr(b(x))\,\nu(\d x)=1.
\]
Let $K:E\to\cL(\KK)$ be measurable with $K(x)=K(x)^*$, and
assume that
$\int_E|\Tr(a(x)K(x))|\,\nu(\d x)<\infty$.
If $\lambda>0$ satisfies
\[
  \mathcal Z:=\int_E\Tr\left(\exp\bigl(\log(b(x))+\lambda K(x)\bigr)\right)\nu(\d x)<\infty,
\]
then
\begin{equation}
  \lambda\int_E\Tr(a(x)K(x))\,\nu(\d x)\le\int_E\grel(a(x)\mid b(x))\,\nu(\d x)+\log(\mathcal Z).
  \label{eq:cq-gibbs}
\end{equation}
\end{lemma}

\begin{proof}
The integrand defining $\mathcal Z$ is strictly positive, so
$\mathcal Z>0$.
If $\int_E\grel(a\mid b)\,\d\nu=\infty$, the conclusion is
immediate.
Otherwise, define
\[
  c(x):=\mathcal Z^{-1}\exp\bigl(\log(b(x))+\lambda K(x)\bigr).
\]
Then $c(x)>0$, $\int_E\Tr(c(x))\,\nu(\d x)=1$, and
\[
  \log(c(x))=\log(b(x))+\lambda K(x)-\log(\mathcal Z)I_{\KK}.
\]
Substitution into the definition of $\grel$ gives, pointwise,
\[
  \grel(a\mid c)=\grel(a\mid b)-\lambda\Tr(aK)+\log(\mathcal Z)\Tr(a)+\Tr(c-b).
\]
Every term on the right has finite integral of its absolute value
under the present assumptions.
Integrating, using $\grel(a\mid c)\ge0$ and the three
normalizations, proves \eqref{eq:cq-gibbs}.
The calculation also applies to singular $a$, with $0\log(0)=0$.
\end{proof}

\begin{lemma}
\label{lem:instrument-DPI}
Let $(E,\mathcal A,\nu)$ be a standard probability space and
let $\Phi_x:\cL(\KK)\to\cL(\KK)$ be a measurable family of
completely positive maps.
Assume that
\[
  \int_E\Tr(\Phi_x(B))\,\nu(\d x)=\Tr(B),\qquad B\ge0.
\]
Then
\begin{equation}
  \int_E\grel(\Phi_x(X)\mid\Phi_x(Z))\,\nu(\d x)\le\grel(X\mid Z),\qquad X,Z\ge0.
  \label{eq:instrument-contraction}
\end{equation}
\end{lemma}

\begin{proof}
We first recall why quantum data processing also applies to
$\grel$.
If $a:=\Tr(X)>0$ and $b:=\Tr(Z)>0$, then
\[
  \grel(X\mid Z)=aD(X/a\mid Z/b)+a\log(a/b)-a+b.
\]
A completely positive trace-preserving map preserves $a$ and $b$,
while quantum data processing applies to the normalized states.
It therefore contracts $\grel$.
If $X=0$, the assertion follows from trace preservation.
If $Z=0$ and $X\ne0$, the right-hand side is infinite.

Fix $X,Z\ge0$ and put
$A(x):=\Phi_x(X)$ and $B(x):=\Phi_x(Z)$.
Positivity and the trace assumption imply
\[
  \int_E\|A(x)\|_1\,\nu(\d x)=\Tr(X),\qquad \int_E\|B(x)\|_1\,\nu(\d x)=\Tr(Z).
\]
Thus their matrix entries are integrable.
By writing an arbitrary matrix as a linear combination of four
positive matrices, the same argument defines
$\int_F\Phi_x(C)\,\nu(\d x)$ for every $C\in\cL(\KK)$
and every measurable $F\subset E$.

For a finite measurable partition $\mathcal P$ of $E$, define
\[
  \Lambda_{\mathcal P}(C):=\bigoplus_{F\in\mathcal P}\int_F\Phi_x(C)\,\nu(\d x).
\]
This map is completely positive and trace preserving.
Data processing and additivity of $\grel$ over direct sums give
\[
  \sum_{F\in\mathcal P}\grel\left(\int_F A\,\d\nu\,\middle|\,\int_F B\,\d\nu\right)\le\grel(X\mid Z).
\]
Using $\grel(cU\mid cV)=c\grel(U\mid V)$ for $c>0$, the
left-hand side equals
\[
  \int_E\grel\left(\E_\nu[A\mid\sigma(\mathcal P)]\,\middle|\,\E_\nu[B\mid\sigma(\mathcal P)]\right)\,\d\nu,
\]
where cells of zero measure are omitted.

Choose increasing finite partitions $\mathcal P_r$ whose union
generates the sigma-field of the real and imaginary parts of
the entries of $A$ and $B$.
Set
$A_r:=\E_\nu[A\mid\sigma(\mathcal P_r)]$ and
$B_r:=\E_\nu[B\mid\sigma(\mathcal P_r)]$.
Martingale convergence gives $A_r\to A$ and $B_r\to B$
almost everywhere, entry by entry and hence in matrix norm.
Lower semicontinuity, nonnegativity of $\grel$, and Fatou's lemma
therefore give
\[
  \int_E\grel(A\mid B)\,\d\nu\le\liminf_{r\to\infty}\int_E\grel(A_r\mid B_r)\,\d\nu\le\grel(X\mid Z).
\]
This proves \eqref{eq:instrument-contraction}.
\end{proof}

With natural logarithms, quantum Pinsker's inequality reads
\begin{equation}
  \|R-S\|_1^2\le2D(R\mid S),\qquad R,S\in\cS(\KK).
  \label{eq:quantum-pinsker}
\end{equation}
To verify it, let $P$ be the spectral projection onto the positive
eigenspaces of $R-S$, and put $p:=\Tr(PR)$ and $q:=\Tr(PS)$.
Since $\Tr(R-S)=0$, we have $p-q=\|R-S\|_1/2$.
Quantum data processing for the measurement
$\{P,I_{\KK}-P\}$ gives
\[
  D(R\mid S)\ge p\log(p/q)+(1-p)\log((1-p)/(1-q))\ge2(p-q)^2.
\]
For $0<q<1$, the last inequality follows by viewing the middle
expression as a function of $p$: its value and first derivative
vanish at $p=q$, and its second derivative is
$1/(p(1-p))\ge4$.
The boundary cases follow by continuity and the support convention.
This proves \eqref{eq:quantum-pinsker}.

We also use monotonicity and convexity of the trace exponential.
For self-adjoint $X,V$, put
$A_t:=X+tV$ and $\phi(t):=\Tr(\exp(A_t))$.
Differentiation gives
\[
  \phi'(t)=\Tr(V\exp(A_t)),\qquad \phi''(t)=\int_0^1\|\exp((1-r)A_t/2)V\exp(rA_t/2)\|_2^2\,\d r\ge0.
\]
Thus $\phi$ is nondecreasing when $V\ge0$, and is convex for
every self-adjoint $V$.
In particular, $X\le Z$ implies
$\Tr(\exp(X))\le\Tr(\exp(Z))$, and
$X\mapsto\Tr(\exp(X))$ is convex on the self-adjoint matrices.
The remaining standard trace inequality used in
Proposition~\ref{prop:partition} is Golden--Thompson:
\begin{equation}
  \Tr(\exp(X+Z))\le\Tr(\exp(X)\exp(Z)),\qquad X=X^*,\quad Z=Z^*.
  \label{eq:standard-trace-inequalities}
\end{equation}

\subsection{Scalar exponential estimates for product spectral laws}
\label{app:scalar}

\begin{lemma}
\label{lem:scalar}
Let $K>0$ and let $X$ be a real random variable satisfying
$\E[X]=0$ and $\E[\exp(X^2/K^2)]\le2$.
For every $N\ge1$, independent copies $X_1,\ldots,X_N$ of $X$,
and $0\le u\le(64K^2)^{-1}$,
\begin{equation}
  \E\left[\exp\left(\frac uN\left(\sum_{j=1}^N X_j\right)^2\right)\right]\le\sqrt2,\qquad \E\left[\exp\left(\frac uN\sum_{j=1}^N X_j^2\right)\right]\le2.
  \label{eq:scalar-square-bounds}
\end{equation}
\end{lemma}

\begin{proof}
Let $X'$ be an independent copy of $X$.
For every $s\in\mathbb R$, centering and conditional Jensen's
inequality give
\[
  \E[\exp(sX)]\le\E[\exp(s(X-X'))].
\]
These exponential moments are finite by the assumed bound on
$\E[\exp(X^2/K^2)]$ and Young's inequality.

For every integer $p\ge1$, the power series of the exponential
gives $\E[|X|^{2p}]\le2K^{2p}p!$.
Since
$|x-y|^{2p}\le2^{2p-1}(|x|^{2p}+|y|^{2p})$,
we obtain
\[
  \E[|X-X'|^{2p}]\le2(4K^2)^p p!.
\]
The variable $X-X'$ is symmetric.
Expanding the hyperbolic cosine and using
$(2p)!\ge(p!)^2$ therefore yields
\[
  \E[\exp(sX)]\le1+\sum_{p=1}^\infty\frac{s^{2p}\E[|X-X'|^{2p}]}{(2p)!}\le2\exp(4K^2s^2)-1\le\exp(8K^2s^2).
\]
The expansion is justified by monotone convergence.
The last inequality follows from $2e^v-1\le e^{2v}$ for $v\ge0$.

Put $S_N:=N^{-1/2}\sum_{j=1}^N X_j$.
Independence implies
\[
  \E[\exp(sS_N)]=\prod_{j=1}^N\E[\exp(sX_j/\sqrt N)]\le\exp(8K^2s^2).
\]
Let $G$ be a standard Gaussian variable independent of
$X_1,\ldots,X_N$.
For real $z$,
$\exp(uz^2)=\E[\exp(\sqrt{2u}\,Gz)]$.
Tonelli's theorem and the preceding estimate give
\[
  \E[\exp(uS_N^2)]\le\E[\exp(16uK^2G^2)]=(1-32uK^2)^{-1/2}\le\sqrt2.
\]
The last inequality uses $u\le(64K^2)^{-1}$.

For the second estimate, independence and concavity of
$x\mapsto x^{1/N}$ give
\[
  \E\left[\exp\left(\frac uN\sum_{j=1}^N X_j^2\right)\right]=\left(\E[\exp(uX^2/N)]\right)^N\le\E[\exp(uX^2)]\le2,
\]
because $u\le K^{-2}$.
\end{proof}

\section{Derivatives of relative entropy}
\label{app:technical}
\label{app:generator}

We establish the derivative bound used in the proof of
Proposition~\ref{prop:H-production}.
Let $\KK$ be a finite-dimensional Hilbert space and put
$f(X,Z):=\grel(X\mid Z)$ for $X,Z>0$.
Derivatives are taken on the real vector space of self-adjoint
matrix pairs, equipped with
$\|(U,V)\|:=\|U\|+\|V\|$, as in
Subsection~\ref{subsec:entropy-production}.

For $Z>0$, the resolvent representation is
\[
  \log(Z)=\int_0^\infty\left(\frac{1}{1+r}I_{\KK}-(Z+rI_{\KK})^{-1}\right)\,\d r.
\]
Write $R_r:=(Z+rI_{\KK})^{-1}$.
Differentiating under the integral gives
\[
  \mathrm D\log(Z)[U]=\int_0^\infty R_rUR_r\,\d r,
\]
and
\[
  \mathrm D^2\log(Z)[U,V]=-\int_0^\infty(R_rUR_rVR_r+R_rVR_rUR_r)\,\d r.
\]
These integrals converge in operator norm.
Indeed, if $\lambda:=\lambda_{\min}(Z)$, then
$\|R_r\|\le(\lambda+r)^{-1}$.
The same bounds hold locally uniformly in $Z>0$, justifying
differentiation under the integral.
They also give
\[
  \|\mathrm D\log(Z)[U]\|\le\|Z^{-1}\|\|U\|,\qquad \|\mathrm D^2\log(Z)[U,V]\|\le\|Z^{-1}\|^2\|U\|\|V\|.
\]
For the second bound, we used
$2\int_0^\infty(\lambda+r)^{-3}\,\d r=\lambda^{-2}$.

Cyclicity of the trace and the identity
$\int_0^\infty X(X+rI_{\KK})^{-2}\,\d r=I_{\KK}$
imply
$\Tr(X\,\mathrm D\log(X)[U])=\Tr(U)$.
Differentiating the definition of $f$ therefore yields
\[
  \mathrm D_Xf(X,Z)[U]=\Tr\bigl(U(\log(X)-\log(Z))\bigr),
\]
and
\[
  \mathrm D_Zf(X,Z)[V]=\Tr(V)-\Tr\bigl(X\,\mathrm D\log(Z)[V]\bigr).
\]
The components of the second derivative are
\[
  \mathrm D^2f(X,Z)[(U,0),(U',0)]=\Tr\bigl(U\,\mathrm D\log(X)[U']\bigr),
\]
\[
  \mathrm D^2f(X,Z)[(U,0),(0,V)]=-\Tr\bigl(U\,\mathrm D\log(Z)[V]\bigr),
\]
and
\[
  \mathrm D^2f(X,Z)[(0,V),(0,V')]=-\Tr\bigl(X\,\mathrm D^2\log(Z)[V,V']\bigr).
\]
The full second derivative follows by bilinearity and symmetry.

For $Z>0$, the spectral decomposition gives
$\|\log(Z)\|\le\|Z\|+\|Z^{-1}\|$.
Combining this inequality with the derivative bounds and
$|\Tr(A)|\le\dim(\KK)\|A\|$, we obtain
\begin{equation}
  |f(X,Z)|+\|\mathrm Df(X,Z)\|+\|\mathrm D^2f(X,Z)\|\le C_{\dim(\KK)}\bigl(1+\|X\|+\|Z\|+\|X^{-1}\|+\|Z^{-1}\|\bigr)^3.
  \label{eq:entropy-derivative-bound}
\end{equation}
The derivative norms are induced by the stated norm on matrix pairs.
The term of highest degree comes from
$\Tr(X\,\mathrm D^2\log(Z)[V,V'])$, whose absolute value is
at most
$\dim(\KK)\|X\|\|Z^{-1}\|^2\|V\|\|V'\|$.

\section{Permutation invariance and the initial entropy}
\label{app:initial}

We prove the initial-state estimates used in
Theorem~\ref{thm:main-chaotic}.
Throughout this appendix, $d=\dim(\HH)$ is fixed, and
$\delta_N^n$ is defined in \eqref{eq:weak-tensorization}.

\subsection{Two-particle chaoticity}

\begin{lemma}
\label{lem:two-marginals-suffice}
Let $\rho_0^N\in\cS(\HH^N)$ be permutation invariant and let
$\gamma_0\in\cS(\HH)$.
For every fixed $n\ge1$,
\begin{equation}
  \delta_N^1\le\delta_N^2,\qquad \delta_N^n\le C_{d,n}\left(\delta_N^2+\frac1N\right)^{1/2},\qquad N\ge\max\{2,n\}.
  \label{eq:two-to-all-marginals}
\end{equation}
Consequently, $\delta_N^2\to0$ is equivalent to
$\delta_N^n\to0$ for every fixed $n$.
No consistency between the states for different values of $N$
is required.
\end{lemma}

\begin{proof}
Partial-trace contraction gives $\delta_N^1\le\delta_N^2$.
Put $r_N:=\delta_N^2+N^{-1}$.
For a self-adjoint $A\in\cL(\HH)$ with $\|A\|\le1$, set
$a:=\Tr(\gamma_0A)$, $B:=A-aI_{\HH}$, and
$\overline A_N:=N^{-1}\sum_{j=1}^N A_j$.
Then $\|B\|\le2$ and $\Tr(\gamma_0B)=0$.
Permutation invariance gives
\[
  \Tr\bigl(\rho_0^N(\overline A_N-aI_{\HH^N})^2\bigr)=\frac1N\Tr(\rho_0^{N:1}B^2)+\frac{N-1}{N}\Tr(\rho_0^{N:2}(B\otimes B))\le4r_N.
\]
For the inequality, the second trace is unchanged if
$\rho_0^{N:2}$ is replaced by
$\rho_0^{N:2}-\gamma_0^{\otimes2}$, and
$\|B\otimes B\|\le4$.

Now let $A_1,\ldots,A_n$ be self-adjoint matrices with
$\|A_\ell\|\le1$, and put
$a_\ell:=\Tr(\gamma_0A_\ell)$ and
$\overline A_{\ell,N}:=N^{-1}\sum_{j=1}^N(A_\ell)_j$.
The ordered telescoping identity is
\[
  \prod_{\ell=1}^n\overline A_{\ell,N}-\left(\prod_{\ell=1}^n a_\ell\right)I_{\HH^N}=\sum_{\ell=1}^n\left(\prod_{i<\ell}\overline A_{i,N}\right)(\overline A_{\ell,N}-a_\ell I_{\HH^N})\prod_{i>\ell}a_i.
\]
For the positive functional
$\varphi(C):=\Tr(\rho_0^N C)$, Cauchy--Schwarz gives
$|\varphi(CD)|^2\le\varphi(CC^*)\varphi(D^*D)$.
Each product $\prod_{i<\ell}\overline A_{i,N}$ has norm at most one.
Applying this inequality to each term of the telescoping identity
and using the preceding variance estimate yields
\[
  \left|\Tr\left(\rho_0^N\prod_{\ell=1}^n\overline A_{\ell,N}\right)-\prod_{\ell=1}^n a_\ell\right|\le2n r_N^{1/2}.
\]
No commutation of the averaged matrices is needed in this argument.

In the expansion of $\prod_{\ell=1}^n\overline A_{\ell,N}$,
every term with distinct particle labels has expectation
$\Tr(\rho_0^{N:n}(A_1\otimes\cdots\otimes A_n))$.
The fraction of terms with a repeated label is at most
$n(n-1)/(2N)$, by the union bound over pairs of positions.
Every term has norm at most one.
It follows that
\[
  \left|\Tr\bigl((\rho_0^{N:n}-\gamma_0^{\otimes n})(A_1\otimes\cdots\otimes A_n)\bigr)\right|\le2n r_N^{1/2}+\frac{n(n-1)}N.
\]

Choose a self-adjoint Hilbert--Schmidt orthonormal basis
$(F_a)_{a=1}^{d^2}$.
Each $F_a$ has operator norm at most one, and the tensor products
of these matrices form a Hilbert--Schmidt orthonormal basis on
$\HH^n$.
Writing $\Delta:=\rho_0^{N:n}-\gamma_0^{\otimes n}$, Parseval's
identity gives
\[
  \|\Delta\|_2^2=\sum_{a_1,\ldots,a_n=1}^{d^2}\left|\Tr\bigl(\Delta(F_{a_1}\otimes\cdots\otimes F_{a_n})\bigr)\right|^2.
\]
The preceding estimate therefore implies
\[
  \|\Delta\|_2\le d^n\left(2n r_N^{1/2}+\frac{n(n-1)}N\right).
\]
Since $\|\Delta\|_1\le d^{n/2}\|\Delta\|_2$ and
$N^{-1}\le r_N^{1/2}$, this proves
\eqref{eq:two-to-all-marginals}.
\end{proof}

\subsection{Entropy per particle: a quantitative bound}

\begin{lemma}
\label{lem:initial-entropy}
Let $\rho_0^N\in\cS(\HH^N)$ be permutation invariant and let
$\gamma_0\in\cS(\HH)$ be arbitrary.
There exists $C_d<\infty$, depending only on $d$, such that
\begin{equation}
  \left|\frac1N S(\rho_0^N)-S(\gamma_0)\right|\le C_d\left(\delta_N^2+\frac1N\right)^{1/4},\qquad N\ge2.
  \label{eq:initial-entropy-modulus}
\end{equation}
\end{lemma}

\begin{proof}
We use the Schur--Weyl decomposition and its highest-weight
description; see \cite{KeylWerner2001}.
Let
\[
  \mathcal P_{N,d}:=\left\{\lambda\in\mathbb Z_{\ge0}^d:\lambda_1\ge\cdots\ge\lambda_d,\ \sum_{i=1}^d\lambda_i=N\right\}.
\]
Under the Schur--Weyl identification,
\[
  \HH^N=\bigoplus_{\lambda\in\mathcal P_{N,d}}\mathcal U_\lambda\otimes\mathcal V_\lambda,
\]
where $\mathcal U_\lambda$ carries the irreducible representation
of $GL(d,\mathbb C)$ of highest weight $\lambda$, and
$\mathcal V_\lambda$ carries the corresponding irreducible
representation of the permutation group.
Put $f_\lambda:=\dim(\mathcal V_\lambda)$.
Since $\rho_0^N$ commutes with the permutation action,
\begin{equation}
  \rho_0^N=\bigoplus_{\lambda\in\mathcal P_{N,d}}p_N(\lambda)\,\tau_\lambda\otimes\frac{I_{\mathcal V_\lambda}}{f_\lambda}.
  \label{eq:Schur-state}
\end{equation}
Here $p_N(\lambda):=\Tr(P_\lambda\rho_0^N)$, where $P_\lambda$
is the orthogonal projection onto the $\lambda$-summand.
The numbers $p_N(\lambda)$ are nonnegative and sum to one, and
$\tau_\lambda\in\cS(\mathcal U_\lambda)$ whenever
$p_N(\lambda)>0$.
The value of $\tau_\lambda$ is immaterial when $p_N(\lambda)=0$.

We use the notation $H(p)=-\sum_i p_i\log(p_i)$ for any finite
probability vector $p$, with $0\log(0):=0$.
The entropy of the direct sum in \eqref{eq:Schur-state} is
\[
  S(\rho_0^N)=H(p_N)+\sum_\lambda p_N(\lambda)S(\tau_\lambda)+\sum_\lambda p_N(\lambda)\log(f_\lambda).
\]
There are at most $(N+1)^d$ elements of $\mathcal P_{N,d}$.
Weyl's dimension formula gives
\[
  \dim(\mathcal U_\lambda)=\prod_{1\le i<j\le d}\frac{\lambda_i-\lambda_j+j-i}{j-i}\le(N+d)^{d(d-1)/2}.
\]
Consequently,
\[
  0\le H(p_N)+\sum_\lambda p_N(\lambda)S(\tau_\lambda)\le C_d\log(N+1).
\]
The dimension of $\mathcal V_\lambda$ is
\[
  f_\lambda=N!\,\frac{\prod_{1\le i<j\le d}(\lambda_i-\lambda_j+j-i)}{\prod_{i=1}^d(\lambda_i+d-i)!}.
\]
In the ratio of $f_\lambda$ to $N!/\prod_i\lambda_i!$,
the numerator and denominator each contain $d(d-1)/2$ factors,
all between $1$ and $N+d$.
Thus
\[
  \left|\log(f_\lambda)-\log\left(\frac{N!}{\prod_{i=1}^d\lambda_i!}\right)\right|\le C_d\log(N+1).
\]
The Stirling bound
$|\log(m!)-m\log(m)+m|\le C\log(m+1)$, valid also at $m=0$
with the stated convention, then yields
\[
  |\log(f_\lambda)-NH(\lambda/N)|\le C_d\log(N+1).
\]
Combining these estimates gives
\begin{equation}
  \left|\frac1N S(\rho_0^N)-\sum_\lambda p_N(\lambda)H(\lambda/N)\right|\le C_d\frac{\log(N+1)}N.
  \label{eq:entropy-shape}
\end{equation}
We next estimate the distance between $\lambda/N$ and the
decreasing eigenvalue vector $r=(r_1,\ldots,r_d)$ of $\gamma_0$.
Let $(F_a)_{a=1}^{d^2}$ be a self-adjoint Hilbert--Schmidt
orthonormal basis, put $g_a:=\Tr(\gamma_0F_a)$, and define
\[
  J_N:=\sum_{a=1}^{d^2}\left(\frac1N\sum_{j=1}^N F_{a,j}-g_a I_{\HH^N}\right)^2.
\]
For the swap operator $\mathbb S$ on $\HH\otimes\HH$,
completeness of the basis gives
$\sum_aF_a\otimes F_a=\mathbb S$,
$\sum_aF_a^2=dI_{\HH}$,
$\sum_ag_aF_a=\gamma_0$, and
$\sum_ag_a^2=\Tr(\gamma_0^2)$.
Expanding the squares therefore gives
\begin{equation}
  J_N=\frac dN I_{\HH^N}+\frac2{N^2}\sum_{i<j}\mathbb S_{ij}-\frac2N\sum_{j=1}^N(\gamma_0)_j+\Tr(\gamma_0^2)I_{\HH^N}.
  \label{eq:J-expansion}
\end{equation}
Permutation invariance implies
\[
  \Tr(\rho_0^N J_N)=\frac dN+\frac{N-1}{N}\Tr(\rho_0^{N:2}\mathbb S)-2\Tr(\rho_0^{N:1}\gamma_0)+\Tr(\gamma_0^2).
\]
Using
$\Tr(\gamma_0^{\otimes2}\mathbb S)=\Tr(\gamma_0^2)$,
$\|\mathbb S\|=1$, and $\|\gamma_0\|\le1$, we obtain
\begin{equation}
  \Tr(\rho_0^N J_N)\le\frac dN+\delta_N^2+2\delta_N^1.
  \label{eq:J-expectation}
\end{equation}
We now bound $J_N$ on each summand in \eqref{eq:Schur-state}.
Let $E_{ab}$ be the matrix units on $\HH$ and put
$E_{ab}^{(N)}:=\sum_{j=1}^N(E_{ab})_j$.
Their commutation relations are
\[
  [E_{ab}^{(N)},E_{cd}^{(N)}]=\delta_{bc}E_{ad}^{(N)}-\delta_{ad}E_{cb}^{(N)}.
\]
These relations show that
$\sum_{a,b}E_{ab}^{(N)}E_{ba}^{(N)}$ commutes with every
$E_{cd}^{(N)}$.
It therefore acts as a scalar on each $\mathcal U_\lambda$.
Direct expansion gives
\[
  \sum_{a,b=1}^dE_{ab}^{(N)}E_{ba}^{(N)}=dN I_{\HH^N}+2\sum_{i<j}\mathbb S_{ij}.
\]
On $\mathcal U_\lambda$, a highest-weight vector $v_\lambda$
satisfies
$E_{ii}^{(N)}v_\lambda=\lambda_i v_\lambda$ and
$E_{ij}^{(N)}v_\lambda=0$ for $i<j$.
Consequently,
\[
  \left(\sum_{a,b=1}^dE_{ab}^{(N)}E_{ba}^{(N)}\right)v_\lambda=\left(\sum_{i=1}^d\lambda_i^2+\sum_{i<j}(\lambda_i-\lambda_j)\right)v_\lambda.
\]
Since $\sum_i\lambda_i=N$, it follows that
$2\sum_{i<j}\mathbb S_{ij}$ acts on the $\lambda$-summand as
$\sum_i\lambda_i(\lambda_i-2i+1)I$.

Diagonalize $\gamma_0$ with eigenvalues
$r_1\ge\cdots\ge r_d$.
Then
$\sum_j(\gamma_0)_j=\sum_i r_iE_{ii}^{(N)}$.
Its eigenvalue on $v_\lambda$ is $\sum_i\lambda_i r_i$.
The space $\mathcal U_\lambda$ is spanned by vectors obtained
by applying products of lowering operators
$E_{ji}^{(N)}$, $i<j$, to $v_\lambda$.
The commutation relation
\[
  \left[\sum_{a=1}^d r_aE_{aa}^{(N)},E_{ji}^{(N)}\right]=(r_j-r_i)E_{ji}^{(N)}
\]
shows that each such operation changes this eigenvalue by
$r_j-r_i\le0$.
Thus the largest eigenvalue of $\sum_j(\gamma_0)_j$ on the
$\lambda$-summand is $\sum_i\lambda_i r_i$.

Substituting these two facts into \eqref{eq:J-expansion} gives
\[
  J_N\big|_\lambda\ge\left(\left|\frac{\lambda}{N}-r\right|^2+\frac{d-\sum_{i=1}^d(2i-1)\lambda_i/N}{N}\right)I.
\]
Since $\sum_i\lambda_i/N=1$ and $2i-1\le2d-1$,
\[
  J_N\big|_\lambda\ge\left(\left|\frac{\lambda}{N}-r\right|^2-\frac{d-1}{N}\right)I.
\]
Taking the trace against \eqref{eq:Schur-state} and applying
\eqref{eq:J-expectation}, we obtain
\begin{equation}
  \sum_\lambda p_N(\lambda)\left|\frac{\lambda}{N}-r\right|^2\le\delta_N^2+2\delta_N^1+\frac{2d-1}{N}.
  \label{eq:shape-concentration}
\end{equation}

Finally, the function $f(x):=-x\log(x)$, extended by $f(0)=0$,
is $1/2$-H\"older on $[0,1]$.
Indeed, for $0\le x<y\le1$ and $h:=y-x$,
\[
  |f(y)-f(x)|\le\int_x^y(1-\log(u))\,\d u\le h(2-\log(h))\le2h^{1/2}.
\]
It follows that
$|H(p)-H(r)|\le C_d|p-r|^{1/2}$
for probability vectors $p,r$, where $|\cdot|$ denotes the
Euclidean norm.
Jensen's inequality and \eqref{eq:shape-concentration} imply
\[
  \sum_\lambda p_N(\lambda)|H(\lambda/N)-H(r)|\le C_d\left(\delta_N^2+2\delta_N^1+\frac{2d-1}{N}\right)^{1/4}.
\]
Since $H(r)=S(\gamma_0)$, \eqref{eq:entropy-shape} gives
\[
  \left|\frac1N S(\rho_0^N)-S(\gamma_0)\right|\le C_d\left(\delta_N^2+2\delta_N^1+\frac{2d-1}{N}\right)^{1/4}+C_d\frac{\log(N+1)}N.
\]
Using $\delta_N^1\le\delta_N^2$ and
$\log(N+1)/N\le C N^{-1/4}$ proves
\eqref{eq:initial-entropy-modulus}.
\end{proof}

When $\gamma_0>0$, Lemma~\ref{lem:faithful-initial-conversion}
gives the corresponding relative-entropy estimate, using
\eqref{eq:entropy-shape} and \eqref{eq:shape-concentration}.

\end{document}